\documentclass[10pt]{article}

\usepackage{latexsym, amssymb, amsmath, amsthm, amscd}
\usepackage[all,cmtip]{xy}
\usepackage{amsmath}
\usepackage{amsthm}
\usepackage{amssymb}
\usepackage{amsfonts}
\usepackage{amsrefs}
\usepackage{color,authblk}
\usepackage{tikz}
\usepackage{amscd} 
\usepackage{comment}
\usepackage{mathrsfs}
\usepackage{comment}
\usepackage[all]{xy}
\usepackage{graphicx}

\usepackage{authblk}
\usepackage{hyperref}
\usepackage{bbm}
\usepackage{fullpage}
\usepackage[all,cmtip]{xy}
\usepackage{relsize}
\usepackage{amsmath,amscd}
\usepackage{tikz-cd}
\usepackage{tikz}
\usetikzlibrary{matrix}
\usepackage[mathcal]{euscript}
\usepackage{txfonts}

\usepackage{hyperref}

\numberwithin{equation}{section} \DeclareMathSizes{2}{10}{12}{13}
\makeatletter
\newcommand*{\doublerightarrow}[2]{\mathrel{
		\settowidth{\@tempdima}{$\scriptstyle#1$}
		\settowidth{\@tempdimb}{$\scriptstyle#2$}
		\ifdim\@tempdimb>\@tempdima \@tempdima=\@tempdimb\fi
		\mathop{\vcenter{
				\offinterlineskip\ialign{\hbox to\dimexpr\@tempdima+1em{##}\cr
					\rightarrowfill\cr\noalign{\kern.5ex}
					\rightarrowfill\cr}}}\limits^{\!#1}_{\!#2}}}
\newcommand*{\triplerightarrow}[1]{\mathrel{
		\settowidth{\@tempdima}{$\scriptstyle#1$}
		\mathop{\vcenter{
				\offinterlineskip\ialign{\hbox to\dimexpr\@tempdima+1em{##}\cr
					\rightarrowfill\cr\noalign{\kern.5ex}
					\rightarrowfill\cr\noalign{\kern.5ex}
					\rightarrowfill\cr}}}\limits^{\!#1}}}
\makeatother

\newtheorem{thm}{Proposition}[section]
\newtheorem{Thm}[thm]{Theorem}

\newtheorem{cor}[thm]{Corollary}
\newtheorem{eg}[thm]{Examples}
\newtheorem{lem}[thm]{Lemma}
\newtheorem{defn}[thm]{Definition}

\title{Eilenberg-Moore categories and quiver representations of monads and comonads}

\author{Divya Ahuja\footnote{Department of Mathematics, Indian Institute of Technology, Delhi. Email: divyaahuja1428@gmail.com.} $\qquad$ Abhishek Banerjee\footnote{Department of Mathematics, Indian Institute of Science, Bangalore. Email: abhishekbanerjee1313@gmail.com.}$\qquad$ Surjeet Kour\footnote{Department of Mathematics, Indian Institute of Technology, Delhi. Email: koursurjeet@gmail.com.} $\qquad $  Samarpita Ray\footnote{Stat-Math Unit, Indian Statistical Institute, Bangalore. Email: ray.samarpita31@gmail.com.}}

\date{}

\begin{document}
	
	\maketitle 

\begin{abstract}
We consider representations of quivers taking values in monads or comonads over a Grothendieck category
		$\mathcal C$.  We treat these as scheme like objects whose ``structure sheaf'' consists of monads or comonads. By using systems of adjoint functors between Eilenberg-Moore categories, we obtain a categorical framework of modules over monad quivers, and of comodules over comonad quivers. Our main objective is to give conditions for these to be Grothendieck categories, which play the role of noncommutative spaces. As with usual ringed spaces, we have to study two kinds of module categories
over a monad quiver. The first behaves like a sheaf of modules over a ringed space. The second consists of modules that are cartesian, which resemble quasi-coherent sheaves. We also obtain an extension of the classical quasi-coherator construction to modules over a monad quiver with values in Eilenberg-Moore categories. We establish similar results for comodules over a comonad quiver. One of our key steps is finding a modulus like bound for an endofunctor $U:\mathcal C\longrightarrow \mathcal C$ in terms of $\kappa(G)$, where $G$ is a generator for $\mathcal C$ and $\kappa(G)$ is a cardinal such that $G$ is $\kappa(G)$-presentable.  Another feature of our paper is that we study modules over a monad quiver in two different orientations, which we refer to as ``cis-modules'' and ``trans-modules.'' We conclude with rational pairings of a monad quiver with a comonad quiver, which relate comodules over a comonad quiver to coreflective subcategories of modules over monad quivers. 
\end{abstract}
	
	\medskip
	{MSC(2020) Subject Classification: 18C20, 18E10}

	
	\medskip
\hypersetup{linktocpage}
	\tableofcontents
	\section{Introduction}
	
	Let $Z$ be a scheme. Then, a famous result of Gabber (see, for instance, \cite[Tag 077P]{Stacks}) shows that the category $QCoh(Z)$ of quasi-coherent sheaves over $Z$ is a Grothendieck category. If $S$ is a scheme and $\mathcal Z$ is an algebraic stack over $S$, the category $QCoh(\mathcal Z)$ of quasi-coherent sheaves over $\mathcal Z$ is also a Grothendieck category (see, for instance, \cite[Tag 06WU]{Stacks}). We can ask similar questions in much more general contexts. For example, let $(\mathcal D,\otimes)$ be a monoidal category having an action 
	$\_\_\otimes \_\_:\mathcal D\times \mathcal L\longrightarrow \mathcal L$ on a Grothendieck category $\mathcal L$. Then if  $Alg(\mathcal D)$ denotes the category of monoid objects
	in $\mathcal D$, we may consider for any $A\in Alg(\mathcal D)$ the category $A-Mod^{\mathcal L}$ of ``left $A$-module objects in $\mathcal L$.'' Then, a morphism in $Alg(\mathcal D)$ induces an adjoint pair of functors between the corresponding module categories with objects in $\mathcal L$. Accordingly, one may set up a theory of quasi-coherent modules, with coefficients in $\mathcal L$, for a representation $\mathcal Y\longrightarrow Alg(\mathcal D)$ of a small category $\mathcal Y$. We can ask for conditions for these quasi-coherent modules to form a Grothendieck category. For instance, if $k$ is a field and $\mathcal L$ is a $k$-linear Grothendieck category, we may consider $R$-module objects in $\mathcal L$ for any $k$-algebra $R$. The latter categories play a key role in the study of noncommutative projective schemes by Artin and Zhang \cite{AZ1}, \cite{AZ2}. 
	
	\smallskip
	In this paper, we develop a categorical framework for studying module representations and comodule representations taking values in Eilenberg-Moore categories of monads and comonads. For this, we generalize the usual setup of sheaves in several different ways. First, we replace the system of affine open subsets of a scheme by a quiver $\mathbb Q=(\mathbb V,\mathbb E)$, i.e., a directed graph $\mathbb Q$ with a set of vertices $\mathbb V$ and a set of edges $\mathbb E$. Our motivation comes from several sources. The first  is that of Estrada and Virili \cite{EV}, who studied modules over a representation $\mathcal A:\mathcal X\longrightarrow Add$ of a small category $\mathcal X$ taking values in small preadditive categories. Another is the classical work of Gerstenhaber and Schack \cite{GS0}, \cite{GS1}, \cite{GS2}, which studied modules and cohomology theory for a diagram of algebras, i.e., a functor from a small category taking values in ordinary algebras.  Similarly, the modules and cohomology theory over prestacks, which are certain kinds of functors from a small category taking values in linear categories, were studied by Lowen and Van den Bergh \cite{LVb}. We also note that modules over diagrams of Lie algebras, i.e., functors from a small category taking values in Lie algebras, were studied further by Gerstenhaber, Giaquinto and Schack in \cite{GGS}. We replace algebras by monads over a given Grothendieck category $\mathcal C$. As such, we consider a representation $\mathscr U:\mathbb Q\longrightarrow Mnd(\mathcal C)$ of the quiver $\mathbb Q$ taking values in the category $Mnd(\mathcal C)$ of monads over $\mathcal C$.  Finally, we replace the usual module categories over rings by Eilenberg-Moore categories of the monads over $\mathcal C$.  Similarly, we can take ``comodule sheaves'' over the quiver $\mathbb Q=(\mathbb V,\mathbb E)$ by starting with a representation 
$\mathscr V:\mathbb Q\longrightarrow Cmd(\mathcal C)$ taking values in the category $Cmd(\mathcal C)$ of comonads over $\mathcal C$. Then, Eilenberg-Moore categories of comonads play the role of comodules over coalgebras.

\smallskip
Our focus is on generators in these representation categories taking values in Eilenberg-Moore categories of monads and comonads. These representation categories of quivers play the role of a scheme like object whose ``structure sheaf'' consists of monads or comonads.  We study conditions for these representation categories to become Grothendieck categories. As mentioned in \cite{WLS}, a Grothendieck category often plays the role of a noncommutative space. We also know from a classical result of Gabriel \cite{Gab} that a scheme may be recovered from the Grothendieck category of its quasi-coherent sheaves. 
Hence, the representation categories that we consider in this paper allow us to develop the theory of ``monadic spaces'' and ``comonadic spaces'' (in place of 
the usual ringed spaces) with the flavor of algebraic geometry. We should also mention here that module valued representations of a small category have been studied at several places in the literature (see, for instance, \cite{f13}, \cite{EE}, \cite{f16}, \cite{f17}).

\smallskip We refer to a representation $\mathscr U:\mathbb Q\longrightarrow Mnd(\mathcal C)$ as a monad quiver. To study modules over $\mathscr U$, we combine techniques on monads and adapt our methods from earlier work in \cite{Ban}, \cite{BBR} which are inspired by the cardinality arguments of Estrada and Virili \cite{EV} (see also Enochs and Estrada \cite{EE}, \cite{f16}). In this respect, one of our key tools in this paper is a modulus like bound that we obtain on an endofunctor. More precisely, we   fix a generator $G$ for the Grothendieck category $\mathcal C$. Then, for any object
	$M\in \mathcal C$, the set $el_G(M):=\mathcal C(G,M)$ plays the role of elements of $M$ and we put  $||M||^G:=|\mathcal C(G,M)|$. We choose $\kappa(G)$ such that $G\in \mathcal C$ is 
	$\kappa(G)$-presentable. Our first main step is to give a modulus like bound on an endofunctor $U:\mathcal C\longrightarrow \mathcal C$, i.e., a cardinal $\lambda^U$ (which depends on the generator $G$) such that for any object $M\in \mathcal C$ we have
	\begin{equation}
		||UM||^G\leq \lambda^U\times  (||M||^G)^{\kappa(G)}
	\end{equation} If $\phi:U\longrightarrow U'$ is a morphism of monads over $\mathcal C$, under certain conditions, there is a system of functors (see Section 4)
	\begin{equation}
		\label{1.2fd} \phi^*:EM_{U}\longrightarrow EM_{U'}\qquad \phi_*:EM_{U'}\longrightarrow EM_U\qquad \hat\phi:EM_U\longrightarrow EM_{U'}
	\end{equation} between  Eilenberg-Moore categories $EM_U$ and $EM_{U'}$ of $U$ and $U'$ respectively. Here, $(\phi^*,\phi_*)$ and 
$(\phi_*,\hat\phi)$ are pairs of adjoint functors. For a monad quiver  $\mathscr U:\mathbb Q\longrightarrow Mnd(\mathcal C)$, we use the system of functors in 
\eqref{1.2fd} to define $\mathscr U$-modules of two different orientations. By adapting the terminology from \cite{BBR}, we refer to these as ``cis-modules'' and ``trans-modules.'' More precisely, a cis-module $\mathscr M$ over $\mathscr U$ consists of a family of objects $\{\mathscr M_x\in EM_{\mathscr U_x}\}_{x\in Ob(\mathbb Q)}$ along with compatible morphisms $\mathscr M^\alpha:\mathscr U(\alpha)^*\mathscr M_x\longrightarrow \mathscr M_y$ (equivalently, $\mathscr M_\alpha:
	\mathscr M_x\longrightarrow \mathscr U(\alpha)_*\mathscr M_y$) for each edge $\alpha\in \mathbb Q(x,y)$ (see Definition \ref{D4.2}).  A trans-module $\mathscr M$ over $\mathscr U$ consists of  objects $\{\mathscr M_x\in EM_{\mathscr U_x}\}_{x\in Ob(\mathbb Q)}$ along with compatible morphisms ${^\alpha}\mathscr M:\mathscr U(\alpha)_*\mathscr M_y\longrightarrow \mathscr M_x$ (equivalently, ${_\alpha}\mathscr M:
	\mathscr M_y\longrightarrow \hat{\mathscr U(\alpha)}\mathscr M_x$) for each edge $\alpha\in \mathbb Q(x,y)$ (see Definition \ref{D4.10}).  We give conditions for the categories $Mod^{cs}-\mathscr U$ and $Mod^{tr}-\mathscr U$ of cis-modules and trans-modules respectively to be  Grothendieck categories, and also conditions for  them to have projective generators.
	
	\smallskip The categories $Mod^{cs}-\mathscr U$ and $Mod^{tr}-\mathscr U$ resemble categories of sheaves of modules over a ringed space. As with a ringed space, we also consider modules over the representation $\mathscr U:\mathbb Q\longrightarrow Mnd(\mathcal C)$ which behave like quasi-coherent sheaves.  Accordingly, we consider the full subcategory $Mod^{cs}_c-\mathscr U$  of $Mod^{cs}-\mathscr U$ consisting of objects which are cartesian, i.e., cis-modules $\mathscr M$ for which the morphisms  $\mathscr M^\alpha:\mathscr U(\alpha)^*\mathscr M_x\longrightarrow \mathscr M_y$  are isomorphisms for each edge $\alpha\in  \mathbb Q(x,y)$.  We give conditions for 
	$Mod^{cs}_c-\mathscr U$ to be a Grothendieck category. In that case, the canonical inclusion $Mod^{cs}_c-\mathscr U\hookrightarrow Mod^{cs}-\mathscr U$ has a right adjoint. As such, we  have a generalization of the classical quasi-coherator construction (see \cite[Lemme 3.2]{Ill}) to modules over a monad quiver with values in Eilenberg-Moore categories. We prove similar results for the full subcategory $Mod^{tr}_c-\mathscr U$  of cartesian objects in $Mod^{tr}-\mathscr U$, i.e., trans-modules  $\mathscr M$ for which  ${_\alpha}\mathscr M:
	\mathscr M_y\longrightarrow \hat{\mathscr U(\alpha)}\mathscr M_x$  is an isomorphism for each edge $\alpha\in  \mathbb Q(x,y)$. 

\smallskip
We then consider a comonad quiver $\mathscr V:\mathbb Q\longrightarrow Cmd(\mathcal C)$. If $\phi:V\longrightarrow V'$ is a morphism of comonads, there is a pair $(\phi^\circ,\phi_\circ)$ of adjoint functors
\begin{equation}\label{1.3fdt}
\phi^\circ: EM^V\longrightarrow EM^{V'} \qquad \phi_\circ: EM^{V'}\longrightarrow EM^V 
\end{equation} between  Eilenberg-Moore categories $EM^V$ and $EM^{V'}$ of $V$ and $V'$ respectively. We study comodules over the comonad quiver $\mathscr V$. We mention that unlike the system of functors \eqref{1.2fd} in the case of monads, we do not know of good conditions for the functor 
$\phi_\circ: EM^{V'}\longrightarrow EM^V $ appearing in \eqref{1.3fdt} to have a right adjoint. Accordingly, we only work with cis-comodules over the comonad quiver $\mathscr V$. 

\smallskip More precisely, a cis-comodule $\mathscr M$ over $\mathscr V$ consists of a family of objects $\{\mathscr M_x\in EM^{\mathscr V_x}\}_{x\in Ob(\mathbb Q)}$ along with compatible morphisms $\mathscr M^\alpha:\mathscr V(\alpha)^\circ\mathscr M_x\longrightarrow \mathscr M_y$ (equivalently, $\mathscr M_\alpha:
	\mathscr M_x\longrightarrow \mathscr V(\alpha)_\circ\mathscr M_y$) for each edge $\alpha\in \mathbb Q(x,y)$ (see Definition \ref{D7.3}).  We give conditions for the category $Com^{cs}-\mathscr V$ of cis-comodules to be a Grothendieck category and for it to have projective generators. 
 We mention here an important distinction between Eilenberg-Moore categories of monads and those of comonads, because of which we have to modify our cardinality arguments for treating 
cis-modules in order to apply them to the category $Com^{cs}-\mathscr V$ of cis-comodules. Suppose that $G$ is a generator for $\mathcal C$ that is $\kappa(G)$-presentable. Then if $U$ is a monad on $\mathcal C$ satisfying certain conditions, the object $UG$ becomes a generator for the Eilenberg-Moore category $EM_U$ that is also 
$\kappa(G)$-presentable (see Lemma \ref{L3.2}). However,  
if $V$ is a comonad, $VG$ need not be a generator for the Eilenberg-Moore category $EM^V$.

\smallskip
We will say that a cis-comodule $\mathscr M$ over $\mathscr V$ is cartesian if the morphism $\mathscr M_\alpha:
	\mathscr M_x\longrightarrow \mathscr V(\alpha)_\circ\mathscr M_y$ is an isomorphism for each edge $\alpha\in \mathbb Q(x,y)$. We give conditions 
for the full subcategory $Com^{cs}_c-\mathscr V$ of cartesian cis-comodules to be a Grothendieck category. We also obtain an analogue for the classical quasi-coherator construction in this context, i.e., a right adjoint to the inclusion functor $Com^{cs}_c-\mathscr V\hookrightarrow Com^{cs}-\mathscr V$.

\smallskip
In the final section of this paper, we relate trans-modules over a monad quiver $\mathscr U:\mathbb Q^{op}\longrightarrow Mnd(\mathcal C)$ to cis-comodules over a comonad quiver $\mathscr V:\mathbb Q\longrightarrow Cmd(\mathcal C)$. We do this by defining a rational pairing between monad quivers and comonad quivers (see Definition \ref{D10.2}). For this, we use the theory of rational pairings between adjoint functors developed by Mesablishvili and Wisbauer \cite{MW}. We recall the well known fact that if $C\otimes A\longrightarrow k$ is a rational pairing between a $k$-coalgebra $C$ and a $k$-algebra $A$ (where $k$ is a field), then $C$-comodules correspond to rational modules over $A$ (see, for instance, \cite[$\S$ 4.18]{BWb}). We have considered rational pairings between 
quivers of algebras and quivers of coalgebras in \cite{BBR}. If we have a rational pairing $\mathcal P$ between a monad quiver $\mathscr U$ and a comonad quiver $\mathscr V$, we define the full subcategory $Rat_{\mathcal P}^{tr}-\mathscr U \hookrightarrow Mod^{tr}-\mathscr U$ of rational trans-modules such that the inclusion has a right adjoint
\begin{equation}
\iota_{\mathcal P}^{tr} :Rat_{\mathcal P}^{tr}-\mathscr U \longrightarrow Mod^{tr}-\mathscr U\qquad Rat_{\mathcal P}^{tr}: Mod^{tr}-\mathscr U\longrightarrow Rat_{\mathcal P}^{tr}-\mathscr U
\end{equation} In that case, we show that $Com^{cs}-\mathscr V\cong Rat_{\mathcal P}^{tr}-\mathscr U$. We conclude by giving conditions for the subcategory $Rat_{\mathcal P}^{tr}-\mathscr U$ of rational trans-modules to form a hereditary torsion class in $Mod^{tr}-\mathscr U$. 
	
\smallskip
{\bf Acknowledgements:} This paper will form part of the PhD thesis of the first named author. Divya Ahuja was supported by CSIR PhD fellowship 09/086(1430)/2019-EMR-I. Authors Abhishek Banerjee and Surjeet Kour were partially supported by SERB Core Research Grant 2023/004143.  Samarpita Ray was supported by 
Faculty Fellowship DST/INSPIRE/04/2021/002904.

	\section{Generators and the bound on an endofunctor}

	Throughout this section and the rest of this paper, we assume that $\mathcal C$ is a Grothendieck category. We begin by recalling the following standard definition. 
	
	\begin{defn}\label{D2.1} (see \cite[$\S$ 1.13]{AR}). Let $\kappa$ be a regular cardinal. A partially ordered set $J$ is said to be $\kappa$-directed if every subset of $J$ having cardinality $<\kappa$ has an upper bound in $J$.
		An object $M\in \mathcal C$ is said to be $\kappa$-presentable if the functor $\mathcal C(M,\_\_)$ preserves $\kappa$-directed colimits.
	\end{defn}
	
	From Definition \ref{D2.1} it is clear that if $J$ is partially ordered set that is $\kappa$-directed, then it is also $\kappa'$-directed for any regular cardinal
	$\kappa'\leq \kappa$.  Accordingly,  if an object $M\in  \mathcal C$ is $\kappa$-presentable, then $M$ is also $\kappa''$-presentable for any regular cardinal
	$\kappa''\geq \kappa$.

	\smallskip We now fix a generator $G$ for $\mathcal C$. Because $\mathcal C$ is a  Grothendieck category, it is also locally presentable (see, for instance, \cite[Proposition 3.10]{Beke}) and it follows in particular that for each object $M\in  \mathcal C$ we can choose $\kappa(M)$ such that $M$ is $\kappa(M)$-presentable.  We choose therefore $\kappa(G)$ such that $G$ is
	$\kappa(G)$-presentable. By the above reasoning, we may suppose that $\kappa(G)$ is infinite.
	
	\smallskip
	For each $M\in \mathcal C$, we now define
	\begin{equation}\label{cardM}
		el_G(M):=\mathcal C(G,M)\qquad ||M||^{G}:=|\mathcal C(G,M)|
	\end{equation} From \eqref{cardM}, it is immediately clear that if $M'\hookrightarrow M$ is a monomorphism in $\mathcal C$, then 
	$||M'||^{ G}\leq ||M||^{ G}$. For the rest of this paper, we will assume that the generator $G$  is such that
	for any epimorphism $M\twoheadrightarrow M''$ in $\mathcal C$, we must have $||M''||^{ G}\leq ||M||^{ G}$. This would happen, for instance, if  $G$ were projective.
	
	\smallskip
	For a set $S$ and a regular cardinal $\alpha$, we denote by $\mathcal P_\alpha(S)$ the collection of subsets of $S$ having cardinality $<\alpha$. Since $\alpha$ is regular, we note that $\mathcal P_\alpha(S)$ is $\alpha$-directed. If $\{M_s\}_{s\in S}$ is a collection of objects of $\mathcal C$ indexed by $S$ and $T\subseteq S$ is any subset, we denote by $M_{T}$
	the direct sum $M_{T}:=\underset{s\in T}{\bigoplus}M_s$. 
	
	\begin{lem}\label{L2.1} Let $\{M_s\}_{s\in S}$ be a family of objects in $\mathcal C$. Let $\lambda$, $\mu\geq \aleph_0$ be cardinals such that
		\begin{equation} \lambda\geq \mbox{max$\{\mbox{$|S|$, $\kappa(G)$}\}$} \qquad  \mu\geq \mbox{sup$\{\mbox{ $||M_s||^G$, $s\in S$}\}$}
		\end{equation} Then, $||\underset{s\in S}{\bigoplus}M_s||^{ G}\leq
		\mu^{\kappa(G)}\times \lambda^{\kappa(G)}$.
	\end{lem}
	
	\begin{proof}
		We consider a subset $T\in \mathcal P_{\kappa(G)}(S)$. Then, we have
		\begin{equation}\label{2.2dq}
			||M_T||^G=|\mathcal C(G,M_T)|\leq \left\vert \mathcal C\left(G,\underset{s\in T}{\prod}M_s\right)\right\vert=\left\vert\underset{s\in T}{\prod}\mathcal C(G,M_s)\right\vert\leq \mu^{|T|}\leq \mu^{\kappa(G)}
		\end{equation} We now note that the direct sum $M_S=\underset{s\in S}{\bigoplus}M_s$ may be expressed as the colimit $\underset{T\in \mathcal P_{\kappa(G)}(S)}{\varinjlim}M_T$. Since this colimit is $\kappa(G)$-directed and $G$ is $\kappa(G)$-presentable, we now have
		\begin{equation}\label{2.3eq}
			||M_S||^G=|\mathcal C(G,M_S)|=\left\vert\mathcal C\left(G,\underset{T\in \mathcal P_{\kappa(G)}(S)}{\varinjlim}M_T\right)\right\vert=\left\vert\underset{T\in \mathcal P_{\kappa(G)}(S)}{\varinjlim}\mathcal C(G,M_T)\right\vert
		\end{equation} Since there is an epimorphism 
		$
		\underset{T\in \mathcal P_{\kappa(G)}(S)}{\bigoplus}\mathcal C(G,M_T)\twoheadrightarrow \underset{T\in \mathcal P_{\kappa(G)}(S)}{\varinjlim}\mathcal C(G,M_T)
		$ in the category of abelian groups, it follows from \eqref{2.2dq} and \eqref{2.3eq} that
		\begin{equation}
			||M_S||^G\leq\left\vert \underset{T\in \mathcal P_{\kappa(G)}(S)}{\bigoplus}\mathcal C(G,M_T) \right\vert \leq  \mu^{\kappa(G)}\times \lambda^{\kappa(G)}
		\end{equation}  The last inequality follows from the fact that $|\mathcal P_{\kappa(G)}(S)|\leq |\mathcal P_{{\kappa(G)}^+}(S)|=|S|^{\kappa(G)}\leq \lambda^{\kappa(G)}$, where ${\kappa(G)}^+$ is the successor of $\kappa(G)$ 
		(see, for instance, \cite[$\S$ 8.2]{Hod}).
	\end{proof}
	
	\begin{Thm}\label{T2.2} Let $U:\mathcal C\longrightarrow \mathcal C$ be an endofunctor that preserves colimits. Let $\lambda^U:={(||UG||^G)}^{\kappa(G)}\times \kappa(G)^{\kappa(G)}$. Then, 
		$||UM||^G\leq \lambda^U\times  (||M||^G)^{\kappa(G)}$ for any object $M\in \mathcal C$. 
		
	\end{Thm}
	
	\begin{proof}
		Since $G$ is a generator, we know that for any $M\in \mathcal C$, the canonical morphism  $G^{\mathcal C(G,M)}\longrightarrow M$ is an epimorphism. Since $U$ preserves colimits, it follows that we have an epimorphism $(UG)^{\mathcal C(G,M)}\twoheadrightarrow UM$ in $\mathcal C$. By the assumption on the generator $G$, it follows that $||UM||^G\leq \left\vert\left\vert (UG)^{\mathcal C(G,M)}\right\vert\right\vert^G$.
		Applying Lemma \ref{L2.1} with $\mu=\mbox{max$\{||UG||^G,\aleph_0\}$}$ and $\lambda=\mbox{max$\{\kappa(G), ||M||^G\}$}$, we have
		\begin{equation}
			||UM||^G\leq \left\vert\left\vert (UG)^{\mathcal C(G,M)}\right\vert\right\vert^G\leq \mu^{\kappa(G)}\times \lambda^{\kappa(G)}\leq {(||UG||^G)}^{\kappa(G)}\times \kappa(G)^{\kappa(G)}\times  \aleph_0^{\kappa(G)} \times  (||M||^G)^{\kappa(G)}
		\end{equation} Since $\kappa(G)$ is infinite, the result is now clear. 
	\end{proof}

	\section{Generators in  Eilenberg-Moore categories} 
	
	We continue with $\mathcal C$ being a Grothendieck category. By definition, a monad $(U,\theta,\eta)$ on $\mathcal C$ is a triple consisting of an endofunctor $U:\mathcal C\longrightarrow \mathcal C$ and natural transformations $\theta: U\circ U\longrightarrow U$, $\eta: 1_{\mathcal C}
	\longrightarrow U$ satisfying associativity and unit conditions similar to usual multiplication. A module $(M,f_M)$ over $(U,\theta,\eta)$ consists of  $M\in \mathcal C$ and a morphism  $f_M:UM\longrightarrow M$ in $\mathcal C$ such that the following compatibilities hold.
	\begin{equation}\label{3.1cq}
		f_M\circ\theta M= f_M\circ Uf_M~~\textup{and}~~f_M\circ\eta M= 1_M
	\end{equation}
	A morphism $g:(M,f_M)\longrightarrow (M',f_{M'})$ of $(U,\theta,\eta)$-modules is given by $g:M\longrightarrow M'$ in $\mathcal C$ such that 
	$f_{M'}\circ Ug= g\circ f_M$.  This gives the standard Eilenberg-Moore category of modules over the monad   $(U,\theta,\eta)$  and we denote it by $EM_U$. When there is no danger of confusion, an object $(M,f_M)\in EM_U$ will often be denoted simply by $M$.
	
	\smallskip For any object $M\in \mathcal C$, we note that $(UM,\theta M:U^2M\longrightarrow UM)$ carries the structure of a module over $(U,\theta,\eta)$. Further, it is well known (see, for instance, \cite{Mac}) that there is an adjunction of functors, given by  natural isomorphisms
	\begin{equation}\label{monadj}
		EM_U(UM,N)\cong \mathcal C(M,N)
	\end{equation}
	for $M\in \mathcal C$ and $N\in EM_U$.
	
	\smallskip
	\begin{thm}\label{L3.1}
		Let $(U,\theta,\eta)$ be a monad on $\mathcal C$ such that $U$ is exact and preserves colimits. Then, $EM_U$ is a Grothendieck  category.  If $\mathcal C$ has a projective generator, so does $EM_U$.
		Further, if $\{M_i\}_{i\in I}$ is any system (resp. any finite system) of objects in $EM_U$, the colimit (resp. the finite limit) in $EM_U$ is defined by taking $\underset{i\in I}{colim}\textrm{ }M_i$ (resp.  $\underset{i\in I}{lim}\textrm{ }M_i$) in $\mathcal C$. 
	\end{thm}
	\begin{proof} Let $g:(M,f_M)\longrightarrow (N,f_N)$ be a morphism in $EM_U$. We set
		\begin{equation}
			K:=Ker(g:M\longrightarrow N)\qquad L=Coker(g:M\longrightarrow N)
		\end{equation} Since $U$ is exact, it is clear that we have induced morphisms $f_K:UK\longrightarrow K$ and $f_L:UL\longrightarrow L$ defining objects
		$(K,f_K)$, $(L,f_L)\in EM_U$. It follows that $EM_U$ contains kernels and cokernels and that $Ker((N,f_N)\longrightarrow Coker(g))=Coker(Ker(g)
		\longrightarrow (M,f_M))$. This makes $EM_U$ an abelian category.  Since $U$ is exact and preserves colimits, we see that $U$ can be used to determine both colimits and finite limits in $EM_U$, and that  $EM_U$ satisfies the (AB5) axiom.

		\smallskip
		Let $G$ be a generator for $\mathcal C$ and let $(M,f_M)\in EM_U$.  We choose 
		an epimorphism $p:G^{(X)}\longrightarrow M$ in $\mathcal C$ from a direct sum of copies of $G$. Since $U$ preserves colimits,  $Up: UG^{(X)}=(UG)^{(X)}\longrightarrow UM$ is an epimorphism in $EM_U$. Additionally, it is clear from the condition $f_M\circ\eta_M= 1_M$  in \eqref{3.1cq} that $f_M:UM\longrightarrow M$ is an epimorphism in $EM_U$. Therefore, $f_M\circ Up: UG^{(X)}=(UG)^{(X)}\longrightarrow UM\longrightarrow M$ is an epimorphism in $EM_U$ and it follows  that $(UG,\theta G)$ is a generator for $EM_U$. Finally, if $G\in \mathcal C$ is projective, then $EM_U((UG,\theta G),\_\_)\cong \mathcal C(G,\_\_)$ is exact and $(UG,\theta G)$ becomes projective in $EM_U$.
	\end{proof}

	\begin{lem}\label{L3.2}
		Suppose that $(U,\theta,\eta)$ is  a monad on $\mathcal C$ which is exact and preserves colimits.  Let $M\in \mathcal C$ be an object and suppose that $M$ is $\kappa(M)$-presentable as an object of $\mathcal C$. Then, $(UM,\theta M)$ is $\kappa(M)$-presentable as an object of $EM_U$. 
	\end{lem}
	
	\begin{proof}
		Let $\{N_i\}_{i\in I}$ be a system of objects in $EM_U$ that is $\kappa(M)$-directed. By Lemma \ref{L3.1}, we know that the underlying object of $N:=\underset{i\in I}{colim}
		\textrm{ }N_i$  in $EM_U$ is given by taking the colimit in $\mathcal C$. We now see that
		\begin{equation}\label{3.4cy}
			EM_U(UM,N)\cong \mathcal C\left(M, {\underset{i\in I}{colim}}\textrm{ }N_i\right)={\underset{i\in I}{colim}}\textrm{ }\mathcal C(M,N_i)={\underset{i\in I}{colim}}\textrm{ }
			EM_U(UM,N_i)
		\end{equation} The result is  now clear.
	\end{proof}
	
	\begin{thm}\label{P3.3} Suppose that $(U,\theta,\eta)$ is  a monad on $\mathcal C$ which is exact and preserves colimits.  Let $G$ be a generator for $\mathcal C$ that is $\kappa(G)$-presentable. Then, $EM_U$ is a locally $\kappa(G)$-presentable category.
	\end{thm}
	
	\begin{proof} From the proof of Lemma \ref{L3.1}, we know that the pair $(UG,\theta G)$ is a generator for the Eilenberg-Moore category $EM_U$. Since $G$ is $\kappa(G)$-presentable, it follows from Lemma \ref{L3.2} that $(UG,\theta G)$ is $\kappa(G)$-presentable as an object of $EM_U$. Hence, $EM_U$ is locally $\kappa(G)$-presentable. 
	\end{proof}
	
	For the sake of convenience, we now fix a regular cardinal
	$
	\delta\geq \mbox{max$\{\kappa(G), ||G||^G\}$} 
	$.  We note in particular that since $\delta\geq \kappa(G)$, the object $G$ is also $\delta$-presentable. If the monad $U$ preserves colimits, we  use Theorem \ref{T2.2} to fix  $\lambda^U = {(||UG||^G)}^{\kappa(G)}\times \kappa(G)^{\kappa(G)}$ such that $||UM||^G\leq \lambda^U\times  (||M||^G)^{\kappa(G)}$ for any object $M\in \mathcal C$. 
	
	\begin{thm}\label{P3.4} Suppose that $(U,\theta,\eta)$ is  a monad on $\mathcal C$ which is exact and preserves colimits.  Let $(M,f_M)\in EM_U$ and consider some $x\in el_G(M)$. Then, there is a subobject $N_x\subseteq M$ in $EM_U$ such that $||N_x||^G\leq \lambda^U$ and $x\in el_G(N_x)$. 
		
	\end{thm}
	
	\begin{proof}
		By definition, $x\in el_G(M)=\mathcal C(G,M)$. By \eqref{monadj}, we have a corresponding morphism $\hat{x}\in EM_U(UG,M)\cong \mathcal C(G,M)$ given by setting 
		$\hat{x}:UG \xrightarrow{Ux}UM\xrightarrow{f_M}M$. By setting $N_x:=Im(\hat{x})$ in $EM_U$, we obtain the commutative diagram
		\begin{equation}\label{3.6cd}
			\begin{tikzcd}
				&    &  {UM}  \arrow{ddr}{f_M} &\\
				&   &   &\\[-2ex]
				G\arrow{r}{\eta G}&   UG\arrow{rr}[swap]{\hat{x}} \arrow{uur}{Ux} \arrow{ddr}{p_x} & & M \\
				&     &   &\\[-2ex] 
				& &  N_x  \arrow{uur}{i_x} &\\
			\end{tikzcd}
		\end{equation} We note that since the composition $G\xrightarrow{\eta G}UG \xrightarrow{Ux}UM\xrightarrow{f_M}M$ gives back $x:G\longrightarrow M$, it follows from \eqref{3.6cd} that
		$x\in el_G(N_x)\subseteq el_G(M)$. Finally, since  $p_x:UG\longrightarrow N_x$ is an epimorphism, it follows   that
		$
		||N_x||^G\leq ||UG||^G\leq \lambda^U
		$. This proves the result.
	\end{proof}
	
	\begin{Thm}\label{T3.5} Suppose that $(U,\theta,\eta)$ is  a monad on $\mathcal C$ which is exact and preserves colimits.  Let $(M,f_M)\in EM_U$ and consider  some $X\subseteq  el_G(M)$. Then, there is a subobject $N_X\subseteq M$ in $EM_U$ such that $||N_X||^G\leq\lambda^U\times \delta^\delta\times {|X|}^\delta$ and $X\subseteq el_G(N_X)$. 
		
	\end{Thm}

	\begin{proof}
		By considering the morphisms $x\in X\subseteq el_G(M)=\mathcal C(G,M)$, we obtain $h_X:G^{(X)}\longrightarrow M$ from a direct sum of copies of $G$. Since $ \delta\geq \mbox{max$\{\kappa(G), ||G||^G\}$}$, it follows from Lemma \ref{L2.1} that
		\begin{equation}
			||G^{(X)}||^G\leq \delta^{\kappa(G)}\times (|X|\times \kappa(G))^{\kappa(G)}\leq \delta^\delta\times {|X|}^\delta
		\end{equation}
		By the adjunction in \eqref{monadj} and the fact that $U$ preserves direct sums, we obtain $\hat{h}_X\in EM_U((UG)^{(X)},M)\cong \mathcal C(G^{(X)},M)$ and a commutative diagram
		\begin{equation}\label{3.7cd}
			\begin{tikzcd}
				&    &  {UM}  \arrow{ddr}{f_M} &\\
				&   &   &\\[-2ex]
				&   (UG)^{(X)}\arrow{rr}[swap]{\hat{h}_X} \arrow{uur}{Uh_X} \arrow{ddr}{p_X} & & M \\
				&     &   &\\[-2ex] 
				& &  N_X  \arrow{uur}{i_X} &\\
			\end{tikzcd}
		\end{equation} 
		by setting $N_X:=Im(\hat{h}_X)$ in $EM_U$. As with \eqref{3.6cd} in the proof of Proposition \ref{P3.4}, it follows from that \eqref{3.7cd}  that $X\subseteq el_G(N_X)\subseteq el_G(M)$. Again, since $U$ preserves colimits and $p_X:(UG)^{(X)}\longrightarrow N_X$ is an epimorphism, it follows  that
		\begin{equation}
			||N_X||^G\leq ||(UG)^{(X)}||^G\leq \lambda^U\times (||G^{(X)}||^G)^{\kappa(G)}\leq \lambda^U\times  (\delta^\delta\times {|X|}^\delta)^\delta=\lambda^U\times \delta^\delta\times {|X|}^\delta
		\end{equation}
	\end{proof}
	
	We now come to the case of Eilenberg-Moore categories of comonads. By definition, a comonad $(V,\delta,\epsilon)$ on $\mathcal C$ 
	consists of an endofunctor $V:\mathcal C\longrightarrow \mathcal C$ and two natural transformations $\delta: V\longrightarrow V\circ V$ and $\epsilon:V\longrightarrow 1_{\mathcal C}$ satisfying coassociativity and counit conditions.
	A comodule over $V$ is a pair $(M,f^M)$ where $M\in \mathcal C$ and $f^M:M\longrightarrow VM$ is a morphism in $\mathcal C$ that satisfies
	\begin{equation}\label{3.9cq}
		\delta M\circ f^M = Vf^M \circ f^M \quad \textup{and} \quad \epsilon M \circ f^M= 1_M.
	\end{equation}
	A morphism $g:(M,f^M)\longrightarrow (M',f^{M'})$ of $V$-comodules is a morphism $g:M\longrightarrow N$ in $\mathcal C$ such that $f^{M'}\circ g= Vg \circ f^M.$ This forms the Eilenberg-Moore category over the comonad $(V, \delta, \epsilon)$, which we denote by $EM^V$. When there is no danger of confusion, we will often write an object 
	$(M,f^M)$ of $EM^V$ simply as $M\in EM^V$.
	
	\smallskip
	
	For any object $M$ in $\mathcal C$, it is easy to observe that $(VM,VM\xrightarrow{\delta M} VVM)$ is a comodule over the comonad $V$. Further, it is known (see, for instance, \cite{Mac}) that for any object $M \in EM^V$ and $N \in \mathcal C,$ there is an adjunction of functors given by natural isomorphisms
	\begin{equation}\label{comonadj}
		EM^V(M,VN)\cong \mathcal C(M,N).
	\end{equation}
Suppose in particular that $(V,\delta,\epsilon)$  is a comonad on $\mathcal C$ such that $V$ is exact and preserves colimits. Then, it follows as in the proof of Proposition \ref{L3.1} that $EM^V$ is an abelian category. Further, if $\{M_i\}_{i\in I}$ is any system (resp. any finite system) of objects in $EM^V$, the colimit (resp. the finite limit) in $EM^V$ is defined by taking $\underset{i\in I}{colim}\textrm{ }M_i$ (resp.  $\underset{i\in I}{lim}\textrm{ }M_i$) in $\mathcal C$. 

\begin{lem}\label{L3.6cf}
Let $(V,\delta,\epsilon)$ be a comonad on $\mathcal C$ such that $V$ is exact. Then, the category $EM^V$ is well powered.
\end{lem}
\begin{proof}
We consider an inclusion $i:(M',f^{M'})\longrightarrow (M,f^M)$ in $EM^V$. We have noted that kernels in $EM^V$ are computed in $\mathcal C$. Accordingly $i:M'\longrightarrow M$
is a monomorphism in $\mathcal C$. Since $V$ is exact, we see that $Vi:VM'\longrightarrow VM$ is a monomorphism in $\mathcal C$. We also know that $Vi\circ f^{M'}=f^M\circ i$. Since $Vi$ is a monomorphism, the composition $f^M\circ i$ factors uniquely through $Vi$. Hence, every subobject of $(M,f^M)$ in $EM^V$ corresponds to a unique subobject of $M$ in $\mathcal C$. Since $\mathcal C$ is a Grothendieck category and hence well powered, the result follows. 
\end{proof}

	\begin{thm}\label{P3.7}
		Let $(V,\delta,\epsilon)$ be a comonad on $\mathcal C$ such that $V$ is exact and preserves colimits. Then $EM^V$ is a Grothendieck category.  
	\end{thm}
	\begin{proof} 
	Since $\mathcal C$ is a Grothendieck category, and colimits and finite limits in $EM^V$ are computed in $\mathcal C$, it follows that filtered colimits are exact in $EM^V.$ 
		It  remains to show that $EM^V$ has a generator. We now proceed in a manner similar to the proof of \cite[Theorem 3.3]{BBK}. Let $(M,f^{M})\in EM^{V}$. Since 
		$G$ is a generator for $\mathcal C$, we can find an epimorphism $g:G^{(I)}\longrightarrow M$ in $\mathcal C$ for some indexing set $I$. Since $V$ is exact and preserves colimits, this induces an epimorphism $Vg:VG^{(I)}=(VG)^{(I)} \longrightarrow VM$ in $EM^V.$ We now consider the following pullback diagram in $EM^{V}$ 
		\begin{equation}
			\begin{CD}\label{3.11cd}
				N @> {} >> M\\
				@V {} VV @VVf^{M} V \\
				(VG)^{(I)} @>Vg  >> VM ~\\
			\end{CD}
		\end{equation}
		From (\ref{3.9cq}), we note that  $f^M:M\longrightarrow VM$ is a monomorphism in $\mathcal C$. Since $f^M$ is also a morphism in $EM^V$ and kernels in 
		$EM^V$ are computed in $\mathcal C$, it follows that $f^M$ is a monomorphism in $EM^V$. From the pullback in \eqref{3.11cd}, it now follows that
		we have a monomorphism $N\longrightarrow  (VG)^{(I)}$ in $EM^V$. Now let $Fin(I)$ denote the collection of finite subsets of $I.$ Then for any $S\in Fin(I)$, we define an object $K_{S}:=(VG)^{(I)}\times_{VM} Im(Vg~\vert_{VG^{(S)}})$ in $EM^{V}.$
		We now consider the following pullback diagrams 
		\begin{equation}\label{3.12cd}
			\begin{CD}
				N\cap K_{S} @> {} >> K_{S}\\
				@VVV @VVV \\
				N @>{}  >> (VG)^{(I)} ~\\
			\end{CD} \qquad \text{and} \qquad
			\begin{CD}
				P_{S} @> {} >> Im(Vg~\vert_{VG^{(S)}})\\
				@V {} VV @VV{} V \\
				M @>f^{M}  >> VM ~\\
			\end{CD}
		\end{equation}
		We note that all the morphisms appearing in \eqref{3.12cd} are monomorphisms. Together, they induce a commutative diagram
		\begin{equation}
			\begin{CD}\label{3.13cd}
				N\cap (VG)^{(S)} @>{} >> N\cap K_{S} @> {} >> P_{S}\\
				@V {} VV @V {} VV @VV {} V \\
				(VG)^{(S)} @>{}  >> K_{S}@> {} >> Im(Vg~\vert_{VG^{(S)}})\\
			\end{CD} 
		\end{equation}
		with all the squares as pullbacks in $EM^{V}.$ Further, as $(VG)^{(S)}\longrightarrow K_{S}\longrightarrow Im(Vg~\vert_{VG^{(S)}})$ is an epimorphism in $EM^{V}$ and $EM^{V}$ is an abelian category, its pullback $N\cap (VG)^{(S)} \longrightarrow N\cap K_{S} \longrightarrow P_{S}$ must be an epimorphism in $EM^{V}.$
		We note that $\underset{S\in Fin(I)}\varinjlim Im(Vg~\vert_{VG^{(S)}}) = Im(Vg)=VM$, since $Vg$ is an epimorphism. Applying filtered colimits over all $S\in Fin(I)$ to the right-hand side diagram in (\ref{3.12cd}), we obtain $M=\underset{S\in Fin(I)}\bigcup P_{S}$. Also from  the right-hand side diagram in (\ref{3.12cd}), it follows that
		$P_S$ is the epimorphic  image of some subobject of $(VG)^{(S)}$. Since $S$ is finite and $EM^V$ is well powered as shown in 
		Lemma \ref{L3.6cf}, we see that   the subobjects of $(VG)^{n}$, $n\in\mathbb N,$ form a set of generators for $EM^V.$
		
	\end{proof}
	
	The following notion will be needed at several places in the paper.
	
	\begin{defn}\label{D3.8}
		Let $(V,\delta, \epsilon)$ be a comonad on $\mathcal C.$ We will say that the comonad $V$ is semiperfect if $EM^V$ has a projective generator.
	\end{defn}
	The above terminology is motivated in two ways. First, we know (see \cite[Definition 3.2.4 ]{DNR}) that a coalgebra $C$ over a field $k$ is right (left) semiperfect if every finite dimensional right (left) $C$-comodule has a projective cover. Second, in \cite{SY}, an abelian category has been called semiperfect if every finitely generated object has a projective cover.

	\smallskip
	
	We conclude this section by giving an example of semiperfect comonads on a category. Let $(U, \theta, \eta)$ be a Frobenius monad (see \cite[Definition 1.1]{RS}). Then, by definition, $U$ is equipped with a natural transformation $\epsilon:U\longrightarrow 1$ such that there exists
	$\rho:1\longrightarrow U^2$ satisfying $U\theta\circ \rho U=\theta U\circ U\rho$ and $U\epsilon \circ \rho=\eta=\epsilon U\circ \rho$. Then we know (see \cite[Theorem 1.6]{RS}) that $(U,U)$ is an adjoint pair and there is a natural transformation  $\delta: U\longrightarrow U^2$ such that $(U, \delta, \epsilon)$ is a comonad. 
	It follows (see \cite[$\S$ 2.6]{BBW}) that the Eilenberg-Moore category $EM_{(U,\theta,\eta)}$ of the monad $(U,\theta,\eta)$ and the Eilenberg-Moore category $EM^{(U,\delta,\epsilon)}$ of the comonad $(U,\delta,\epsilon)$ are isomorphic. Since $(U,U)$ is an adjoint pair, we see that $U$ is exact and preserves colimits. If $\mathcal C$ has a projective generator, it now follows by Proposition \ref{L3.1} that so does 
	$EM_{(U,\theta,\eta)}\cong EM^{(U,\delta,\epsilon)}$.  Therefore, every Frobenius monad determines a semiperfect comonad.

	\section{Modules over a monad quiver}

	We continue with $\mathcal C$ being a Grothendieck category as before. We suppose from now on that the generator $G$ of $\mathcal C$ is projective. A morphism $\phi:(U,\theta,\eta)\longrightarrow (U',\theta',\eta')$ of monads over $\mathcal C$ is a natural transformation $\phi:U\longrightarrow U'$ that satisfies
	\begin{equation}
		\phi\circ \theta =\theta'\circ (\phi\ast\phi): U \circ U \longrightarrow U'\qquad \eta'=\phi\circ \eta: 1\longrightarrow U'
	\end{equation} This forms the category $Mnd(\mathcal C)$ of monads over $\mathcal C$. A morphism $\phi: U\longrightarrow U'$ of monads induces a restriction functor 
	\begin{equation}\label{forg4}
		\phi_*:EM_{U'}\longrightarrow EM_U \qquad (M',f_{M'})\mapsto (M',f_{M'}\circ \phi(M'))
	\end{equation}  Suppose that  
	$U$ and $U'$ are exact and preserve colimits.
	Given $(M,f_M)\in EM_U$, we set
	\begin{equation}\label{4.3qi}
		\phi^*(M):=Coeq\left(U'UM\doublerightarrow{\qquad\theta'(M)\circ (U'\phi(M))\qquad }{U'f_M}U'M\right)
	\end{equation} This determines a functor $\phi^*:EM_{U}\longrightarrow EM_{U'}$ that is left adjoint to $\phi_*$ (see, for instance, \cite[Proposition 1]{Lint}). Then, we know from Proposition \ref{L3.1} that $EM_{U'}$ is a Grothendieck category. By 
	Proposition \ref{L3.1}, it is also clear that the restriction functor $\phi_{*}: EM_{U'}\longrightarrow EM_{U}$ preserves colimits. It follows by \cite[Proposition 8.3.27]{KS} that $\phi_{*}$ also has a right adjoint. From the proof of Proposition \ref{L3.1}, we see that $(UG,\theta G)$ and $(U'G,\theta' G)$ are generators for $EM_U$ and $EM_{U'}$ respectively. We note that for any $(M',f_{M'})\in EM_{U'}$, we have natural isomorphisms
	\begin{equation}\label{4.4dk}
		EM_{U'}(\phi^*(UG,\theta G),(M',f_{M'}))\cong  EM_{U}((UG,\theta G),\phi_*(M',f_{M'}))\cong \mathcal C(G,M')\cong EM_{U'}((U'G,\theta' G),(M',f_{M'}))
	\end{equation} whence it follows by Yoneda lemma that $\phi^*(UG,\theta G)=(U'G,\theta' G)\in EM_{U'}$. 
	
	\begin{defn}\label{D4.0}
		Let $\phi:(U,\theta,\eta)\longrightarrow (U',\theta',\eta')$  be a morphism of monads over $\mathcal C$. We will say that $\phi$ is flat if the functor
		$\phi^*:EM_U\longrightarrow EM_{U'}$ is exact.
	\end{defn}
	
	We now recall that a quiver $\mathbb Q=(\mathbb V,\mathbb  E)$ is a directed graph, consisting of a set of $\mathbb V$ of vertices and a set $\mathbb E$ of edges. We will use $\phi: x\longrightarrow y$ to denote an arrow in $\mathbb Q$ going from $x$ to $y$. We will treat a quiver $\mathbb Q$ as a category in the obvious manner.
	
	\begin{defn}\label{D4.1} Let  $\mathbb Q=(\mathbb V,\mathbb  E)$ be a quiver. A monad quiver over $\mathcal C$ is a functor $\mathscr U:\mathbb Q\longrightarrow Mnd(\mathcal C)$. We will say $\mathscr U$ is flat if for each arrow $\phi:x\longrightarrow y$ in $\mathbb Q$, the induced morphism $\mathscr U(\phi):\mathscr U(x)\longrightarrow \mathscr U(y)$ of monads is flat. For $x\in \mathbb V$, we will often denote the monad $\mathscr U(x)$ by $\mathscr U_x$. 
	\end{defn}
	
	If $\mathscr U$ is a monad quiver over $\mathcal C$ and $\phi: x\longrightarrow y$ is an edge of $\mathbb Q$, by abuse of notation, we will continue to denote $\mathscr U(\phi):
	\mathscr U_x\longrightarrow \mathscr U_y$ simply by $\phi$. Accordingly, we will often write
	$
	\phi^*=\mathscr U(\phi)^*:EM_{\mathscr U_x}\longrightarrow EM_{\mathscr U_y}$ and $\phi_*=\mathscr U(\phi)_*:EM_{\mathscr U_y}\longrightarrow EM_{\mathscr U_x}
	$ for an edge $\phi:x\longrightarrow y$ in $\mathbb Q$.

	\begin{eg}\label{eg4.8p}
		\emph{We now give several examples of situations where the framework of monad quivers would apply. Let $k$ be a field.  We use Sweedler notation for coproducts and
			coactions, with summation symbols suppressed. }
	\end{eg}
	
	\smallskip
	(1) Let $Alg_k$ denote the category of $k$-algebras. Each $A\in Alg_k$ defines a monad $A\otimes_k\_\_$ on the category $Vect_k$ of $k$-vector spaces. If $T:\mathbb Q\longrightarrow Alg_k$ is any functor, we see that $\mathscr U:\mathbb Q\longrightarrow Mnd(Vect_k)$, $x\mapsto T(x)\otimes_k\_\_$ becomes a monad quiver. For any $x\in \mathbb Q$, the category $EM_{\mathscr U_x}$ of $\mathscr U_x$-modules takes values in the category of left $T(x)$-modules. 
	
	\smallskip
	(2) Let $(\mathcal D,\otimes)$ be a $k$-linear monoidal category and let $\mathcal L$ be a $k$-linear Grothendieck category along with an action $\_\_\otimes \_\_:\mathcal D\times 
	\mathcal L\longrightarrow \mathcal L$ such that the functor $X\otimes \_\_:\mathcal L\longrightarrow \mathcal L$ is exact and preserves colimits for any $X\in \mathcal D$. Then, any monoid object
	$A\in Alg(\mathcal D)$ determines a monad $A\otimes\_\_:\mathcal L\longrightarrow \mathcal L$.  If $T:\mathbb Q\longrightarrow Alg(\mathcal D)$ is any functor, we see that $\mathscr U:\mathbb Q\longrightarrow Mnd(\mathcal L)$, $x\mapsto T(x)\otimes \_\_$ becomes a monad quiver. For any $x\in \mathbb Q$, the category $EM_{\mathscr U_x}$ of $\mathscr U_x$-modules takes values in the category of ``left $T(x)$-module objects in $\mathcal L$.''  For instance, we may take $\mathcal D=Vect_k$. Then, any $k$-algebra 
	$R$ determines a monad $R\otimes \_\_:\mathcal L\longrightarrow \mathcal L$ on $\mathcal L$. We note that the categories of ``$R$-module objects in $\mathcal L$'' play a key role in the theory of noncommutative projective schemes studied by Artin and Zhang \cite{AZ1}, \cite{AZ2}. 
	
	\smallskip
	(3)  Let $H$ be a Hopf algebra over $k$. Then, the category $H-Mod$ of left $H$-modules is monoidal, with $H$-action on the tensor product  given by $h(m\otimes n):=h_{(1)}m\otimes h_{(2)}n$ for $h\in H$, 
	$m\in M$, $n\in N$ where $M$, $N\in H-Mod$. Let $A$ be an $H$-module algebra, i.e., a monoid object in $H-Mod$. Then, $A\otimes_k\_\_:H-Mod\longrightarrow H-Mod$ is a monad on $H-Mod$ that is exact and preserves colimits. If $T:\mathbb Q\longrightarrow Alg(H-Mod)$ is a functor taking values in the category $Alg(H-Mod)$ of monoids in $H-Mod$, it is clear that
	\begin{equation}
		\mathscr U:\mathbb Q\longrightarrow Mnd(H-Mod)\qquad x\mapsto T(x)\otimes_k\_\_:H-Mod\longrightarrow H-Mod
	\end{equation} determines a monad quiver over $H-Mod$. For any $x\in \mathbb Q$, the category $EM_{\mathscr U_x}$ of $\mathscr U_x$-modules takes values in the   category of left $T(x)$-module objects in $H-Mod$.
	
	\smallskip
	(4) Let $H$ be a Hopf algebra over $k$ and let $Comod-H$ be the category of right $H$-comodules. If $A$ is a right $H$-comodule algebra, the category  $Mod_A^H$of 
	right $(A,H)$-Hopf modules has been extensively studied in the literature (see, for instance, \cite{BBR0}, \cite{CMZ}, \cite{CMIZ}, \cite{CG}).  An object $M\in Mod_A^H$ has a right $A$-module structure and a right $H$-comodule structure that are compatible in the sense that \begin{equation} (ma)_{(0)}\otimes (ma)_{(1)}=m_{(0)}a_{(0)}\otimes m_{(1)}a_{(1)} \qquad m\in M, a\in A
	\end{equation} We know that $Mod_A^H$ is a Grothendieck category (see \cite[$\S$ 1]{CG}). For any right $H$-comodule algebra
	$B$ and any $M\in Mod_A^H$, it may be verified that $B\otimes M\in Mod_A^H$ with $A$-action $(b\otimes m)\cdot a:=b\otimes ma$ and $H$-coaction $(b\otimes m)_{(0)}\otimes 
	(b\otimes m)_{(1)}=b_{(0)}\otimes m_{(0)}\otimes b_{(1)}m_{(1)}$ for $a\in A$, $b\in B$ and $m\in M$. Accordingly, any such $B\otimes_k\_\_:Mod_A^H\longrightarrow Mod_A^H$ is a monad that is exact and preserves colimits. Its Eilenberg-Moore category consists of right $(B^{op}\otimes A,H)$-Hopf modules. 
	
	\smallskip If $T:\mathbb Q\longrightarrow Alg(Comod-H)$ is a functor taking values in the category $Alg(Comod-H)$ of right $H$-comodule algebras, we see that 
	$
	\mathscr U:\mathbb Q\longrightarrow Mnd(Mod_A^H)$, $x\mapsto T(x)\otimes_k\_\_:Mod_A^H\longrightarrow Mod_A^H
	$ determines a monad quiver on $Mod_A^H$. For any $x\in \mathbb Q$, the category $EM_{\mathscr U_x}$ of $\mathscr U_x$-modules takes values in the category of right $(T(x)^{op}\otimes A,H)$-Hopf modules.

	\smallskip
	(5) Let $(\mathcal D,\otimes)$ be a multitensor category, i.e., a locally finite $k$-linear abelian rigid monoidal category (see \cite[$\S$ 4.1]{Et}). Let $\mathcal L$ be a  locally
	finite $k$-linear abelian category that carries the structure $\otimes:\mathcal D\times \mathcal L\longrightarrow \mathcal L$ of a $\mathcal D$-module category with $\_\_\otimes \_\_$ being exact in the first variable (see \cite[$\S$ 7.3]{Et}). In this situation, it can be shown (see \cite[$\S$ 4.2.1, $\S$ 7.3]{Et}) that the functors $\otimes:\mathcal D\times \mathcal D\longrightarrow \mathcal D$ and $\otimes:\mathcal D\times \mathcal L\longrightarrow \mathcal L$
	are exact in both variables. As such, if $A\in Alg(\mathcal D)$ is a monoid object in $\mathcal D$, the functor $A\otimes\_\_:\mathcal L\longrightarrow \mathcal L$ determines a monad
	on $\mathcal L$ that is exact. 
	
	\smallskip
	In this setup, the category $\mathcal L$ is locally finite (see \cite[$\S$ 1.8]{Et}) and therefore does not contain arbitrary direct sums. Accordingly, we consider   the ind-completion 
	$Ind(\mathcal L)$ of $\mathcal L$. Since $\mathcal L$ is essentially small, $Ind(\mathcal L)$ must be a Grothendieck category (see \cite[Theorem 8.6.5]{KS}). For $A\in Alg(\mathcal D)$, 
	the monad $A\otimes\_\_:\mathcal L\longrightarrow \mathcal L$ extends canonically to a monad $A\otimes\_\_:Ind(\mathcal L)\longrightarrow Ind(\mathcal L)$ on $Ind(\mathcal L)$. Since
	$A\otimes\_\_:\mathcal L\longrightarrow \mathcal L$  is exact, so is  $A\otimes\_\_:Ind(\mathcal L)\longrightarrow Ind(\mathcal L)$  (see \cite[Corollary 8.6.8]{KS}).  By the universal property of the ind-completion, we know that the extension $A\otimes\_\_:Ind(\mathcal L)\longrightarrow Ind(\mathcal L)$ preserves filtered colimits. Since every colimit can be expressed as a combination of a finite colimit and a filtered colimit (see, for instance, \cite[Tag 002P]{Stacks}), it now follows that $A\otimes\_\_:Ind(\mathcal L)\longrightarrow Ind(\mathcal L)$ preserves all colimits. Now if 
	$T:\mathbb Q\longrightarrow Alg(\mathcal D)$ is any functor, we note that $\mathscr U:\mathbb Q\longrightarrow Mnd(Ind(\mathcal L))$, $x\mapsto T(x)\otimes\_\_:Ind(\mathcal L)
	\longrightarrow Ind(\mathcal L)$ gives a monad quiver on $Ind(\mathcal L)$.
	
	\smallskip
	This particular setup of a locally finite module category $\mathcal L$ over a multitensor category $\mathcal D$ is especially interesting, since it has a large number of naturally occurring examples
	in the literature (see \cite[$\S$ 7.4]{Et}).
	
	\smallskip
	(a) Let $(\mathcal E,\otimes)$  be a multitensor category and let $(\mathcal D,\otimes)$ be a multitensor subcategory. Then, $\mathcal E$ carries the structure of a $\mathcal D$-module category in an obvious manner. More generally, if $F:(\mathcal D,\otimes)\longrightarrow (\mathcal E,\otimes)$ is a tensor functor between multitensor categories, then $\mathcal E$ carries the structure of a $\mathcal D$-module category with $X\otimes Y:=F(X)\otimes Y$ for $X\in \mathcal D$, $Y\in \mathcal E$. 
	
	\smallskip
	(b) Let $G$ be a finite group. Then,  the category $Rep(G)$ of finite dimensional representations of $G$ over a field $k$ is a multitensor category (see \cite[$\S$ 4.1.2]{Et}). If $H\subseteq G$ is a subgroup, then the restriction 
	$Rep(G)\longrightarrow Rep(H)$ is a tensor functor, which makes $Rep(H)$ into a $Rep(G)$-module category.
	
	\smallskip
	(c) Let $G$ be a finite group and let $Vec_G$ be the category of finite dimensional $G$-graded $k$-vector spaces. Then, $Vec_G$ is a multitensor category (see \cite[$\S$ 4.1.2]{Et}). A  module category $\mathcal L$ over $Vec_G$ is a category with a $G$-action, i.e., there are autoequivalences (see \cite[$\S$ 7.4.10]{Et}) $F_g:\mathcal L\longrightarrow \mathcal L$, $g\in G$ along with isomorphisms
	\begin{equation*}
		\eta_{g,h}:F_g\circ F_h\longrightarrow F_{gh} \qquad g,h\in G
	\end{equation*} satisfying $\eta_{gh,k}\circ \eta_{g,h}=\eta_{g,hk}\circ \eta_{h,k}$ for $g,h,k\in G$.
	\subsection{Cis-modules over a monad quiver}
	\begin{defn}\label{D4.2} Let $\mathscr U:\mathbb Q=(\mathbb V,\mathbb  E)\longrightarrow Mnd(\mathcal C)$ be a monad quiver over $\mathcal C$. A cis-module $\mathscr M$ over $\mathscr U$ consists of a collection $\{\mathscr M_x\in EM_{\mathscr U_x}\}_{x\in \mathbb V}$ connected by morphisms $\mathscr M_\phi: \mathscr M_x\longrightarrow \phi_*\mathscr M_y$ in $EM_{\mathscr U_x}$ (equivalently, morphisms $\mathscr M^\phi:\phi^*\mathscr M_x\longrightarrow \mathscr M_y$ in $EM_{\mathscr U_y}$) for each edge $\phi: x\longrightarrow y$ in $\mathbb E$ such that $\mathscr M_{id_x}=id_{\mathscr M_x}$ for each $x\in \mathbb V$ and 
		$\phi_*(\mathscr M_\psi)\circ \mathscr M_\phi=\mathscr M_{\psi\phi}:\mathscr M_x\xrightarrow{\mathscr M_\phi} \phi_*\mathscr M_y
		\xrightarrow{\phi_*(\mathscr M_\psi)}\phi_*\psi_*\mathscr M_z$ (equivalently, $\mathscr M^\psi\circ \psi^*(\mathscr M^\phi)=\mathscr M^{\psi\phi}$)
		for each pair of composable morphisms $x\xrightarrow{\phi}y\xrightarrow{\psi}z$ in $\mathbb Q$.
		
		\smallskip
		A morphism $\xi:\mathscr M\longrightarrow \mathscr M'$ of cis-modules over $\mathscr U$ consists of morphisms $\xi_x:\mathscr M_x\longrightarrow \mathscr M'_x$ for each $x\in \mathbb V$ such that $\mathscr M'_\phi\circ \xi_x=\phi_*(\xi_y)\circ \mathscr M_\phi $ (equivalently, $\mathscr M'^\phi\circ \phi^*(\xi_x)=\xi_y\circ \mathscr M^\phi$) for each edge $\phi: x\longrightarrow y$ in $\mathbb E$. We denote the category of cis-modules over $\mathscr U$ by
		$Mod^{cs}-\mathscr U$.
		
		\smallskip
		Additionally, we say that  $\mathscr M \in Mod^{cs}-\mathscr U$  is cartesian if for each edge $x \xrightarrow{\psi} y$ in $\mathbb E$, the morphism $\mathscr M^{\psi} : \psi^{*}\mathscr M_x \longrightarrow \mathscr M_y$ is an isomorphism in $EM_{\mathscr U_y}$. We denote by $Mod_{c}^{cs}-\mathscr U$ the full subcategory of $Mod^{cs}-\mathscr U$ consisting of cartesian cis-modules. 
	\end{defn}
	
	In this subsection, we assume that the functor $\mathscr U:\mathbb Q\longrightarrow Mnd(\mathcal C)$ takes values in monads which are exact and preserve colimits. From the definition in \eqref{forg4} and the proof of Proposition \ref{L3.1}, it is now clear that the restriction functors $\phi_*=\mathscr U(\phi)_*:EM_{\mathscr U_y}\longrightarrow EM_{\mathscr U_x}
	$ are exact for each edge $\phi:x\longrightarrow y$. Further, $Mod^{cs}-\mathscr U$ becomes an abelian category, with kernel $Ker(\xi)_x=Ker(\xi_x)$ and
	$Coker(\xi)_x=Coker(\xi_x)$ computed pointwise for any morphism $\xi:\mathscr M\longrightarrow \mathscr M'$. 
	
	\smallskip
	For an object $\mathscr M\in Mod^{cs}-\mathscr U$, we now set 
	\begin{equation}\label{el45t}
		el_G(\mathscr M):=\underset{x\in \mathbb V}{\coprod}EM_{\mathscr U_x}(\mathscr U_xG,\mathscr M_x)=\underset{x\in \mathbb V}{\coprod}\mathcal C(G,\mathscr M_x)
	\end{equation} From \eqref{el45t} it is clear that for any subobject $\mathscr M'\subseteq \mathscr M$ in $Mod^{cs}-\mathscr U$, we must have
	$el_G(\mathscr M')\subseteq el_G(\mathscr M)$. Additionally, since $\mathscr U_xG$ is a generator for $EM_{\mathscr U_x}$, we see that
	the subobject $\mathscr M'\subseteq \mathscr M$ equals $\mathscr M$ if and only if $el_G(\mathscr M')=el_G(\mathscr M)$. 
	
	\smallskip
	We will now use an argument similar to our previous work in \cite{Ban}, \cite{BBR} which is motivated by the work of Estrada and Virili \cite{EV}  (see also Enochs and Estrada \cite{EE}, \cite{f16}). For this, we start by fixing some $\zeta\in el_G(\mathscr M)$, i.e., $\zeta:\mathscr U_xG\longrightarrow \mathscr M_x$ is a morphism in $EM_{\mathscr U_x}$ for some $x\in \mathbb V$. For each $y\in \mathbb V$, we set
	\begin{equation}\label{start4}
		\mathscr P_y:=Im\left(\underset{\psi\in \mathbb Q(x,y)}{\bigoplus}\psi^*\mathscr U_xG\xrightarrow{\psi^*\zeta}\psi^*\mathscr M_x\xrightarrow{\mathscr M^\psi}\mathscr M_y\right)=\underset{\psi\in \mathbb Q(x,y)}{\sum}Im\left(\psi^*\mathscr U_xG\xrightarrow{\psi^*\zeta}\psi^*\mathscr M_x\xrightarrow{\mathscr M^\psi}\mathscr M_y\right)\in EM_{\mathscr U_y}
	\end{equation} For each $y\in \mathbb V$, let $\iota_y: \mathscr P_y\hookrightarrow \mathscr M_y$ be the inclusion. For each $\psi\in 
	\mathbb Q(x,y)$, we have a canonical morphism $\zeta'_\psi:\psi^*\mathscr U_xG\longrightarrow\mathscr P_y$ determined by \eqref{start4}.
	
	\begin{thm}\label{P4.5}
		For an edge $y\xrightarrow{\phi}z$ in $\mathbb{Q},$ the morphism $\mathscr{M}_{\phi} : \mathscr{M}_{y} \longrightarrow \phi_{*}\mathscr{M}_{z}$ restricts to a morphism $\mathscr{P}_{\phi} : \mathscr{P}_{y} \longrightarrow \phi_{*}\mathscr P_z$ such that 
		\begin{equation}
			\phi_*(\iota_z)\circ\mathscr P_\phi=\mathscr M_\phi\circ\iota_y
		\end{equation}
		As such, the objects $\{\mathscr{P}_x \in EM_{\mathscr{U}_{x}}\}_{x \in \mathbb{V}}$ together determine a subobject $\mathscr{P} \subseteq \mathscr{M}$ in $Mod^{cs}-\mathscr U$.
	\end{thm}
	\begin{proof}
		Since $\iota_{z} : \mathscr{P}_{z} \hookrightarrow \mathscr{M}_{z}$ is a monomorphism and $\phi_{*}$ is a right adjoint, $\phi_{*}(\iota_z) $ is also a monomorphism. We claim that the composition $\mathscr P_y\xrightarrow{\iota_y}\mathscr M_y \xrightarrow{\mathscr M_\phi}\phi_*(\mathscr M_z)$ factors through $\iota_{z} : \mathscr{P}_{z} \hookrightarrow \mathscr{M}_{z}$.  Since $\mathscr{U}_yG$ is a projective generator for the Grothendieck category $EM_{\mathscr U_y}$, it suffices  (see \cite[Lemma $3.2$]{Ban}) to show that for any morphism $\tau : \mathscr{U}_{y}G \longrightarrow \mathscr{P}_y$, there exists a morphism $\tau' : \mathscr{U}_{y}G \longrightarrow \phi_{*}\mathscr{P}_z$ such that $\phi_{*}(\iota_z) \circ \tau' = \mathscr{M}_{\phi} \circ \iota_{y}\circ \tau.$ By \eqref{start4}, there is an epimorphism \begin{equation}
			\underset{\psi\in \mathbb Q(x,y)}{\bigoplus} \zeta'_{\psi}: \underset{\psi\in \mathbb Q(x,y)}{\bigoplus} \psi^{*}\mathscr{U}_{x}G \longrightarrow \mathscr{P}_y
		\end{equation}
		in $EM_{\mathscr U_y}$. Since $\mathscr{U}_{y}G$ is projective in $EM_{\mathscr U_y}$, we can lift the morphism $\tau : \mathscr U_yG \longrightarrow \mathscr P_y$ to  $\tau'' : \mathscr U_yG \longrightarrow \underset{\psi\in \mathbb Q(x,y)}{\bigoplus}  \psi^{*}\mathscr U_xG$ such that 
		$
		\tau = \left(\underset{\psi\in \mathbb Q(x,y)}{\bigoplus} \zeta'_{\psi}\right) \circ \tau''
		$.
		By \eqref{start4}, we know that for each $\psi\in \mathbb Q(x,y)$, the composition $\psi^{*}\mathscr U_{x}G \xrightarrow{\zeta'_{\psi}} \mathscr P_y \xrightarrow{\iota_{y}} \mathscr M_y$ factors through $\psi^{*}\mathscr M_x$ as 
		$
		\iota_y \circ \zeta'_{\psi} = \mathscr{M}^\psi \circ \psi^{*}\zeta.
		$ 
		Then applying $\phi^{*}$ and composing with $\mathscr M^{\phi},$ we get \begin{equation}
			\mathscr{M}^{\phi} \circ \phi^{*}(\iota_{y}) \circ \phi^{*}(\zeta'_{\psi}) = \mathscr{M}^{\phi} \circ \phi^{*}(\mathscr M^{\psi}) \circ \phi^{*}(\psi^{*}\zeta) = \mathscr{M}^{\phi\psi} \circ \phi^{*}\psi^{*}\zeta
		\end{equation}
		which clearly factors through $\iota_{z} : \mathscr{P}_z \longrightarrow \mathscr M_z.$ Since $(\phi^{*}, \phi_{*})$ is an adjoint pair, it follows that the composition $\psi^*\mathscr{U}_xG \xrightarrow{\zeta'_\psi}\mathscr P_y\xrightarrow{\iota_y}\mathscr M_y \xrightarrow{\mathscr{M}_{\phi}}\phi_{*}\mathscr{M}_z$ factors through   $\phi_{*}(\iota_{z}) : \phi_{*} \mathscr P_{z} \longrightarrow \phi_{*}\mathscr{M}_{z}.$ The result is now clear.
	\end{proof}

	\begin{lem}\label{L4.5dk}
		Let $\zeta'_{1} : \mathscr U_{x}G \longrightarrow \mathscr{P}_x$ be the canonical morphism corresponding to the identity map in $\mathbb{Q}(x,x)$. Then, for any $y \in \mathbb{V}$, we have 
		\begin{equation}
			\mathscr{P}_y = Im\left(\underset{\psi\in \mathbb Q(x,y)}{\bigoplus}\psi^*\mathscr U_xG\xrightarrow{\psi^*\zeta'_1}\psi^*\mathscr P_x \xrightarrow{\mathscr P^\psi}\mathscr P_y\right)
		\end{equation}
	\end{lem}
	\begin{proof}
		Let $x\xrightarrow{\psi}y$ be an edge in $\mathbb Q$. We consider the following commutative diagram
		\begin{equation}
			\begin{CD}
				\psi^{*}\mathscr{U}_xG @>\psi^{*}\zeta'_{1} >> \psi^{*}\mathscr P_x @>\mathscr P^{\psi} >> \mathscr P_{y} \\
				@. @VV\psi^{*}(\iota_{x})V @VV\iota_{y}V \\
				@.\psi^{*}\mathscr M_x @>\mathscr M^{\psi}>> \mathscr{M}_y \\
			\end{CD}
		\end{equation} 
		Clearly, $\iota_x \circ \zeta'_{1} = \zeta$. Applying $\psi^{*}$, composing with $\mathscr{M}^{\psi}$ and using the fact that $\iota_y$ is monic, we obtain
		\begin{equation}
			Im(\mathscr{M}^{\psi} \circ \psi^{*}\zeta) = Im(\mathscr{M}^{\psi} \circ \psi^{*} (\iota_{x})\circ \psi^{*}\zeta'_{1}) = Im(\iota_{y} \circ \mathscr{P}^{\psi}\circ \psi^{*}\zeta'_{1}) = Im(\mathscr{P}^{\psi} \circ \psi^{*}\zeta'_{1})
		\end{equation}
		The result now follows from \eqref{start4}.
	\end{proof}
	
	We now fix an infinite regular cardinal $\gamma$ such that
	\begin{equation}\label{crd4rf}
		\gamma\geq sup\{\mbox{$|Mor(\mathbb Q)|$, $\kappa(G)$, $||\mathscr U_yG||^G$, $y\in Ob(\mathbb Q)$} \}
	\end{equation}
	
	\begin{lem}\label{L4.6tte}
		We have $|el_G(\mathscr P)|\leq  \gamma^{\kappa(G)}$.
	\end{lem}
	
	\begin{proof}
		For each $\psi\in  \mathbb Q(x,y)$, we know that $\psi^*\mathscr U_xG=\mathscr U_yG\in EM_{\mathscr U_y}$. From Lemma \ref{L4.5dk}, it now follows that $\mathscr P_y$ is a quotient of $\underset{\psi\in \mathbb Q(x,y)}{\bigoplus}\mathscr U_yG$. We recall that $\mathscr U_yG$ is projective in $EM_{\mathscr U_y}$. Using Lemma \ref{L2.1} and the assumption in \eqref{crd4rf}, we now see that
		\begin{equation}\label{413ed}
			|EM_{\mathscr U_y}(\mathscr U_yG,\mathscr P_y)|\leq \left\vert EM_{\mathscr U_y}\left(\mathscr U_yG,\underset{\psi\in \mathbb Q(x,y)}{\bigoplus}\mathscr U_yG\right)\right\vert=\left\vert \mathcal C\left(G,\underset{\psi\in \mathbb Q(x,y)}{\bigoplus}\mathscr U_yG\right)\right\vert\leq \gamma^{\kappa(G)}
		\end{equation} From the definition in \eqref{el45t} and the assumption in \eqref{crd4rf}, the result is now clear.
	\end{proof}
	
	\begin{Thm}\label{Th4.7b}
		 $Mod^{cs}-\mathscr U$ is a Grothendieck category.
	\end{Thm}
	
	\begin{proof} Both filtered colimits and finite limits in $Mod^{cs}-\mathscr U$ are computed pointwise at each vertex $x\in \mathbb V$. Hence, they commute with each other and $Mod^{cs}-\mathscr U$ satisfies (AB5). We take $\mathscr M\in Mod^{cs}-\mathscr U$ and some $\zeta\in el_G(\mathscr M)$, given by $\zeta: \mathscr U_xG\longrightarrow \mathscr M_x$ for some $x\in \mathbb V$. We consider the subobject $\mathscr P\subseteq \mathscr M$ corresponding to $\zeta$ as in 
		Proposition \ref{P4.5}. From the definition in \eqref{start4}, we know that $\zeta\in el_G(\mathscr P)$. From Lemma \ref{L4.6tte}, we know that 
		$|el_G(\mathscr P)|\leq  \gamma^{\kappa(G)}$. 
		
		\smallskip
		By Proposition \ref{3.1cq}, each $EM_{\mathscr U_x}$ is a Grothendieck category, and hence well-powered. Since $\mathscr U_xG$ is a generator for 
		$EM_{\mathscr U_x}$, the object $\mathscr M'_x$ for any $\mathscr M'\in Mod^{cs}-\mathscr U$ can be expressed as a quotient of $(\mathscr U_xG)^{EM_{\mathscr U_x}(\mathscr U_xG,\mathscr M'_x)}$ over some subobject. Hence, the isomorphism classes of cis-modules 
		$\mathscr M'$ satisfying $|el_G(\mathscr M')|\leq \gamma^{\kappa(G)}$ form a set. It is now clear that this  collection gives a set of generators
		for $Mod^{cs}-\mathscr U$.
	\end{proof}
	\subsection{Trans-modules over a monad quiver}
	Let $(U, \theta, \eta)$ be monad on $\mathcal C$ that preserves colimits. Since $\mathcal C$ is a Grothendieck category, it follows from  \cite[Proposition 8.3.27]{KS} $U$ has a right adjoint $V$. Then, $V$ is canonically equipped with the structure of a comonad $(V,\delta, \epsilon)$. For monads $(U, \theta, \eta)$ and $(U', \theta', \eta')$ on $\mathcal C$ that preserve colimits,we have comonads $(V,\delta, \epsilon)$ and $(V',\delta',\epsilon')$ on $\mathcal C$ such that $(U,V)$ and $(U',V')$ are pairs of adjoint functors. We now consider a natural transformation $\phi: U\longrightarrow U'$ in $Mnd(\mathcal C).$  By\cite[Lemma 6.5]{MW},  $\phi$ induces a morphism $\bar{\phi}:V'\longrightarrow V$ of comonads $V$ and $V'$ on $\mathcal C$. 
	
	\smallskip
By \cite[\S 2.6]{BBW}), we know that $EM_{U}\cong EM^{V}$ and $EM_{U'}\cong EM^{V'}$. Accordingly, we may treat an object $(M,f_M)\in EM_U$ as an object 
$(M,f^M)\in EM^V$. Suppose that $V$ and $V'$ are exact. There is now a  functor $\hat{\phi}:EM_{U}\cong EM^{V}\longrightarrow EM^{V'}\cong EM_{U'}$ defined by setting
	\begin{equation}\label{eq4.17}
		\hat{\phi}(M):=Eq\bigg(V'M\doublerightarrow {\qquad (V'\bar{\phi}(M))\circ \delta'(M)\qquad}{V'f^{M}} V'VM\bigg).
	\end{equation} for each $(M,f_M)\in EM_U$. The following fact is well known.
	\begin{lem}\label{L4.9} The functor $\hat\phi:EM_U\longrightarrow EM_{U'}$ is right adjoint to 
	$\phi_*:EM_{U'}\longrightarrow EM_U$,  i.e., for each $M\in EM_U$ and $M'\in EM_{U'}$, we have natural isomorphisms
		\begin{equation*} EM_{U}(\phi_{*}(M'), M)\cong EM_{U'}(M', \hat{\phi}(M)) 
		\end{equation*}
	\end{lem}
	\begin{proof}
		Let $f:M'\longrightarrow \hat{\phi}(M)$ be a morphism in $EM_{U'}.$ Since, $EM_{U'}\cong EM^{V'}$, $f$ may be treated as a morphism in $EM^{V'}$. Then it 
		may be verified that the composition $M'\xrightarrow {f} \hat{\phi}(M) \hookrightarrow V'M \xrightarrow{\epsilon'_{M}} M$ is a morphism in $EM^{V}$ and hence a morphism in $EM^{U}$. 
		
		\smallskip
		Conversely, let $g:\phi_{*}(M')\longrightarrow M$ be a morphism in $EM_{U}$. Then $g$ may be treated as a morphism in $EM^{V}$, i.e., $f^{M}\circ g= Vg \circ f^{\phi_{*}M'}= Vg \circ \bar{\phi}M'\circ f^{M'}.$ It follows  that
		\begin{align*}
			V'(f^{M}\circ  g)  \circ f^{M'} &= V'Vg \circ V'\bar{\phi}(M')\circ V'f^{M'}\circ f^{M'}\\\notag &= V'\bar{\phi}M\circ V'V'g\circ \delta'{M'}\circ f^{M'} \\\notag &= V'\bar{\phi}M \circ \delta'{M} \circ V'g \circ f^{M'} 	
		\end{align*}
		From the equalizer in \eqref{eq4.17}, it follows that the morphism $V'g\circ f^{M'}$ factors through $\hat\phi(M).$ Hence, we obtain a morphism in $EM^{V'}(M', \hat{\phi}(M))$ and hence in $EM_{U'}(M', \hat{\phi}(M)).$   It may be verified that these two associations are inverse to each other.
	\end{proof}
	In this subsection, we consider a monad quiver $\mathscr U:\mathbb Q\longrightarrow Mnd(\mathcal C)$ taking values in monads that are exact and preserve colimits. For each vertex $x\in \mathbb V$, 
	let  $\mathscr V_x$ denote the comonad that is right adjoint to the monad $\mathscr U_x$. We will assume that each $\mathscr V_x$ is exact.
	\begin{defn}\label{D4.10}
		A trans-module $\mathscr M$ over $\mathscr U$ consists of a collection $\{\mathscr M_x \in EM_{\mathscr U_x}\}_{x\in\mathbb V}$ connected by morphisms $_\phi\mathscr M: \mathscr M_y \longrightarrow \hat{\phi}\mathscr M_x$ in $EM_{\mathscr U_y}$ (equivalently, morphisms $^{\phi}\mathscr M:\phi_{*}\mathscr M_y \longrightarrow \mathscr M_x$ in $EM_{\mathscr U_x}$) for each edge $\phi:x\longrightarrow y$ in $\mathbb E$ such that for each $x\in\mathbb V$, we have $_{id_x}\mathscr M=id_{\mathscr M_{x}}$, and $\hat{\psi}(_{\phi}\mathscr M)\circ_{\psi}\mathscr M= _{\psi\phi}\mathscr M$ (equivalently, $^{\phi}\mathscr M\circ\phi_{*}(^{\psi}\mathscr M)= ^{\psi\phi}\mathscr M$) for every pair of composable morphisms $x\xrightarrow{\phi}y\xrightarrow{\psi}z$ in $\mathbb E$.
		
		\smallskip
		
		A morphism $\xi: \mathscr M \longrightarrow \mathscr M'$ of trans-modules over $\mathscr U$ is a collection $\{\xi_x:\mathscr M_x\longrightarrow \mathscr M'_x~|~x\in \mathbb V\}$  of morphisms in $EM_{\mathscr U_x}$ such that $_\phi{\mathscr M'}\circ \xi_y = \hat{\phi}(\xi_x) \circ $$ _\phi{\mathscr M}$ (equivalently, $^{\phi}\mathscr M' \circ \phi_{*}(\xi_y)= \xi_x\circ   { }^{\phi}\mathscr M$) for each edge $\phi:x\longrightarrow y$ in $\mathbb E$.
		This determines the category of trans-modules over $\mathscr U$, and we denote it by $Mod^{tr}-\mathscr U.$
		
		\smallskip
		Further, we say that $\mathscr M \in Mod^{tr}-\mathscr U$ is cartesian if for each edge $\phi:x\longrightarrow y$ in $\mathbb E,$ the morphism $_{\phi}\mathscr M:\mathscr M_y\longrightarrow \Hat{\phi}\mathscr M_x$ is an isomorphism in $EM_{\mathscr U_y}.$ This defines the full subcategory $Mod_{c}^{tr}-\mathscr U$ of cartesian trans-modules in $Mod^{tr}-\mathscr U$.
	\end{defn}
	
	Since $\phi_{*}$ is exact and preserves colimits, we see that for any morphism $\xi:\mathscr M\longrightarrow \mathscr M'$ in $Mod^{tr}-\mathscr U$, $Ker(\xi)$ and $Coker(\xi)$ exist in $Mod^{tr}-\mathscr U$ with $Ker(\xi)_x:=Ker(\xi_x)$ and $Coker(\xi)_x:= Coker(\xi_x)$ for each $x\in \mathbb V.$ Hence, $Mod^{tr}-\mathscr U$ becomes an abelian category.
	Now for an object $\mathscr M\in Mod^{tr}-\mathscr U$, we set 
	\begin{equation}\label{eq4.18}
		el_G(\mathscr M):=\underset{x\in \mathbb V}{\coprod}EM_{\mathscr U_x}(\mathscr U_xG,\mathscr M_x)=\underset{x\in \mathbb V}{\coprod}\mathcal C(G,\mathscr M_x)
	\end{equation} 
	We note that for any subobject $\mathscr M'\subseteq \mathscr M$ in $Mod^{tr}-\mathscr U$, we have $el_G(\mathscr M')\subseteq el_G(\mathscr M)$. Since $G$ is a generator for $\mathcal C$, the subobject $\mathscr M'\subseteq \mathscr M$ equals $\mathscr M$ if and only if $el_G(\mathscr M')=el_G(\mathscr M)$.
	We now use arguments similar to Section 4.1 to show the existence of generators in $Mod^{tr}-\mathscr U$. First, we consider an element $\zeta\in el_{G}(\mathscr M).$ Then, there exists some $x\in \mathbb V$ such that $\zeta:\mathscr U_{x}G\longrightarrow \mathscr M_x$ is a morphism in $EM_{\mathscr U_x}$. For each $y\in\mathbb V,$ we set
	\begin{equation}\label{eq4.19}
		\mathscr P_y:= Im \left(\underset{\psi\in \mathbb Q(y,x)}{\bigoplus}\psi_{*}\mathscr U_xG\xrightarrow{\psi_{*}\zeta}\psi_{*}\mathscr M_x\xrightarrow{^\psi\mathscr M}\mathscr M_y\right) = \underset{\psi\in \mathbb Q(y,x)} \sum Im \left(\psi_{*}\mathscr U_xG\xrightarrow{\psi_{*}\zeta}\psi_{*}\mathscr M_x\xrightarrow{^\psi\mathscr M}\mathscr M_y\right)~.
	\end{equation}
	For each $y\in\mathbb V$, we have an inclusion $\iota_y:\mathscr P_y\hookrightarrow \mathscr M_y$. From \eqref{eq4.19}, we have  for each $\psi\in \mathbb Q(y,x)$,  a  morphism $\zeta'_{\psi}: \psi_{*}\mathscr U_xG\longrightarrow \mathscr P_y$. 
	\begin{thm}\label{P4.11}
		The subobjects $\{\mathscr{P}_x \in EM_{\mathscr{U}_{x}}\}_{x \in \mathbb{V}}$ together determine a subobject $\mathscr{P} \subseteq \mathscr{M}$ in $Mod^{tr}-\mathscr U$.	 
	\end{thm}
	\begin{proof}
		Let $\phi:z\longrightarrow y$ be any edge in $\mathbb E$. Since $\iota_z:\mathscr P_{z}\hookrightarrow\mathscr M_z$ is a monomorphism in $EM_{\mathscr U_z}$ and $\hat{\phi}$ is a right adjoint, $\hat{\phi}(\iota_z)$ is a monomorphism in $EM_{\mathscr U_y}.$ To prove that $\mathscr P\subseteq \mathscr M$ in $Mod^{tr}-\mathscr U$, it is enough to show that the morphism $_{\phi}\mathscr M:\mathscr M_y\longrightarrow \hat{\phi}\mathscr M_z$ restricts to a morphism $_{\phi}\mathscr P:\mathscr P_y\longrightarrow \hat{\phi}\mathscr P_z$ such that $_{\phi}\mathscr M\circ \iota_y = \hat{\phi}(\iota_z)\circ~ _{\phi}\mathscr P$. From \cite[Lemma $3.2$]{Ban}, it is sufficient to show that for any morphism $\tau: \mathscr U_yG\longrightarrow \mathscr P_y$, there exists a morphism $\tau':\mathscr U_yG\longrightarrow \hat{\phi}\mathscr P_z$ such that $\hat{\phi}(\iota_z)\circ\tau'=_{\phi}\mathscr M\circ \iota_y\circ\tau.$ From (\ref{eq4.19}), there exists an epimorphism 
		\begin{equation*}
			\underset{\psi\in \mathbb Q(y,x)}{\bigoplus}\zeta_{\psi}':\underset{\psi\in \mathbb Q(y,x)}{\bigoplus}\psi_* \mathscr U_xG\longrightarrow \mathscr P_y\notag
		\end{equation*}
		in $EM_{\mathscr U_y}.$ Since $\mathscr U_yG$ is projective in $EM_{\mathscr U_y},$ the morphism $\tau$ can be lifted to a morphism $\tau'':\mathscr U_yG\longrightarrow \underset{\psi\in \mathbb Q(y,x)}{\bigoplus}\psi_* \mathscr U_xG $ such $\tau=\bigg(\underset{\psi\in \mathbb Q(y,x)}{\bigoplus}\zeta_{\psi}'\bigg)\circ\tau''.$
		We know that the composition $\psi_{*} \mathscr U_xG \xrightarrow{\zeta'_{\psi}}\mathscr P_y\xrightarrow{\iota_y}\mathscr M_y$ factors through $\psi_{*}\mathscr M_x$ as $\iota_y\circ \zeta'_{\psi} = { }^{\psi}\mathscr M\circ \psi_{*}\zeta$. Then, on applying $\phi_{*}$ and composing with $^{\phi}\mathscr M$, we get $^{\phi}\mathscr M\circ \phi_{*}(\iota_y)\circ \phi_{*}(\zeta'_{\psi}) = { }^{\phi}\mathscr M\circ \phi_{*}(^{\psi}\mathscr M)\circ \phi_{*}(\psi_{*}\zeta) = { }^{\psi\phi}\mathscr M \circ (\psi\phi)_{*}(\zeta)$. The last term factors through $\iota_z:\mathscr P_{z}\longrightarrow \mathscr M_{z}$. Further, as $(\phi_{*}, \hat{\phi})$ is a pair of adjoint functor, 	
		the composition $_{\phi}\mathscr M\circ \iota_y\circ \zeta_{\psi}'$ factors through $\hat{\phi}(\iota_z):\hat{\phi}\mathscr P_{z}\longrightarrow \hat{\phi}\mathscr M_{z}$. Therefore, the composition $_{\phi}\mathscr M\circ \iota_y\circ \bigg(\underset{\psi\in \mathbb Q(y,x)}{\bigoplus}\zeta_{\psi}'\bigg)$ also factors through $\phi_{*}(\iota_z):\phi_{*}\mathscr P_{z}\longrightarrow \phi_{*}\mathscr M_{z}$ as
		\begin{equation}\label{eq4.20}
			_{\phi}\mathscr M\circ \iota_y\circ \bigg(\underset{\psi\in \mathbb Q(y,x)}{\bigoplus}\zeta_{\psi}'\bigg) = \hat{\phi}(\iota_z)\circ \xi
		\end{equation}
		for some $\xi :\underset{\psi\in \mathbb Q(y,x)}{\bigoplus}\psi_* \mathscr U_xG \longrightarrow \hat{\phi}\mathscr P_z$ in $EM_{\mathscr U_z}.$ Now on composing (\ref{eq4.20}) with $\tau''$ we obtain 
		\begin{equation}
			_{\phi}\mathscr M\circ \iota_y\circ \bigg(\underset{\psi\in \mathbb Q(y,x)}{\bigoplus}\zeta_{\psi}'\bigg)  \circ \tau''= _{\phi}\mathscr M\circ \iota_y\circ \tau = \hat{\phi}(\iota_z)\circ \xi \circ \tau''.\notag
		\end{equation}
		Finally, we define $\tau' = \xi \circ \tau'' : \mathscr U_yG\longrightarrow \hat{\phi}\mathscr P_z.$ This completes the proof.
	\end{proof}
	We now recall the regular cardinal $\gamma\geq sup\{\mbox{$|Mor(\mathbb Q)|$, $\kappa(G)$, $||\mathscr U_yG||^G$, $y\in Ob(\mathbb Q)$} \}$ as fixed  in (\ref{crd4rf}).
	\begin{lem}\label{L4.12}
		 We have  $|el_{G}(\mathscr P)|\leq \gamma^{\kappa(G)}.$
	\end{lem}
	\begin{proof}
		By \eqref{eq4.19}, $\mathscr P_y$ is an epimorphic image of $\underset{\psi\in \mathbb Q(y,x)}{\bigoplus}\psi_{*}(\mathscr U_{x}G)$ in $EM_{\mathscr U_{y}}$.  Since 
		$\mathscr U_yG$ is projective in $EM_{\mathscr U_y}$, we have
		\begin{equation}
			|EM_{\mathscr U_y}(\mathscr U_yG, \mathscr P_y)|\leq  \left\vert EM_{\mathscr U_y}\left(\mathscr U_yG, \underset{\psi\in \mathbb Q(y,x)}{\bigoplus}\psi_{*}(\mathscr U_xG)\right) \right\vert=  \left\vert\mathcal C\left(G, \underset{\psi\in \mathbb Q(y,x)}{\bigoplus}\psi_{*}(\mathscr U_xG)\right) \right\vert~. \notag
		\end{equation}
		Now, for each edge $\psi:y\longrightarrow x$ in $\mathbb Q$, we note that $\psi_{*}(\mathscr U_xG)=\mathscr U_xG$ as objects of $\mathcal C.$ Finally from Lemma \ref{L2.1}, we obtain 
		\begin{equation}
			\left\vert\mathcal C\left(G, \underset{\psi\in \mathbb Q(y,x)}{\bigoplus}\psi_{*}(\mathscr U_xG)\right) \right\vert =  \left\vert\mathcal C\left(G, \underset{\psi\in \mathbb Q(y,x)}{\bigoplus}\mathscr U_xG\right) \right\vert\leq  \gamma^{\kappa(G)}. 
		\end{equation}
		The result now follows from the definition in \eqref{eq4.18}.
	\end{proof}
	\begin{Thm}\label{T4.13}
 $Mod^{tr}-\mathscr U$ is a  Grothendieck category. 
	\end{Thm}
	\begin{proof}
		It is clear that $Mod^{tr}-\mathscr U$ is a cocomplete abelian category and filtered colimits are exact in  $Mod^{tr}-\mathscr U$. Now, consider an object $\mathscr M$ in  $Mod^{tr}-\mathscr U$. Then, by Proposition \ref{P4.11}, for any element $\zeta\in el_{G}(\mathscr M),$ there exists a subobject $\mathscr P\subseteq \mathscr M$ in $Mod^{tr}-\mathscr U$ such that $\zeta\in el_{G}(\mathscr P)$. Further, from Lemma \ref{L4.12}, we know that $|el_{G}(\mathscr P)|\leq \gamma^{\kappa(G)}.$ As in the proof of 
		Theorem \ref{Th4.7b}, we now see that  $\{\mathscr P'\in Mod^{tr}-\mathscr U~|~ |el_{G}(\mathscr P')|\leq \gamma^{\kappa(G)}\}$ is a set of generators for $Mod^{tr}-\mathscr U$.
	\end{proof}

	\section{Projective generators for modules over a monad quiver}
	In this section, we assume that the quiver $\mathbb Q=(\mathbb V,\mathbb E)$ is a partially ordered set. We will give conditions for $Mod^{cs}-\mathscr U$ and $Mod^{tr}-\mathscr U$ to have projective generators. 
	
	\subsection{Projective generators for cis-modules}
	We suppose that the functor $\mathscr U:\mathbb Q\longrightarrow Mnd(\mathcal C)$ takes values in monads that are exact and preserve colimits.  We begin by constructing a pair of adjoint functors $ex_x^{cs}: EM_{\mathscr U_x} \longrightarrow Mod^{cs}-\mathscr U$ and $ev_{x}^{cs}:Mod^{cs}-\mathscr U\longrightarrow EM_{\mathscr U_x}$ for each $x\in\mathbb V.$ 
	\begin{thm}\label{P5.1}
		Let $x\in\mathbb V$. Then,
		
		\smallskip
		(1) There is a functor $ex_x^{cs}: EM_{\mathscr U_x} \longrightarrow Mod^{cs}-\mathscr U$ defined by setting for each $M\in EM_{\mathscr U_x}$ and $y\in \mathbb V$:
		\begin{equation}
			ex_{x}^{cs}(M)_{y}=\left\{\begin{array}{ll} \psi^* M & \mbox{if $\psi\in\mathbb Q(x,y)$} \\
				0 & \mbox{if $\mathbb Q(x,y) = \emptyset$}\\
			\end{array}\right.
		\end{equation}
		
		\smallskip
		(2) The evaluation $ev_{x}^{cs}:Mod^{cs}-\mathscr U\longrightarrow EM_{\mathscr U_x}$ that takes $\mathscr M\longrightarrow \mathscr M_x$ is an exact functor.
		
		\smallskip
		(3) $(ex_{x}^{cs},ev_{x}^{cs})$ is a pair of adjoint functors.
	\end{thm}
	\begin{proof}
		
		\smallskip
		(1) Clearly, $ex_{x}^{cs}(M)_{y} \in EM_{\mathscr U_{y}}$. Let $\phi:y\longrightarrow y'$ be an edge in $\mathbb Q$. If $x\nleq y$, then $0=ex_{x}^{cs}(M)^\phi:0=\phi^*ex_{x}^{cs}(M)_{y}\longrightarrow ex_{x}^{cs}(M)_{y'}$ in $EM_{\mathscr U_{y'}}$. Otherwise, if there is $\psi:x\longrightarrow y$ and $\rho:x\longrightarrow y'$, then since $\phi\circ\psi=\rho$, we have
		\begin{equation}
			id=ex_{x}^{cs}(M)^\phi:\phi^*ex_{x}^{cs}(M)_{y}=\phi^*\psi^*M\longrightarrow \rho^*M=ex_{x}^{cs}(M)_{y'}
		\end{equation}
		in $EM_{\mathscr U_{y'}}$. Therefore, for each pair of composable morphisms $\phi,\varphi$ in $\mathbb Q$, we have $ex_{x}^{cs}(M)^{\varphi\phi} = ex_{x}^{cs}(M)^{\varphi}\circ\varphi^{*}(ex_{x}^{cs}(M)^{\phi}).$
		
		\smallskip
		(2) Clearly, $ev_{x}^{cs}$ is a functor. Further, since finite limits and finite colimits in $Mod^{cs}-\mathscr U$ are computed pointwise, $ev_{x}^{cs}$ is exact. 
		
		\smallskip
		(3) Given $M\in EM_{\mathscr U_{x}}$ and $\mathscr{P}\in Mod^{cs}-\mathscr U$, we will show that $Mod^{cs}-\mathscr U(ex_{x}^{cs}(M),\mathscr P)\cong EM_{\mathscr U_x}(M,ev_{x}^{cs}(\mathscr P)).$ We start with a morphism $f:M\longrightarrow \mathscr P_{x}$ in $EM_{\mathscr U_x}.$ Then we  define $\xi^f:ex_{x}^{cs}(M)\longrightarrow \mathscr P$ by setting for each $y\in\mathbb Q$:
		\begin{equation}
			\xi_{y}^f:ex_{x}^{cs}(M)_y=\psi^*M\xrightarrow{\psi^*f}\psi^*\mathscr P_{x}\xrightarrow{\mathscr P^{\psi}} \mathscr P_y
		\end{equation}
		whenever $x\leq y$ and $\psi\in\mathbb Q(x,y)$ and $\xi_{y}^f=0$ otherwise. Now for an edge $\phi:y\longrightarrow y'$ in $\mathbb Q$, we will show that $\mathscr P^\phi \circ \phi^*\xi_{y}^f=\xi_{y'}^f\circ ex_{x}^{cs}(M)^\phi$. If $x\nleq y$, then $ex_{x}^{cs}(M)_y=0$ and the equality holds. Otherwise, consider $\psi\in\mathbb Q(x,y)$ and $\rho=\phi\circ\psi:x\longrightarrow y'\in \mathbb Q(x,y')$. Then, we have the following commutative diagram
		\begin{equation}
			\begin{CD}
				\phi^*ex_{x}^{cs}(M)_y=\phi^*\psi^*M @>\phi^*(\mathscr P^\psi\circ\psi^*f) >> \phi^{*}\mathscr{P}_y \\
				@V id VV @VV\mathscr P^{\phi}V \\
				\phi^*\psi^*M=\rho^*M@>\mathscr P^{\rho}\circ\rho^*(f)=\mathscr P^{\phi\psi}\circ\phi^*\psi^*f>> \mathscr{P}_{y'} \\
			\end{CD}
		\end{equation} 
		which shows that $\xi^f$ is a morphism in $Mod^{cs}-\mathscr U$. Conversely, if $\xi:ex_{x}^{cs}(M)\longrightarrow \mathscr P$ is a morphism in $Mod^{cs}-\mathscr U$, then we have an induced morphism $f^\xi:M\longrightarrow\mathscr P_x$ in $EM_{\mathscr U_x}$. It may be verified directly that these two associations are inverse to each other.
	\end{proof}
	\noindent
	We also record here the fact that the functor $ev_{x}^{cs}:Mod^{cs}-\mathscr U\longrightarrow EM_{\mathscr U_x}$ has a right adjoint.
	\begin{thm}\label{P5.2}
		Let $x\in\mathbb V$. Then the functor $ev_{x}^{cs}:Mod^{cs}-\mathscr U\longrightarrow EM_{\mathscr U_x}$ has a right adjoint $coe_x: EM_{\mathscr U_x} \longrightarrow Mod^{cs}-\mathscr U$ given as follows for $M\in  EM_{\mathscr U_x}$ and $y\in\mathbb V$:
		\begin{equation}
			coe_{x}(M)_{y}=\left\{\begin{array}{ll} \psi_* M & \mbox{if $\psi\in\mathbb Q(y,x)$} \\
				0 & \mbox{if~$\mathbb Q(y,x) = \emptyset$}\\
			\end{array}\right.
		\end{equation}
	\end{thm}
	\begin{proof}
		It is clear that $coe_{x}(M)_{y}\in EM_{\mathscr U_y}$ for each $y\in\mathbb V$. Now, consider an edge $\phi:y'\longrightarrow y$. If $y\nleq x$ then $coe_x(M)_\phi=0$. Otherwise, if we have edges $\psi:y\longrightarrow x$ and $\rho:y'\longrightarrow x$, then, since $\psi\circ\phi=\rho$, we get $id=coe_{x}(M)_\phi:\rho_*(M)\longrightarrow\phi_*\psi_*(M)$. It follows that $coe_{x}(M)\in Mod^{cs}-\mathscr U$. The adjunction $(ev_x,coe_x)$ can now be shown as in the proof of Proposition \ref{P5.1}(3). 
	\end{proof}
	\begin{cor}\label{C5.3}
		Let $x\in\mathbb V$. Then the functor $ex_{x}^{cs}: EM_{\mathscr U_x} \longrightarrow Mod^{cs}-\mathscr U$ preserves projectives.
	\end{cor}
	\begin{proof}
		By Proposition \ref{P5.1}, we know that   $(ex_{x}^{cs},ev_{x}^{cs})$ is an adjoint pair   and that the right adjoint functor $ev_x$ is exact. It therefore follows that the left adjoint $ex_{x}^{cs}$ preserves projective objects.
	\end{proof}
	\begin{Thm}\label{Th5.4}
		The category $Mod^{cs}-\mathscr U$ has a set of projective generators.
	\end{Thm}
	\begin{proof}
		By the proof of Proposition \ref{L3.1}, we know that for any  $x\in\mathbb V$, $\mathscr U_xG$ is a projective generator in $EM_{\mathscr U_{x}}$. Using Corollary \ref{C5.3}, it now follows that each $ex_{x}^{cs}(\mathscr U_xG)$ is projective in $Mod^{cs}-\mathscr U$. We will now show that the family
		\begin{equation}
			\mathcal{G}=\{ex_{x}^{cs}(\mathscr U_xG)~\vert~x\in\mathbb{V}\}
		\end{equation}
		is a set of generators for $Mod^{cs}-\mathscr U$. We start with a monomorphism $\iota:\mathscr N\hookrightarrow \mathscr M$ in $Mod^{cs}-\mathscr U$ such that $\mathscr N\subsetneq \mathscr M$. We know that kernels and cokernels in $Mod^{cs}-\mathscr U$ are computed pointwise. Hence, there exists some $x\in\mathbb V$ such that $\iota_x:\mathscr N_x\hookrightarrow\mathscr M_x$ is a monomorphism with $\mathscr N_x\subsetneq\mathscr M_x$. Since $\mathscr U_xG$  is a generator of $EM_{\mathscr U_x}$, we may choose a morphism $f:\mathscr U_xG\longrightarrow \mathscr M_x$ in $EM_{\mathscr U_x}$ which does not factor through $\iota_x:\mathscr N_x\hookrightarrow\mathscr M_x$. Since
		$(ex_{x}^{cs},ev_{x}^{cs})$ is an adjoint pair, we obtain a morphism $\xi^f:ex_{x}^{cs}(\mathscr U_xG)\longrightarrow \mathscr M$ such that $\xi^f$ does not factor through $\iota:\mathscr N\longrightarrow\mathscr M$. It now follows from \cite[\S1.9]{Gro} that $\mathcal G$ is a set of generators for $Mod^{cs}-\mathscr U$. 
	\end{proof}
	\subsection{Projective generators for trans-modules}

We continue with the functor $\mathscr U:\mathbb Q\longrightarrow Mnd(\mathcal C)$ taking values in monads that are exact and preserve colimits. We suppose additionally that for each $x\in\mathbb V$, the right adjoint $\mathscr V_x$ of $\mathscr U_x$ is exact.
	\begin{thm}\label{P5.5}
		Let $x\in\mathbb V$. Then,
		
		\smallskip
		(1) There exists a functor $ex_x^{tr}: EM_{\mathscr U_x} \longrightarrow Mod^{tr}-\mathscr U$ defined by setting for each $M\in EM_{\mathscr U_x}$ and $y\in \mathbb V:$
		\begin{equation}
			ex_{x}^{tr}(M)_{y}=\left\{\begin{array}{ll} \psi_{*} M & \mbox{if $\psi\in\mathbb Q(y,x)$} \\
				0 & \mbox{$otherwise$}\\
			\end{array}\right.
		\end{equation}
		\smallskip
		(2) The evaluation $ev_{x}^{tr}:Mod^{tr}-\mathscr U\longrightarrow EM_{\mathscr U_x}$, $\mathscr M\longrightarrow \mathscr M_x$ is an exact functor.
		
		\smallskip
		(3) $(ex_{x}^{tr},ev_{x}^{tr})$ is a pair of adjoint functors.
	\end{thm}
	\begin{proof}
	This is proved in a manner similar to Proposition \ref{P5.1}.

	\end{proof}
	\begin{thm}\label{P5.6w}
		Let $x\in\mathbb V$. Then the functor $ev_{x}^{tr}:Mod^{tr}-\mathscr U\longrightarrow EM_{\mathscr U_x}$ has a right adjoint $coe_x^{tr}: EM_{\mathscr U_x} \longrightarrow Mod^{tr}-\mathscr U$ given by:
		\begin{equation}
			coe_{x}^{tr}(M)_{y}=\left\{\begin{array}{ll} \hat{\psi} M & \mbox{if $\psi\in \mathbb Q(x,y)$} \\
				0 & \mbox{$otherwise$}\\
			\end{array}\right.
		\end{equation}
		for all $M\in  EM_{\mathscr U_x}$ and $y\in\mathbb V$. Additionally, the functor $ex_{x}^{tr}: EM_{\mathscr U_x} \longrightarrow Mod^{tr}-\mathscr U$ preserves projectives.
	\end{thm}
	\begin{proof} The proof of this is similar to that of Proposition \ref{P5.2} and Corollary \ref{C5.3}. 
	\end{proof}
	\begin{Thm}\label{Th5.8}
		The category $Mod^{tr}-\mathscr U$ has a set of projective generators.
	\end{Thm}
	\begin{proof}
		For each  $x\in\mathbb V$, we know that $\mathscr U_xG$ is a projective generator in $EM_{\mathscr U_{x}}$. Using Proposition \ref{P5.6w}, it now follows that each $ex_{x}^{tr}(\mathscr U_xG)$ is projective in $Mod^{tr}-\mathscr U$. As in the proof of Theorem \ref{Th5.4}, we will show that the family
		\begin{equation}
			\mathcal{G}=\{ex_{x}^{tr}(\mathscr U_xG)~\vert~x\in\mathbb{V}\}
		\end{equation}
		is a set of generators for $Mod^{tr}-\mathscr U$. We consider a monomorphism $\iota:\mathscr N\hookrightarrow \mathscr M$ in $Mod^{tr}-\mathscr U$ such that $\mathscr N\subsetneq \mathscr M$. Then, there exists some $x\in\mathbb V$ such that $\iota_x:\mathscr N_x\hookrightarrow\mathscr M_x$ is a monomorphism with $\mathscr N_x\subsetneq\mathscr M_x$. As $\mathscr U_xG$ is a generator for $EM_{\mathscr U_x}$, there exists a morphism $f:\mathscr U_xG\longrightarrow \mathscr M_x$ in $EM_{\mathscr U_x}$ that does not factor through $\iota_x:\mathscr N_x\hookrightarrow\mathscr M_x$. Since
		$(ex_{x}^{tr},ev_{x}^{tr})$ is an adjoint pair, there exists a morphism $\xi^f:ex_{x}^{tr}(\mathscr U_xG)\longrightarrow \mathscr M$ such that $\xi^f$ does not factor through $\iota:\mathscr N\longrightarrow\mathscr M$. The result now follows from \cite[\S1.9]{Gro}. 
	\end{proof}
	\section{Cartesian modules over a monad quiver}

We continue with $\mathbb Q$ being a poset and the functor $\mathscr U:\mathbb Q\longrightarrow Mnd(\mathcal C)$ taking values in monads that are exact and preserve colimits.  
	\subsection{Cartesian cis-modules over a monad quiver}
	
	Suppose additionally that $\mathscr U:\mathbb Q\longrightarrow Mnd(\mathcal C)$ is flat, i.e., for any edge $\psi:x\longrightarrow y$ in $\mathbb Q$, the functor  $\psi^*:EM_{\mathscr U_{x}}\longrightarrow EM_{\mathscr U_{y}}$  is exact.  Let $\xi : \mathscr M \longrightarrow \mathscr M'$ be a morphism in $Mod^{cs}_c-\mathscr U$. It follows that $Ker(\xi)$, $Coker(\xi)\in Mod^{cs}_c-\mathscr U$, where  $Ker(\xi)_x=Ker(\xi_x)$ and
	$Coker(\xi)_x=Coker(\xi_x)$ for each $x \in \mathbb V$.  We see therefore that $Mod^{cs}_c-\mathscr U$ is an abelian category. 
	
	\smallskip
	We continue with $\gamma\geq sup\{\mbox{$Mor(\mathbb Q)$, $\kappa(G)$, $||\mathscr U_yG||^G$, $y\in Ob(\mathbb Q)$} \}$ as in \eqref{crd4rf}. For an endofunctor $U:\mathcal C\longrightarrow \mathcal C$ as in Theorem \ref{T2.2}, we recall that we have $\lambda^U$ such that $||UM||^G\leq \lambda^U\times  (||M||^G)^{\kappa(G)}$ for any object $M\in \mathcal C$. In this section, we only consider monads which are exact and preserve colimits.
	
	\begin{lem}\label{L6.1}
		Let $\phi:(U,\theta,\eta)\longrightarrow (U',\theta',\eta')$  be a flat morphism of monads over $\mathcal C$. Let $\gamma'\geq \{\mbox{$Mor(\mathbb Q)$, $\kappa(G)$, $||UG||^G$, $||U'G||^G$} \}$ and  $\alpha\geq \gamma',\lambda^U$. Let $(M,f_M)\in EM_U$ and let $X\subseteq el_G(\phi^*M)$ be a subset such that $|X|\leq \alpha$.  Then, there exists a subobject $N\subseteq M$ in $EM_U$ such that
		$||N||^G\leq  \alpha^{\gamma'}$ and $X\subseteq \phi^*N$. 
	\end{lem}
	
	\begin{proof}
		We choose $x\in X\subseteq \mathcal C(G,\phi^*M)$ and consider the corresponding morphism $\hat{x}\in EM_{U'}(U'G,\phi^*M)$. Since $(UG,\theta G)$ is a generator for $EM_U$, we can choose an epimorphism $p:(UG)^{(I)}\longrightarrow M$ in $EM_U$ from a direct sum of copies of $UG$. As noted in \eqref{4.4dk}, we know that $\phi^*(UG)=U'G$.  Since $\phi^*$ is a left adjoint, we have an induced epimorphism $\phi^*(p):(U'G)^{(I)}=\phi^*((UG)^{(I)})\longrightarrow
		\phi^*M$. 
		
		\smallskip
		Since $U'G$ is projective in $EM_{U'}$, we  may now lift $\hat{x}:U'G\longrightarrow \phi^*M$ over $\phi^*(p)$ to obtain $\zeta_x: U'G\longrightarrow (U'G)^{(I)}=\phi^*((UG)^{(I)})$ such that $\hat x=\phi^*(p)\circ \zeta_x$. Since $\gamma'\geq\kappa(G)$, we know by  Lemma \ref{L3.2} that $U'G$ is $\gamma'$-presentable in $EM_{U'}$. Accordingly, we may find a subset $J_x\subseteq I$ with $|J_x|<\gamma'$ such that $\zeta_x$ factors through the direct sum $(U'G)^{(J_x)}$. We now have a diagram in $EM_{U'}$.
		\begin{equation}\label{61cd}
			\begin{tikzcd}
				U'G \arrow{ddrr}{\zeta_x} \arrow{rr}{} \arrow{dd}{\hat{x}} & & (U'G)^{(J_x)}=\phi^*((UG)^{(J_x)})\arrow{dd}{} \\
				& &\\
				\phi^*M  & & \arrow{ll}{\phi^*(p)} (U'G)^{(I)}=\phi^*((UG)^{(I)}) \\
			\end{tikzcd}
		\end{equation} From \eqref{61cd}, we have a morphism $\xi_x: (UG)^{(J_x)}\longrightarrow (UG)^{(I)}\longrightarrow M$ such that $\hat{x}$ factors through $\phi^*(\xi_x)$. In $EM_{U}$, we now set
		\begin{equation}\label{62sq}
			N:=Im\left(\xi:= \underset{x\in X}{\bigoplus} \xi_x:\underset{x\in X}{\bigoplus} (UG)^{(J_x)}\longrightarrow M \right)\subseteq M
		\end{equation} By assumption, $\phi^*:EM_{U}\longrightarrow EM_{U'}$ is exact. Additionally, since $\phi^*$ is a left adjoint, we have
		\begin{equation}\label{63sq}
			\phi^*N:=Im\left(\phi^*(\xi)=\underset{x\in X}{\bigoplus} \phi^*\xi_x:\underset{x\in X}{\bigoplus} \phi^*((UG)^{(J_x)})\longrightarrow \phi^*M \right)
		\end{equation} By \eqref{61cd}, we see that each $x\in X$ lies in the image $\phi^*N$. 
		It remains to show that $||N||^G\leq \alpha^{\gamma'}$. By definition, $||UG||^G\leq \lambda^U={(||UG||^G)}^{\kappa(G)}\times \kappa(G)^{\kappa(G)}$. Applying Lemma \ref{L2.1}, we now obtain
		\begin{equation}
			||N||^G \leq || \underset{x\in X}{\bigoplus} (UG)^{(J_x)}||^G \leq (\lambda^U)^{\kappa(G)}\times (\alpha\times \gamma')^{\kappa(G)}\leq (\lambda^U)^{\gamma'}\times {\gamma'}^{\gamma'}\times \alpha^{\gamma'}= \alpha^{\gamma'}
		\end{equation} where the last equality follows from the fact that $\alpha\geq \gamma,\lambda^U$. 
	\end{proof}
	
	\begin{lem}\label{L6.2}
		Let $\phi:(U,\theta,\eta)\longrightarrow (U',\theta',\eta')$  be a flat morphism of monads over $\mathcal C$ and let $(M,f_M)\in EM_U$. Let $\gamma'\geq \{\mbox{$Mor(\mathbb Q)$, $\kappa(G)$, $||UG||^G$, $||U'G||^G$} \}$ and  $\alpha\geq \gamma',\lambda^U,\lambda^{U'}$. Let $X\subseteq el_G(M)$ and $Y\subseteq el_G(\phi^*M)$ be subsets such that $|X|, |Y|\leq \alpha^{\gamma'}$.  Then, there exists a subobject $N\subseteq M$ in $EM_U$ such that
		
		\smallskip
		(1) $X\subseteq el_G(N)$ and $Y\subseteq el_G(\phi^*N)$. 
		
		\smallskip
		(2)  $||N||^G\leq  \alpha^{\gamma'}$ and $||\phi^*N||^G\leq \alpha^{\gamma'}$. 
	\end{lem}
	
	\begin{proof} Applying Lemma \ref{L6.1} to the morphism $\phi:(U,\theta,\eta)\longrightarrow (U',\theta',\eta')$, we obtain $N_1\subseteq M$ in $EM_U$ with 
		$||N_1||^G\leq (\alpha^{\gamma'})^{\gamma'}=\alpha^{\gamma'}$  such that $Y\subseteq el_G(\phi^*N_1)$. Applying Lemma \ref{L6.1} again, this time to the identity morphism
		on $(U,\theta,\eta)$, we obtain $N_2\subseteq M$ in $EM_U$ such that $||N_2||^G\leq (\alpha^{\gamma'})^{\gamma'}=\alpha^{\gamma'}$  such that $X\subseteq el_G(N_2)$.  We set $N:=N_1+N_2\subseteq M$ in $EM_U$. We note that
		\begin{equation}
			X\subseteq el_G(N_2)\subseteq el_G(N)\qquad Y\subseteq el_G(\phi^*N_1) \subseteq el_G(\phi^*N)
		\end{equation} where the second relation follows from the fact that $\phi^*$ is exact, which gives $\phi^*N_1\subseteq \phi^*N$ in $EM_{U'}$. Since $N=N_1+N_2$, we have an epimorphism $N_1\oplus N_2\twoheadrightarrow N$. Accordingly, we have
		\begin{equation}\label{6.6ty}
			||N||^G\leq ||N_1\oplus N_2||^G\leq \alpha^{\gamma'}
		\end{equation} It remains to show that $||\phi^*N||^G\leq \alpha^{\gamma'}$.  For this, we note that by the definition in \eqref{4.3qi}, we have
		\begin{equation}\label{6.7ty}
			\phi^*(N):=Coeq\left(U'UN\doublerightarrow{\qquad\qquad}{}U'N\right)
		\end{equation} In particular, this means that there is an epimorphism $U'N\twoheadrightarrow \phi^*N$ in $\mathcal C$. By Theorem \ref{T2.2}, we know that $||U'N||^G\leq \lambda^{U'}\times (||N||^G)^{\kappa(G)}$. Accordingly, we have
		\begin{equation}
			||\phi^*N||^G\leq ||U'N||^G\leq \lambda^{U'}\times (||N||^G)^{\kappa(G)}\leq \alpha^{\gamma'}
		\end{equation}

	\end{proof}
	
	We will now show that $Mod^{cs}_c-\mathscr U$ has a generator.  We fix an infinite cardinal $\alpha$ such that
	\begin{equation}\label{sup6}
		\alpha\geq sup \{\mbox{$\gamma$, $\lambda^{\mathscr U_x}$, $x\in \mathbb V$}\}
	\end{equation}
	Let $\mathscr M\in Mod^{cs}_c-\mathscr U$ and take some $\zeta\in el_G(\mathscr M)$, given by $\zeta: \mathscr U_xG\longrightarrow \mathscr M_x$ for some $x\in \mathbb V$. Corresponding to $\zeta$, we consider as in the proof of Theorem \ref{Th4.7b} the subobject $\mathscr  P \subseteq \mathscr M$ in $Mod^{cs}-\mathscr U$ such that $\zeta\in el_G(\mathscr P)$ and $|el_G(\mathscr P)|\leq  \gamma^{\kappa(G)}\leq \alpha^\gamma$. We now choose a well ordering of the set $Mor(\mathbb Q)$ and consider the induced lexicographic order on $\mathbb N \times Mor(\mathbb Q)$. We proceed by induction on $\mathbb N \times Mor(\mathbb Q)$ to construct a family of subobjects $\{\mathscr N(n,\phi) : n\in \mathbb N, \phi\in Mor(\mathbb Q)\}$ of $\mathscr M$ in $Mod^{cs}-\mathscr U$ satisfying the following conditions.
	
	\smallskip
	(1) If $\phi_0$ is the least element of $Mor(\mathbb Q)$, then $\zeta\in el_G(\mathscr N(1,\phi_0))$.
	
	\smallskip
	(2) For any $(n,\phi)\leq (m,\psi)$ in $\mathbb N \times Mor(\mathbb Q)$, we have $\mathscr N(n,\phi) \subseteq \mathscr N(m,\psi)$
	
	\smallskip
	(3) For each $(n,\phi:y \longrightarrow z)$ in $\mathbb N \times Mor(\mathbb Q)$, the morphism $\mathscr N(n,\phi)^{\phi}:\phi^*\mathscr N(n,\phi)_y \longrightarrow \mathscr N(n,\phi)_z$ is an isomorphism in $EM_{\mathscr U_z}$.
	
	\smallskip
	(4) $|el_G(\mathscr N(n,\phi))|\leq \alpha^\gamma$.
	
	For  $(n,\phi:y \longrightarrow z)$ in $\mathbb N \times Mor(\mathbb Q)$, we begin the transfinite induction argument by setting
	\begin{equation}
		A_0^0(w):= 
		\begin{cases} 
			el_G(\mathscr P_w),\quad & \text{if}~~(n,\phi) = (1,\phi_0) \\
			\underset{(m,\psi)< (n,\phi)}\bigcup el_G(\mathscr N(m,\psi)_w), \quad &~~ \mbox{otherwise}
		\end{cases}
	\end{equation}
	for each $w \in \mathbb V$. Since each $A_0^0(w)\subseteq el_G(\mathscr M_w)$, $|A_0^0(w)|\leq \alpha^\gamma$, and $\mathscr M$  is cartesian, we use Lemma \ref{L6.2} to obtain a subobject $A_1^0(y) \subseteq \mathscr M_y$ in $EM_{\mathscr U_y}$ such that
	\begin{equation}\label{611dt}
		||A_1^0(y)||^G\leq \alpha^\gamma \quad ||\phi^*A_1^0(y)||^G\leq \alpha^\gamma\quad A^0_0(y) \subseteq el_G(A_1^0(y)) \quad A_0^0(z) \subseteq el_{G}(\phi^*A_1^0(y))
	\end{equation} 
	We now set $A_1^0(z) = \phi^*A_1^0(y)$ and set for each $w\in \mathbb V$:
	\begin{equation} \label{612dx}
		B_1^0(w)=\left\{\begin{array}{ll} el_G(A^0_1(w))  & \mbox{if $w=y,z$} \\
			A^0_0(w) & \mbox{otherwise}\\
		\end{array}\right.
	\end{equation} From \ref{611dt} and \eqref{612dx} it follows that for each $w \in \mathbb V$, $A^0_0(w) \subseteq B^0_1(w)$ and $|B^0_1(w)|\leq \alpha^\gamma$.
	\begin{lem}\label{L6.3}
		Let $X\subseteq el_G(\mathscr M)$ with $|X|\leq\alpha^\gamma$. Then there exists a subobject $\mathscr D\hookrightarrow\mathscr M$ in $Mod^{cs}-\mathscr U$ such that $X \subseteq el_G(\mathscr D)$ and  $|el_G(\mathscr D)|\leq\alpha^\gamma$.
	\end{lem}
	\begin{proof}
		Let $\zeta\in X\subseteq el_G(\mathscr M)$. Then, using Theorem \ref{Th4.7b}, we choose a subobject $\mathscr D_{\zeta}\hookrightarrow\mathscr M$ such that $\zeta\in el_G(\mathscr D_{\zeta})$ and  $|el_G(\mathscr D_\zeta)|\leq\gamma^{\kappa(G)}\leq \alpha^\gamma$. Now, we set $\mathscr D := \underset{\zeta\in X}\sum\mathscr D_\zeta$. Clearly, $\mathscr D$ is a quotient of $\underset{\zeta\in X}\bigoplus\mathscr D_\zeta$ and $X \subseteq el_G(\mathscr D)$. Further, using Lemma \ref{L2.1} and the definition in \eqref{el45t}, we get
		\begin{equation}
			|el_G(\mathscr D)| \leq \left\vert el_G \left(\underset{\zeta\in X}\bigoplus\mathscr D_\zeta\right)\right\vert \leq  \underset{y\in\mathbb V}\sum\left\vert EM_{\mathscr U_y}\left(\mathscr U_yG,\underset{\zeta\in X}\bigoplus\mathscr{D}_{\zeta_y}\right)\right\vert=\underset{y\in\mathbb V}\sum\left\vert\mathcal C\left(G,\underset{\zeta\in X}\bigoplus\mathscr{D}_{\zeta_y}\right)\right\vert\leq \alpha^\gamma
		\end{equation}
	\end{proof}
	Now using Lemma \ref{L6.3}, we choose a subobject $\mathscr D^0(n,\phi)\hookrightarrow\mathscr M$ in $Mod^{cs}-\mathscr U$ such that $\underset{w\in\mathbb V}\bigcup B^0_1(w) \subseteq el_G(\mathscr D^0(n,\phi))$ and $|el_G(\mathscr D^0(n,\phi))|\leq \alpha^\gamma$. In particular, for each $w\in\mathbb V$, $B^0_1(w)\subseteq el_G(\mathscr D^0(n,\phi)_w)$. 
	
	\smallskip
	We now iterate this construction. Suppose that for every $r\leq s$ we have constructed a subobject $\mathscr D^r(n,\phi)\hookrightarrow\mathscr M$ in
	$Mod^{cs}-\mathscr U$  such that  $\underset{w\in\mathbb V}\bigcup B^r_1(w) \subseteq el_G(\mathscr D^r(n,\phi))$ and $|el_G(\mathscr D^r(n,\phi))|\leq \alpha^\gamma$.  Then, for each $w\in\mathbb V$, we set $A_0^{s+1}(w):= el_G(\mathscr D^s(n,\phi)_w)$. Again using Lemma \ref{L6.2}, we get $A_1^{s+1}(y)\subseteq\mathscr M_y$ in $EM_{\mathscr U_y}$ such that
	\begin{equation}
		||A_1^{s+1}(y)||^G\leq \alpha^\gamma \quad ||\phi^*A_1^{s+1}(y)||^G\leq \alpha^\gamma \quad A^{s+1}_0(y) \subseteq el_G(A_1^{s+1}(y)) \quad A_0^{s+1}(z) \subseteq el_{G}(\phi^*A_1^{s+1}(y))
	\end{equation} 
	We now set $A_1^{s+1}(z) = \phi^*A_1^{s+1}(y)$. For $w\in \mathbb V$, we set $B_1^{s+1}(w)=el_G(A^{s+1}_1(w))$ if $w=y,z,$  and 
	$B^{s+1}_1(w)=A_0^{s+1}(w)= el_G(\mathscr D^s(n,\phi)_w)$ otherwise. It follows that for each $w \in \mathbb V$, $A_0^{s+1}(w) \subseteq B^{s+1}_1(w)$ and $|B_1^{s+1}(w)|\leq \alpha^\gamma$. Using Lemma \ref{L6.3}, we now choose $\mathscr D^{s+1}(n,\phi)\hookrightarrow \mathscr M$ such that $\underset{w\in\mathbb V}\bigcup B^{s+1}_1(w) \subseteq el_G(\mathscr D^{s+1}(n,\phi))$ and $|el_G(\mathscr D^{s+1}(n,\phi))|\leq \alpha^\gamma$.   In particular, for each $w\in\mathbb V$, $B^{s+1}_1(w)\subseteq el_G(\mathscr D^{s+1}(n,\phi)_w)$. 
	We note that we have constructed an ascending chain 
	\begin{equation}\label{615th}
		\mathscr D^{0}(n,\phi)\leq \mathscr D^{1}(n,\phi)\leq\ldots\leq\mathscr D^{s}(n,\phi)\leq \ldots
	\end{equation}
	of subobjects of $\mathscr M$ in $Mod^{cs}-\mathscr U$.
	Finally, we define 
	\begin{equation}\label{616th}
		\mathscr N(n,\phi):= \underset{s\geq0}\varinjlim\textrm{ } \mathscr D^{s}(n,\phi)
	\end{equation}
	in $Mod^{cs}-\mathscr U$. Since each $|el_G(\mathscr D^{s}(n,\phi))| \leq \alpha^\gamma$, we have $|el_G(\mathscr N(n,\phi))| \leq \alpha^\gamma.$ Clearly, the family $\{\mathscr N(n,\phi)~|~(n,\phi)\in\mathbb N\times Mor(\mathbb Q)\}$ satisfies the conditions $(1)$ and $(2)$. For $(4)$, we note that $\mathscr N(n,\phi)$ is a quotient of $\underset{s\geq 0}
	{\bigoplus}\mathscr D^{s}(n,\phi)$ and that each $||\underset{s\geq 0}
	{\bigoplus}\mathscr D^{s}(n,\phi)_w||^G\leq (\alpha^\gamma)^{\kappa(G)}=\alpha^\gamma$ by Lemma \ref{L2.1}.
	For $(3)$, we note that $\mathscr N(n,\phi)_y$ can be expressed as the filtered union 
	\begin{equation}
		A_1^{0}(y)\hookrightarrow\mathscr D^{0}(n,\phi)_y\hookrightarrow A_1^{1}(y)\hookrightarrow\mathscr D^{1}(n,\phi)_y\hookrightarrow\cdots\hookrightarrow A_1^{s}(y)\hookrightarrow\mathscr D^{s}(n,\phi)_y\hookrightarrow \cdots
	\end{equation}
	of objects in $EM_{\mathscr U_y}$. Since $\phi^*$ is exact and a left adjoint, it preserves monomorphisms and filtered colimits. Hence we can also express $\phi^*\mathscr N(n,\phi)_y$ as a filtered union
	\begin{equation}
		\phi^*A_1^{0}(y)\hookrightarrow\phi^*\mathscr D^{0}(n,\phi)_y\hookrightarrow \phi^*A_1^{1}(y)\hookrightarrow\phi^*\mathscr D^{1}(n,\phi)_y\hookrightarrow\cdots\hookrightarrow \phi^*A_1^{s}(y)\hookrightarrow\phi^*\mathscr D^{s}(n,\phi)_y\hookrightarrow \cdots
	\end{equation}
	of objects in $EM_{\mathscr U_z}$. Similarly, $\mathscr N(n,\phi)_z$ can be expressed as the filtered union 
	\begin{equation}
		A_1^{0}(z)\hookrightarrow\mathscr D^{0}(n,\phi)_z\hookrightarrow A_1^{1}(z)\hookrightarrow\mathscr D^{1}(n,\phi)_z\hookrightarrow\cdots\hookrightarrow A_1^{s}(z)\hookrightarrow\mathscr D^{s}(n,\phi)_z\hookrightarrow \cdots
	\end{equation}
	of objects in $EM_{\mathscr U_z}$. By definition, we know that $\phi^*A_1^{s}(y) = A_1^{s}(z)$ for each $s\geq 0$. Therefore, we obtain the required isomorphism $\mathscr N(n,\phi)^\phi : \phi^*\mathscr N(n,\phi)_y\longrightarrow\mathscr N(n,\phi)_z$.
	\begin{lem}\label{L6.4}
		Let $\mathscr M$ be a cartesian module over a flat monad quiver $\mathscr U:\mathbb Q\longrightarrow Mnd(\mathscr C)$. Let $\zeta\in el_G(\mathscr M)$. Then there exists a subobject $\mathscr N\subseteq\mathscr M$ in $Mod^{cs}_c-\mathscr U$ such that $\zeta\in el_G(\mathscr N)$ and $|el_G(\mathscr N)| \leq \alpha^\gamma$.
	\end{lem}
	\begin{proof}
		Since $\mathbb N\times Mor(\mathbb Q)$ is filtered, we set 
		\begin{equation}
			\mathscr N = \underset{(n,\phi)\in\mathbb N\times Mor(\mathbb Q)}\bigcup\mathscr N(n,\phi)\subseteq\mathscr M
		\end{equation}
		in $Mod^{cs}-\mathscr U$. Clearly, $\zeta\in el_G(\mathscr N).$ Also, as each $|el_G(\mathscr N(n,\phi))|\leq \alpha^\gamma$, we have $|el_G(\mathscr N)|\leq\alpha^\gamma$. Next, we note that for a fixed morphism $\rho:z\longrightarrow w$ in $\mathbb Q$, the family $\{(m,\rho)~|~m\geq 1\}$ is cofinal in $\mathbb N\times Mor(\mathbb Q)$. Therefore, 
		\begin{equation}
			\mathscr N = \underset{m\geq 1}\varinjlim\mathscr N(m,\rho)
		\end{equation}
		Further, as $\mathscr N(m,\rho)^\rho : \rho^*\mathscr N(m,\rho)_z\longrightarrow\mathscr N(m,\rho)_w$ is an isomorphism, it follows that the filtered colimit $\mathscr N^\rho:\rho^*\mathscr N_z \longrightarrow\mathscr N_w$ is also an isomorphism.
	\end{proof}
	\begin{Thm}\label{Th6.5}
		 $Mod^{cs}_c-\mathscr U$  is a Grothendieck category.
	\end{Thm}
	\begin{proof}
		We already know that $Mod^{cs}_c-\mathscr U$ is an abelian category. Now, since filtered colimits and finite limits of $Mod^{cs}_c-\mathscr U$ are computed in $Mod^{cs}-\mathscr U$, and $\mathscr U:\mathbb Q\longrightarrow Mnd(\mathcal C)$ is flat, it is also clear  $Mod^{cs}_c-\mathscr U$   satisfies the (AB5) condition. Further, from Lemma \ref{L6.4}, we see that any $\mathscr M \in Mod^{cs}_c-\mathscr U$ can be expressed as sum of a family $\{\mathscr N_\zeta~|~\zeta\in el_G(\mathscr M)\}$ of cartesian subobjects where each $|el_G(\mathscr N_\zeta)|\leq\alpha^\gamma$. Therefore, the isomorphism classes of cartesian modules $\mathscr N$ satisfying $|el_G(\mathscr N)|\leq\alpha^\gamma$ give a set of generators for $Mod^{cs}_c-\mathscr U$.
	\end{proof}
	\begin{Thm}\label{Th6.6}
		The  inclusion functor $i: Mod^{cs}_c-\mathscr U \longrightarrow Mod^{cs}-\mathscr U$ has a right adjoint.
	\end{Thm}
	\begin{proof}
		We see that the inclusion functor $i: Mod^{cs}_c-\mathscr U \longrightarrow Mod^{cs}-\mathscr U$ preserves colimits.  Since $Mod^{cs}_c-\mathscr U$ and $Mod^{cs}-\mathscr U$ are Grothendieck categories, it follows (see, for instance, \cite[Proposition 8.3.27]{KS}) that $i$ has a right adjoint.
	\end{proof}
	\subsection{Cartesian trans-modules over a monad quiver}
	
	Let $(U,\theta,\eta)$ and $(U',\theta',\eta')$ be monads on $\mathcal C$ that are exact and preserves colimits. Let $\phi: U\longrightarrow U'$ be a morphism in 
	$Mnd(\mathcal C)$. We know that the induced restriction functor $\phi_{*}: EM_{U'}\longrightarrow EM_{U}$ is exact. We will say that $\phi$ is coflat if the functor $\hat{\phi}: EM_{U}\longrightarrow EM_{U'}$ is exact. We continue with $\gamma\geq sup\{\mbox{$Mor(\mathbb Q)$, $\kappa(G)$, $||\mathscr U_yG||^G$, $y\in Ob(\mathbb Q)$} \}$

	\smallskip
	
	Let $V$ and $V'$ be respectively the right adjoint comonads of $U$ and $U'$. Suppose additionally that  $V$ and $V'$ preserve colimits. Then, it is clear from definition in \eqref{eq4.17} that $\hat{\phi}$ preserves filtered colimits.

	\begin{lem}\label{L6.8}
		Let $\phi: U \longrightarrow U'$ be a coflat morphism of monads on $\mathcal C$. Let $\gamma'\geq \{\mbox{$Mor(\mathbb Q)$, $\kappa(G)$, $||UG||^G$, $||U'G||^G$} \}$ and $\alpha\geq \lambda^U$, $\gamma'$. Let $(M,f_{M})\in EM_{U}$ and let $X\subseteq el_{G}(\hat{\phi}M)$ be a subset such that $|X|\leq \alpha$. Then, there exists a subobject $N\subseteq M$ in $EM_{U}$ such that $X\subseteq el_{G}(\hat{\phi}N)$ and $||N||^{G}\leq \alpha^{\gamma'}$. 
	\end{lem}
	\begin{proof}
		We consider a morphism $x\in \mathcal C(G, \hat{\phi}M)$. Let  $\hat{x}\in EM_{U'}(U'G,\hat{\phi}M)$ be the corresponding morphism in $EM_{U'}$. Since $UG$ is a generator for $EM_{U},$ there exists an epimorphism $g: (UG)^{(I)}\longrightarrow M$ for some index set $I$. Since $\hat{\phi}$ is exact, it follows that $\hat{\phi}(g)$ is an epimorphism in $EM_{U'}$. 
		As $U'G$ is  projective and $\hat{\phi}$ preserves direct sums, the morphism $\hat{x}:U'G\longrightarrow \hat{\phi}M$ factors through a morphism $\zeta:U'G\longrightarrow (\hat{\phi}(UG)^{(I)})$ in $EM_{U'}.$  Since $\gamma'\geq \kappa(G)$, it follows from Lemma \ref{L3.2} that  $U'G$ is $\gamma'$-presentable in $EM_{U'}$. Hence,  there exists a subset $I_x\subseteq I$ with $|I_x|<\gamma'$ such that the following diagram commutes
		\begin{equation}\label{eq6.23}
			\begin{tikzcd}
				(\hat{\phi}(UG)^{(I_x)}) \arrow{dd}{\iota} & & U'G\arrow{dd}{\hat{x}}\arrow{ll}{}\arrow{ddll}{\zeta} \\
				& &\\
				(\hat{\phi}(UG)^{(I)}) \arrow{rr}{\hat{\phi}(g)} & & \hat{\phi}M ~. \\
			\end{tikzcd}
		\end{equation}
		Therefore, we have a morphism $\zeta_x: (UG)^{(I_x)}\longrightarrow (UG)^{(I)}\longrightarrow M$ in $EM_U$ such that $\hat{x}$ factors through $\hat{\phi}(\zeta_x)$. We now define a subobject $N\subseteq M$ in $EM_U$ as 
		\begin{equation}
			N:=Im\left(\zeta':= \underset{x\in X}{\bigoplus} \zeta_x:\underset{x\in X}{\bigoplus} (UG)^{(I_x)}\longrightarrow M \right),
		\end{equation}
		Since $\hat{\phi}$ is exact and preserves direct sums, we have 
		\begin{equation}
			\hat{\phi}(N):=Im\left(\hat{\phi}(\zeta')=\underset{x\in X}{\bigoplus} \hat{\phi}(\zeta_x):\underset{x\in X}{\bigoplus} \hat{\phi}((UG)^{(I_x)})\longrightarrow \hat{\phi}M \right).
		\end{equation}
		It is now clear that $X\subseteq  EM_{U'}(U'G,\hat{\phi}N)$. Further,   
		\begin{equation}
			||N||^{G}\leq ||\underset{x\in X}{\bigoplus} (UG)^{(I_x)}||^{G} \leq (\lambda^U)^{\kappa(G)}\times (\alpha\times \gamma')^{\kappa(G)}\leq  \alpha^{\gamma'}~.
		\end{equation}
	\end{proof}
	
	\begin{lem}\label{L6.9}
		Let $\phi: U \longrightarrow U'$ be a coflat morphism of monads on $\mathcal C$. Let $\gamma'\geq \{\mbox{$Mor(\mathbb Q)$, $\kappa(G)$, $||UG||^G$, $||U'G||^G$} \}$ and $\alpha\geq \gamma', \lambda^U,\lambda^{U'}, \lambda^V, 
		\lambda^{V'}$. Let $(M,f_M)\in EM_U$ and let $X\subseteq el_{G}(M)$, $Y\subseteq el_{G}(\hat{\phi}M)$ be subsets such that $|X|, |Y|\leq \alpha^{\gamma'}$. Then, there exists a subobject $N\subseteq M$ in $EM_U$ such that
		\begin{enumerate}
			\item $X\subseteq el_{G}(N)$, $Y\subseteq el_{G}(\hat{\phi}N)$, 
			\item $||N||^{G}\leq \alpha^{\gamma'}$ and $|| \hat{\phi}N||^{G}\leq \alpha^{\gamma'}$. 
		\end{enumerate}
	\end{lem}
	\begin{proof} Applying Lemma \ref{L6.8} to the morphism $\phi:U\longrightarrow U'$, we obtain  $N_1\subseteq M$ in $EM_U$ such that $Y\subseteq el_{G}(\hat{\phi}N_1)$
	and $||N_1||^G \leq (\alpha^{\gamma'})^{\gamma'}=\alpha^{\gamma'}$. Also applying Lemma \ref{L6.8} to the identity morphism on $U$, we obtain $N_2\subseteq M$ in 
	$EM_U$ such that $X\subseteq el_{G}(N_2)$ and $||N_1||^G\leq (\alpha^{\gamma'})^{\gamma'}=\alpha^{\gamma'}$. We set $N:=N_1+N_2\subseteq M$ in $EM_U$. 
	We have $X\subseteq  el_{G}(N_2)\subseteq  el_{G}(N)$ and $Y\subseteq el_{G}(\hat{\phi}N_1)\subseteq  el_{G}(\hat{\phi}N)$ (note that $\hat\phi(N_1)\subseteq \hat\phi(N)$
	since $\hat\phi$ is a right adjoint). 
		As in the  proof of Lemma \ref{L6.2}, we now see that $||N||^G\leq ||N_1\oplus N_2||^G\leq \alpha^{\gamma'}$.  From the definition in 
		\eqref{eq4.17}, we know that   $\hat{\phi}N \subseteq V'N$.  By  Theorem \ref{T2.2} we now have
		\begin{equation}
			|| \hat{\phi}N||^{G}=|EM_{U'}(U'G, \hat{\phi}N)|\leq |EM_{U'}(U'G, V'N)|=|\mathcal{C}(G,V'N)|=||V'N||^G\leq \lambda^{V'}\times (||N||^{G})^{\kappa(G)}\leq \alpha^{\gamma'}~.
		\end{equation}
	\end{proof}

	\begin{defn}\label{D6.7}
		Let $\mathscr U:\mathbb Q\longrightarrow Mnd(\mathcal C)$ be a monad quiver that takes values in monads that are exact and preserve colimits. Then, we say that $\mathscr U$ is coflat if for each edge $\phi:x\longrightarrow y$ in $\mathbb E$, the morphism $\mathscr U(\phi):\mathscr U(x)\longrightarrow \mathscr U(y)$ of monads is coflat. 
	\end{defn}
	Let $\mathscr U:\mathbb Q\longrightarrow Mnd(\mathcal C)$ be a monad quiver that is coflat and which takes values in monads that are exact and preserve colimits. For each vertex $x\in\mathbb V$, let $\mathscr V_x$ be the comonad such that $(\mathscr U_x, \mathscr V_x)$ is an adjoint pair. Suppose additionally that each $\mathscr V_x$ is exact and preserves colimits. Then, it is clear from definition in \eqref{eq4.17} that $\hat{\phi}$ preserves direct sums for each edge $\phi\in \mathbb E$.  As in Section 6.1, we see that $Mod^{tr}_{c}-\mathscr U$ is a cocomplete abelian category, with colimits and finite limits computed pointwise at each $x\in \mathbb V$.

\smallskip	
	We will now show that there is a set of generators in $Mod^{tr}_{c}-\mathscr U$. From now onwards, we take $\mathbb Q$ to be a poset and fix a regular cardinal $\alpha$ such that
	\begin{equation}\label{eq10.4}
		\alpha\geq sup\{\gamma, \lambda^{\mathscr U_{x}},\lambda^{\mathscr V_{x}}, x\in\mathbb V\},
	\end{equation}
	where $(\mathscr U_x,\mathscr V_x)$ is a pair of adjoint functors for each $x\in\mathbb V$.
	
	\smallskip
	
	Let $\mathscr M \in Mod^{tr}_{c}-\mathscr U$ and $\zeta\in el_{G}(\mathscr M)$. Then by Proposition \ref{P4.11} and Lemma \ref{L4.12}, there exists a subobject $\mathscr N\subseteq \mathscr M$ in  $Mod^{tr}-\mathscr U$ such that $\zeta\in el_{G}(\mathscr N)$ and $|el_{G}(\mathscr N)|\leq \alpha^{\gamma}$. We fix a well ordering on 
	$Mor(\mathbb Q)$, which induces a lexicographic ordering on $\mathbb N\times Mor(\mathbb Q)$. We proceed with arguments similar to
	Section 6.1 to show that there exists a collection of subobjects $\{\mathscr P(n,\phi)~|~(n,\phi)\in \mathbb N\times Mor(\mathbb Q)\}$ of 
	$\mathscr M$ in $Mod^{tr}-\mathscr U$ that satisfies the following conditions.
	
	\smallskip
	
	(1) If $\phi_0$ is the least element of $Mor(\mathbb Q)$, then $\zeta\in el_G(\mathscr P(1,\phi_0))$.
	
	\smallskip 
	
	(2) If $(n,\phi)\leq (m,\psi)$ in $\mathbb N \times Mor(\mathbb Q)$, then $\mathscr P(n,\phi) \subseteq \mathscr P(m,\psi)$.
	
	\smallskip
	(3)For given $(n,\phi:y \longrightarrow z)$ in $\mathbb N \times Mor(\mathbb Q)$, the morphism $_{\phi}\mathscr P(n,\phi):\mathscr P(n,\phi)_z \longrightarrow \hat{\phi}\mathscr P(n,\phi)_y$ is an isomorphism in $EM_{\mathscr U_z}$.

	\smallskip
	
	(4) $|el_{G}(\mathscr P(n,\phi))|\leq \alpha^{\gamma}$.
	
	\smallskip
	For  $(n,\phi:y\longrightarrow z)\in \mathbb N\times Mor(\mathbb Q)$ and $w\in\mathbb V,$ we begin the transfinite induction by defining  a subset $X_0^0(w)\subseteq el_{G}(\mathscr M_w)$ as follows
	\begin{equation*}
		X_0^0(w):= 
		\begin{cases} 
			el_{G}(\mathscr N_w),\quad & \text{if}~~(n,\phi) = (1,\phi_0) \\
			\underset{(m,\psi)< (n,\phi)}\bigcup el_{G}(\mathscr P(m,\psi)_w), \quad &~~ \mbox{otherwise}~.
		\end{cases}
	\end{equation*}
	
	Clearly, $X_0^0(w)\subseteq el_{G}(\mathscr M_w)$ and $|X^0_0(w)|\leq \alpha^\gamma$. Since $\mathscr M$ is cartesian,  by Lemma \ref{L6.9}, there exists a subobject $X_{1}^0(y)\subseteq \mathscr M_y$ such that 
	\begin{equation*}
		||X_1^0(y))||^G\leq \alpha^{\gamma} \quad ||\hat{\phi}(X_1^0(y))||^G\leq \alpha^{\gamma}  \quad X^0_0(y) \subseteq el_{G}(X_1^0(y))  \quad  X_0^0(z) \subseteq el_{G}(\hat{\phi}(X_1^0(y))).
	\end{equation*} 
	Now, set $X_1^0(z) = \hat{\phi}X_1^0(y)$.  For each $w\in \mathbb V$, we set
	\begin{equation*}
		Y_1^0(w):=\left\{\begin{array}{ll} el_{G}(X_1^0(w))  & \mbox{if $w=y,z$} \\
			X^0_0(w) & \mbox{otherwise}.\\
		\end{array}\right.
	\end{equation*}
	Then $X_{0}^{0}(w)\subseteq Y_{1}^{0}(w)$ and $|Y_{1}^{0}(w)|\leq \alpha^{\gamma}.$
	
	\begin{lem}\label{L6.11}
		Let $X\subseteq el_G(\mathscr M)$ with $|X|\leq\alpha^\gamma$. Then there exists a subobject $\mathscr D\hookrightarrow\mathscr M$ in $Mod^{tr}-\mathscr U$ such that $X \subseteq el_G(\mathscr D)$ and  $|el_G(\mathscr D)|\leq\alpha^\gamma$.
	\end{lem}
	\begin{proof}
		Let $\zeta\in X\subseteq el_G(\mathscr M)$. By the proof of Theorem \ref{T4.13}, there exists a subobject $\mathscr D_{\zeta}\hookrightarrow\mathscr M$ in 
		$Mod^{tr}-\mathscr U$ such that $\zeta\in el_G(\mathscr D_{\zeta})$ and  $|el_G(\mathscr D_\zeta)|\leq\gamma^{\kappa(G)}\leq \alpha^\gamma$. Then, we set $\mathscr D := \underset{\zeta\in X}\sum\mathscr D_\zeta$. Since $\mathscr D$ is a quotient of $\underset{\zeta\in X}\bigoplus\mathscr D_\zeta$ and $X \subseteq el_G(\mathscr D)$, we have  from Lemma \ref{L2.1} that
		\begin{equation}
			|el_G(\mathscr D)| \leq \left\vert el_G \left(\underset{\zeta\in X}\bigoplus\mathscr D_\zeta\right)\right\vert \leq  \underset{y\in\mathbb V}\sum\left\vert EM_{\mathscr U_y}\left(\mathscr U_yG,\underset{\zeta\in X}\bigoplus\mathscr{D}_{\zeta_y}\right)\right\vert=\underset{y\in\mathbb V}\sum\left\vert\mathcal C\left(G,\underset{\zeta\in X}\bigoplus\mathscr{D}_{\zeta_y}\right)\right\vert\leq \alpha^\gamma
		\end{equation}
	\end{proof}
	Since $\underset{w\in\mathbb V}\bigcup Y_1^0(w)\subseteq el_{G}(\mathscr M)$ and $|\underset{w\in\mathbb V}\bigcup Y_1^0(w)|\leq \alpha^{\gamma},$ by Lemma \ref{L6.11}, we obtain a subobject $\mathscr Q^{0}(n,\phi)\subseteq \mathscr M$ such that $\underset{w\in\mathbb V}\bigcup Y_1^0(w)\subseteq el_{G}(\mathscr Q^{0}(n,\phi))$ (i.e., for each $w\in\mathbb V$, $Y_1^0(w)\subseteq el_{G}(\mathscr Q^0(n,\phi)_w)$ and $|el_{G}(\mathscr Q^{0}(n,\phi))|\leq \alpha^{\gamma}.$ 
	
	\smallskip
	
	We proceed by induction and suppose that for every $s\leq t,$ we have constructed a subobject $\mathscr Q^{s}(n,\phi)\subseteq \mathscr M$ that satisfies $\underset{w\in\mathbb V}\bigcup Y_1^s(w)\subseteq el_{G}(\mathscr Q^{s}(n,\phi))$ and $|el_{G}(\mathscr Q^{s}(n,\phi))|\leq \alpha^{\gamma}.$ Then, for each $w\in\mathbb V,$ we set $$X_{0}^{t+1}(w):=el_{G}(\mathscr Q^{t}(n,\phi)_w).$$ 
	On applying Lemma \ref{L6.9} again, we obtain a subobject  $X_{1}^{t+1}(y)\subseteq\mathscr M_y$ such that
	\begin{equation*}
		||X_1^{t+1}(y))||^G\leq \alpha^{\gamma} \quad ||\hat{\phi}(X_1^{t+1}(y))||^G\leq \alpha^{\gamma}  \quad X^{t+1}_0(y) \subseteq el_{G}( X_1^{t+1}(y))  \quad  X_0^{t+1}(z) \subseteq el_{G}( \hat{\phi}(X_1^{t+1}(y))).
	\end{equation*} 
	Now set $X_{1}^{t+1}(z)=\hat{\phi}(X_{1}^{t+1}(y))$ and for each $w\in\mathbb V$
	\begin{equation*} 
		Y_1^{t+1}(w):=\left\{\begin{array}{ll} el_{G} (X_1^{t+1}(w))  & \mbox{if $w=y,z$} \\
			X_0^{t+1}(w)=el_{G}(\mathscr Q^{t}(n,\phi)_w) & \mbox{otherwise}\\
		\end{array}\right.
	\end{equation*}
	Then, $X_{1}^{t+1}(w)\subseteq Y_{1}^{t+1}(w)$ and $|Y_{1}^{t+1}(w)|\leq \alpha^{\gamma}.$ Using Lemma \ref{L6.11}, there exists a subobject $\mathscr Q^{t+1}(n,\phi)\subseteq \mathscr M$ such that 
	$\underset{w\in\mathbb V}\bigcup Y_1^{t+1}(w)\subseteq el_{G}(\mathscr Q^{t+1}(n,\phi)),$ and $|el_{G}(\mathscr Q^{t+1}(n,\phi))|\leq \alpha^{\gamma}.$ In particular, for each $w\in\mathbb V$, we have $Y_1^{t+1}(w)\subseteq el_{G}(\mathscr Q^{t+1}(n,\phi)_w).$
	In this way, we have constructed subobjects of $\mathscr M$ in $Mod^{tr}-\mathscr U$ such that 
	\begin{equation*}
		\mathscr Q^{0}(n,\phi)\hookrightarrow\mathscr Q^{1}(n,\phi)\hookrightarrow\ldots\hookrightarrow\mathscr Q^{t}(n,\phi)\hookrightarrow\ldots~.
	\end{equation*}
	Finally, we set 
	\begin{equation}\label{eq6.30}
		\mathscr P(n,\phi):= \underset{t\geq 0}\bigcup \mathscr Q^{t}(n,\phi).
	\end{equation} Since  $\hat{\phi}$ is exact and preserves filtered colimits, it follows as in Section 6.1 that the  family $\{\mathscr P(n,\phi)~|~(n,\phi)\in\mathbb N\times Mor(\mathbb Q)\}$  of subobjects of $\mathscr M$ in $Mod^{tr}-\mathscr U$ constructed above satisfies the conditions (1)-(4).

	\begin{lem}\label{L6.13}
		Let $\mathscr M\in Mod^{tr}_{c}-\mathscr U$ and $\zeta\in el_{G}(\mathscr M).$ Then there exists a subobject $\mathscr P\subseteq \mathscr M$ in $Mod^{tr}_{c}-\mathscr U$ with $\zeta\in el_{G}(\mathscr P)$ such that $|el_{G}(\mathscr P)|\leq \alpha^{\gamma}.$  
	\end{lem}
	\begin{proof}
	
	The proof of this is similar to that of Lemma \ref{L6.4}
	\end{proof}
	\begin{Thm}\label{T6.14}
		 $Mod^{tr}_{c}-\mathscr U$ is a Grothendieck category. 
	\end{Thm}
	\begin{proof}
		Since filtered colimits and finite limits in $Mod^{tr}_{c}-\mathscr U$ are both computed in $Mod^{tr}-\mathscr U$, we see that  $Mod^{tr}_{c}-\mathscr U$ satisfies (AB5).  For any object $\mathscr M\in Mod^{tr}_{c}-\mathscr U$ and any element $\zeta\in el_{G}(\mathscr M)$, it follows from Lemma \ref{L6.13} that there exists a subobject $\mathscr P\subseteq \mathscr M$ in $Mod^{tr}_{c}-\mathscr U$ such that $\zeta\in el_{G}(\mathscr P)$ and $|el_{G}(\mathscr P)|\leq\alpha^{\gamma}.$ Therefore the isomorphism classes of cartesian modules $\mathscr P'$ with $|el_{G}(\mathscr P')|\leq \alpha^{\gamma}$ give a set of generators for $Mod^{tr}_{c}-\mathscr U$.
	\end{proof}
	\begin{Thm}\label{T6.15}
		The canonical inclusion $\iota: Mod^{tr}_{c}-\mathscr U\longrightarrow Mod^{tr}(\mathscr U)$ has a right adjoint.  
	\end{Thm}
	\begin{proof}
		Since  $Mod^{tr}_{c}-\mathscr U$ is a Grothendieck category and the inclusion functor $\iota$ preserves colimits, the result follows from \cite[Proposition 8.3.27]{KS}.  
	\end{proof}
	\section{Comodules over a comonad quiver} 
	We continue to assume that $\mathcal C$ is a Grothendieck category with a projective generator $G$.  We denote by $Cmd(\mathcal C)$ the category
of comonads over $\mathcal C$.  Let $(V,\delta,\epsilon)$ and $(V',\delta',\epsilon')$ be comonads in $\mathcal C$ and let $\phi:(V,\delta,\epsilon) \longrightarrow (V',\delta',\epsilon')$ be a morphism in $Cmd(\mathcal C)$. Then, $\phi$ induces a corestriction functor $\phi^{\circ}:EM^V\longrightarrow EM^{V'}$ defined by setting $\phi^{\circ}(M,f^M):= (M,\phi(M) \circ f^M)$ for 
$(M,f^M)\in EM^V$. Suppose that $V$ and $V'$ are exact and preserve colimits. Then, $\phi^{\circ}$ is exact and preserves colimits.
	We also have a functor $\phi_{\circ}:EM^{V'}\longrightarrow EM^{V}$ defined by setting 
	\begin{equation}\label{eq7.1}
		\phi_{\circ}(M'):= Eq\bigg( VM'\doublerightarrow{\qquad V\phi(M')\circ \delta(M')\qquad }{Vf^{M'}}VV'M'\bigg).
	\end{equation}  for each $(M',f^{M'})\in EM^{V'}$.
	As in Lemma \ref{L4.9}, we know that 	$(\phi^{\circ}, \phi_{\circ})$ is a pair of adjoint functors. In other words, we have natural isomorphisms 
\begin{equation*}EM^{V}(M,\phi_{\circ}(M'))\cong EM^{V'}(\phi^{\circ}(M), M').
\end{equation*} for any $(M,f^M) \in EM^V$ and $(M',f^{M'}) \in EM^{V'}$.
	
	\begin{defn}\label{D7.2}  A comonad quiver is a functor $\mathscr V:\mathbb Q\longrightarrow Cmd(\mathcal C)$. 
		
	\end{defn}

For each $x\in \mathbb V$, we denote the comonad $\mathscr V(x)$ by $\mathscr V_x$.  If $\phi: x\longrightarrow y$ is an edge of $\mathbb Q$, by abuse of notation, we will continue to denote $\mathscr V(\phi):
	\mathscr V_x\longrightarrow \mathscr V_y$ simply by $\phi$. Accordingly, we will often write
	$
	\phi^\circ=\mathscr V(\phi)^\circ:EM^{\mathscr V_x}\longrightarrow EM^{\mathscr V_y}$ and $\phi_\circ=\mathscr U(\phi)_\circ:EM^{\mathscr V_y}\longrightarrow EM^{\mathscr V_x}
	$ for an edge $\phi:x\longrightarrow y$ in $\mathbb Q$.

	\begin{eg}
		We shall now give a few examples of a comonad quiver on a category. Let $k$ be a field of characteristic zero. 
	\end{eg}

	\smallskip
	
	(1) Let $Vect_{k}$ denote the category of vector spaces over $k,$ and $Coalg_{k}$ denote the category of coalgebras over $k$. For any coalgebra $C\in Coalg_{k}$, $C\otimes_{k} -: Vect_{k}\longrightarrow Vect_{k}$ gives a comonad on $Vect_{k}.$  Let $\mathcal T: \mathbb Q\longrightarrow Coalg_{k}$ be a functor. Then, $\mathscr V:\mathbb Q\longrightarrow Cmd(Vect_{k})$ defined by  $\mathscr V(x):= \mathcal T(x)\otimes (-)$ is  a comonad quiver on $Vect_{k}.$

	\smallskip
	
	(2) Let $\mathfrak S_{k}$ be a $k$-linear Grothendieck category. For any $C\in Coalg_{k,}$ $C\otimes -:\mathfrak S_{k}\longrightarrow \mathfrak S_{k}$ defines a comonad on $\mathfrak S_{k}.$ The Eilenberg-Moore category of the comonad $C\otimes-$ is the category $\mathfrak S_{C}$ of $C$-comodule objects of $\mathfrak S_{k}$ (see, \cite[\S 39]{BWb}). Now, let $\mathcal T: \mathbb Q\longrightarrow Coalg_{k}$ be a functor. Then, $\mathscr V:\mathbb Q\longrightarrow Cmd(\mathfrak S_{k})$ defined by $\mathscr V(x):= \mathcal T(x)\otimes (-)$ is a comonad quiver on $\mathfrak S_{k}$.
	
	\smallskip
	
	(3) Let $H$ be a Hopf algebra. Then, the category $H-Mod$ of left $H$-modules is a monoidal category. Let $C$ be any coalgebra object in $H-Mod.$ Then, the functor $C\otimes -: H-Mod \longrightarrow H-Mod$ defines a comonad on $H-Mod$.  If $\mathcal T:\mathbb Q\longrightarrow Coalg(H-Mod)$ is a functor that takes values in coalgebra objects in $H-Mod$,  then the functor $\mathscr V:\mathbb Q\longrightarrow Cmd(H-Mod)$, $x\longrightarrow \mathcal T(x)\otimes (-)$ defines a comonad quiver on $H-Mod.$
	
	\smallskip
	
	In general, let $(\mathcal D,\otimes)$ be a $k$-linear monoidal category and $\mathcal L$ be a $k$-linear Grothendieck category with an action $-\otimes -:\mathcal D\times 
	\mathcal L\longrightarrow \mathcal L$. Then, any object in the category $Coalg(\mathcal D)$ of coalgebra objects in $\mathcal D$ defines a comonad on $\mathcal L$ given by $C\otimes-:\mathcal L\longrightarrow \mathcal L$.  If $\mathcal T: \mathbb Q\longrightarrow Coalg(\mathcal D)$ is a functor, then  $\mathscr V:\mathbb Q\longrightarrow Cmd(\mathcal L)$ given by $\mathscr V(x):= \mathcal T(x)\otimes (-)$ defines a comonad quiver on $\mathcal L$. 
	\smallskip

	(4) Let $\mathcal S$ be a Grothendieck category with a projective generator. Let $(U, \theta, \eta)$ be a Frobenius monad on $\mathcal S$. Then, we know from  \cite[Theorem 1.6]{RS} that $U$ has a right adjoint comonad structure $V=(U, \delta, \epsilon)$. Then, $EM_{U}\cong EM^V$, and thus $EM^{V}$ has a projective generator. Therefore, it follows that $V=(U, \delta, \epsilon)$ is a semiperfect comonad on $\mathcal C$. We now consider a functor $\mathscr U:\mathbb Q^{op}\longrightarrow Mnd(\mathcal S)$ such that for each $x\in\mathbb V$, $\mathscr U(x)$ is a Frobenius monad on $\mathcal S$. Then, $\mathscr U$ determines a comonad quiver on $\mathcal S$.
	
	\smallskip
	
	(5) Let $k$ be a field of characteristic zero and let $(\mathcal D,\otimes)$ be a multitensor category, i.e., a locally finite $k$-linear abelian rigid monoidal category (see \cite[$\S$ 4.1]{Et}). Consider a module category $\mathcal L$ over $\mathcal D$ (see \cite[$\S$ 7.3]{Et}). Then for any coalgebra object $C\in\mathcal D$, the functor $C\otimes-:\mathcal L\longrightarrow \mathcal L$ determines a comonad on $\mathcal L$. Let $Ind(\mathcal L)$
 denote the ind-completion of $\mathcal L$. Similar to Section 4, 
we see that $C\otimes-:\mathcal L\longrightarrow \mathcal L$ extends to a comonad $C\otimes-:Ind(\mathcal L)\longrightarrow Ind(\mathcal L)$ that is exact and preserves colimits. 
Now, let
	$\mathcal T :\mathbb Q\longrightarrow Coalg(\mathcal D)$ be a functor taking values in the category 
$Coalg(\mathcal D)$ of coalgebra objects in $(\mathcal D,\otimes)$. Then, the functor $\mathscr V:\mathbb Q\longrightarrow Cmd(Ind(\mathcal L))$, $x\mapsto \mathcal T(x)\otimes(-)$ determines comonad quiver on $Ind(\mathcal L)$. As mentioned in Section 4, there are several examples of such situations
in the literature (see \cite[$\S$ 7.4]{Et}).
	
	
	\smallskip
	
	(6) Let $\mathcal D$ be a multitensor category and $\mathcal L$ be an exact module category over $\mathcal D$ \cite[see Definition 7.5.1]{Et}. Then, the category $Fun_{\mathcal D}(\mathcal L,\mathcal L)$ of right exact $\mathcal D$-module functors is a monoidal category with the tensor product being the composition of functors (see \cite[\S 7.11]{Et}). 
	We note that by \cite[Lemma 7.11.4.]{Et}, every $F\in Fun_{\mathcal D}(\mathcal L,\mathcal L)$ has left and right adjoints. Let 
$Cmd_{\mathcal D}(\mathcal L,\mathcal L)\subseteq Fun_{\mathcal D}(\mathcal L,\mathcal L)$ be the (not necessarily full) subcategory consisting 
of $\mathcal D$-module endofunctors that are comonads on  $\mathcal L$.  As before, any such comonad on $\mathcal L$ extends to a comonad
on $Ind(\mathcal L)$. Accordingly,  any functor $\mathscr V:\mathbb Q\longrightarrow Cmd_{\mathcal D}(\mathcal L,\mathcal L)$ determines a comonad quiver on $Ind(\mathcal L)$.

	\subsection{Cis-comodules over a comonad quiver}
	
	We fix a comonad quiver $\mathscr V:\mathbb Q\longrightarrow Cmd(\mathcal C)$ over $\mathcal C$. For each $x\in \mathbb V$, we assume that the comonad
	$\mathscr V_x$ is exact and preserves colimits.
	
	\begin{defn}\label{D7.3}
 A cis-comodule $\mathscr M$ over $\mathscr V$ consists of a collection $\{\mathscr M_x \in EM^{\mathscr V_x}\}_{x\in\mathbb V}$ connected by morphisms $\mathscr M_\phi: \mathscr M_x \longrightarrow \phi_{\circ}\mathscr M_y$ in $EM^{\mathscr V_x}$ (equivalently, morphisms $\mathscr M^{\phi}:\phi^{\circ}\mathscr M_x \longrightarrow \mathscr M_y$ in $EM^{\mathscr V_y}$) for each edge $\phi:x\longrightarrow y\in\mathbb E$ such that $\mathscr M_{id_x}=id_{\mathscr M_{x}}$ for each $x\in\mathbb V$, and for any pair of composable morphisms $x\xrightarrow{\phi} y\xrightarrow{\psi} z$ in $\mathbb E$, we have $\phi_{\circ}(\mathscr M_{\psi})\circ\mathscr M_{\phi}=\mathscr M_{\psi\phi}:\mathscr M_{x}\longrightarrow \phi_{\circ}\mathscr M_{y}\longrightarrow \phi_{\circ}\psi_{\circ}\mathscr M_{z}$ (equivalently, $\mathscr M^{\psi}\circ\psi^{\circ}(\mathscr M^{\phi})=\mathscr M^{\psi\phi}$).
		
		\smallskip
		
		A morphism $\xi:\mathscr  M\longrightarrow \mathscr N$ of cis-comodules over $\mathscr V$ consists of set of morphisms $\{\xi_{x}:\mathscr M_x\longrightarrow \mathscr N_x \in EM^{\mathscr V_x}~|~x\in\mathbb Q\}$ such that for each edge $\phi:x\longrightarrow y$ in $\mathbb E$, we have $\mathscr N_{\phi}\circ \xi_x = \phi_{\circ}\xi_y\circ \mathscr M_{\phi}.$
		This forms the category of cis-comodules over $\mathscr V$, and we denote it by $Com^{cs}-\mathscr V$.
		
		\smallskip
		We will say that an object $\mathscr M\in Com^{cs}-\mathscr V$ is cartesian if for each edge $\phi:x\longrightarrow y$ in $\mathbb E,$ the connecting morphism $\mathscr M_\phi:\mathscr M_x\longrightarrow \phi_{\circ}\mathscr M_y$ is an isomorphism in $EM^{\mathscr V_x}.$ We let $Com^{cs}_{c}-\mathscr V$ denote the full subcategory of cartesian cis-comodules over $\mathscr V$.
	\end{defn}
	
	We note that $Com^{cs}-\mathscr V$ is a cocomplete abelian category, with colimits and finite limits computed pointwise at each $x\in \mathbb V$. 
	Additionally in this section,  we assume that the functor $\mathscr V:\mathbb Q\longrightarrow Cmd(\mathcal C)$ is such that for each $x\in\mathbb V$, the comonad $\mathscr V_x$ is semiperfect.
	Then, for each $x\in\mathbb V$, we fix a projective generator $G_x$ of the Grothendieck category $EM^{\mathscr V_x}$.
	
	\smallskip
For any object $\mathscr M\in Com^{cs}-\mathscr V$, we now set 
	\begin{equation}\label{eq7.4}
		el_{\mathscr V}(\mathscr M):= \underset{x\in \mathbb V}{\coprod} EM^{\mathscr V_x}(G_x,\mathscr M_x)
	\end{equation}
	From (\ref{eq7.4}), it is easy to observe that for any subobject $\mathscr M'\subseteq\mathscr M$ in $Com^{cs}-\mathscr V$, we must have $el_{\mathscr V}(\mathscr M')\subseteq el_{\mathscr V}(\mathscr M)$. The equality holds if and only if $el_{\mathscr V}(\mathscr M')= el_{\mathscr V}(\mathscr M).$ 
		In a manner similar to earlier sections, we will now show that $Com^{cs}_{c}-\mathscr V$ has a set of generators.
	\smallskip

	Let us fix an element $\eta\in el_{\mathscr V}(\mathscr M).$ Then, there exists some $x\in\mathbb V$ such that $\eta:G_x\longrightarrow \mathscr M_x$ is a morphism in $EM^{\mathscr V_x}$.
	Then, for each $y\in\mathbb V$, we define a subobject $\mathscr N_y\subseteq \mathscr M_y$ in $EM^{\mathscr V_y}$ as
	\begin{equation}\label{eq7.5}
		\mathscr N_y = Im\bigg(\underset{\psi\in \mathbb Q(x,y)}{\bigoplus}\psi^{\circ} G_x \xrightarrow{\psi^{\circ}\eta}\psi^{\circ}\mathscr M_x\xrightarrow{\mathscr M^\psi}\mathscr M_y\bigg) = \underset{\psi\in \mathbb Q(y,x)} \sum Im \left(\psi^{\circ}G_x\xrightarrow{\psi^{\circ}\eta}\psi^{\circ}\mathscr M_x\xrightarrow{\mathscr M^\psi}\mathscr M_y\right)~.
	\end{equation}
	Then, for each $\psi\in\mathbb Q(x,y)$, let $\eta_{\psi}'=\psi^{\circ}G_x\longrightarrow \mathscr N_y$ denote the morphism induced from (\ref{eq7.5}). For each 
	$y\in \mathbb V$, we have an inclusion $\iota_y:\mathscr N_y\longrightarrow \mathscr M_y$.  By \eqref{eq7.5}, it is clear that $\eta\in EM^{\mathscr V_x}(G_x,\mathscr N_x)$.
	
	\begin{lem}\label{L7.5}
		The objects $\{\mathscr N_y\in EM^{\mathscr V_y}\}_{y\in\mathbb V}$ together determine a subobject $\mathscr N\subseteq\mathscr M$ in $Com^{cs}-\mathscr V.$
	\end{lem}
	\begin{proof}
		We consider an edge $\phi:y\longrightarrow z$ in $\mathbb E$. Since, $\phi_{\circ}$ is a right adjoint, $\phi_{\circ}(\iota_z)$ is  a monomorphism in $EM^{\mathscr V_y}.$  We will  show that the morphism $\mathscr M_{\phi}:\mathscr M_y\longrightarrow \phi_{\circ}\mathscr M_z$ restricts to the morphism $\mathscr N_{\phi}:\mathscr N_y\longrightarrow \phi_{\circ}\mathscr N_z$ such that $\mathscr M_\phi\circ \iota_y=\phi_{\circ}(\iota_z)\circ \mathscr N_\phi$. 
		As in previous sections, it suffices to show that for any morphism $\xi: G_y\longrightarrow \mathscr N_y$, there exists a morphism $\xi':G_y\longrightarrow \phi_{\circ}\mathscr N_z$ such that $\phi_{\circ}(\iota_z)\circ\xi'=\mathscr M_{\phi}\circ \iota_y\circ\xi.$ From (\ref{eq7.5}), there exists an epimorphism 
$
			\underset{\psi\in \mathbb Q(x,y)}{\bigoplus}\eta_{\psi}':\underset{\psi\in \mathbb Q(x,y)}{\bigoplus}\psi^{\circ} G_x\longrightarrow \mathscr N_y\notag
	$
		in $EM^{\mathscr V_y}.$ Since $G_y$ is projective in $EM^{\mathscr V_y},$ the morphism $\xi$ can be lifted to a morphism $\xi'':G_y\longrightarrow \underset{\psi\in \mathbb Q(x,y)}{\bigoplus}\psi^{\circ} G_x $ such that $
			\xi=\bigg(\underset{\psi\in \mathbb Q(x,y)}{\bigoplus}\eta_{\psi}'\bigg)\circ\xi''
$.   Since $(\phi^{\circ},\phi_{\circ})$ is an adjoint pair, it follows as in the proof of  Proposition \ref{P4.5} that 
		the  composition $\mathscr M_{\phi}\circ \iota_y\circ \bigg(\underset{\psi\in \mathbb Q(x,y)}{\bigoplus}\eta_{\psi}'\bigg)$ factors through $\phi_{\circ}(\iota_z):\phi_{\circ}\mathscr N_{z}\longrightarrow \phi_{\circ}\mathscr M_{z}$ as
		$
			\mathscr M_{\phi}\circ \iota_y\circ \bigg(\underset{\psi\in \mathbb Q(x,y)}{\bigoplus}\eta_{\psi}'\bigg) = \phi_{\circ}(\iota_z)\circ \tau
		$.  Composing this with $\xi''$, we obtain 
		\begin{equation}
			\mathscr M_{\phi}\circ \iota_y\circ \bigg(\underset{\psi\in \mathbb Q(x,y)}{\bigoplus}\eta_{\psi}'\bigg) \circ \xi''= \mathscr M_{\phi}\circ \iota_y\circ \xi = \phi_{\circ}(\iota_z)\circ \tau \circ \xi''.\notag
		\end{equation}
		We can now set $\xi' = \tau \circ \xi'' : G_y\longrightarrow \phi_{\circ}\mathscr N_z.$ This completes the proof.
	\end{proof}
	 
	Since $G$ is a generator for $\mathcal C$, for each $x\in\mathbb V$, there exists an indexing set $I_x$ such that $G^{(I_x)}\longrightarrow G_x$ is an epimorphism in $\mathcal C$. 
	Accordingly, we may choose an indexing set $I$ that is large enough so that we have epimorphisms 
	$G^{(I)}\longrightarrow G_x$ for each $x\in \mathbb V$. We note that $G^{(I)}$ is also a projective generator of $\mathcal C$ and we set
	$G':=G^{(I)}$.   
	We now fix an infinite regular cardinal  $\delta$ such that 
	\begin{equation}\label{eq7.9}
		\delta\geq sup\{ |Mor(\mathbb Q)|, \kappa(G'), ||G_x||^{G'}, x\in\mathbb V\} 
	\end{equation} where $||M||^{G'}:=|\mathcal C(G',M)|=|\mathcal C(G^{(I)},M)|$ for any object $M\in \mathcal C$.
	\begin{lem}\label{L7.6}
		We have $|el_{\mathscr V}(\mathscr N)|\leq \delta^{\kappa(G')}.$
	\end{lem}
	\begin{proof} 
		From (\ref{eq7.5}), $\mathscr N_y$ is a quotient of $\underset{\psi\in \mathbb Q(x,y)}{\bigoplus}\psi^{\circ}G_x$ in 
		$EM^{\mathscr V_{y}}$.  We know that $G_y$ is projective in $EM^{\mathscr V_{y}}$.  Accordingly, we have
		\begin{equation}\label{7.5vt}
			\vert EM^{\mathscr V_{y}}(G_y,\mathscr N_y)| \leq \left|EM^{\mathscr V_{y}}\left(G_y,\underset{\psi\in \mathbb Q(x,y)}{\bigoplus}\psi^{\circ}G_x\right)\right|
		\end{equation}
		We note that if $(N,f^N)$ is an object in $EM^{\mathscr V_{y}}$, the structure map $f^N: N\longrightarrow \mathscr V_y N$ is a monomorphism in 
		$\mathcal C$ and hence in $EM^{\mathscr V_{y}}$. From \eqref{7.5vt}, it now follows that 
		\begin{equation}\label{7.6tc}
			\vert EM^{\mathscr V_{y}}(G_y,\mathscr N_y)| \leq \left|EM^{\mathscr V_{y}}\left(G_y,\mathscr V_{y}(\underset{\psi\in \mathbb Q(x,y)}{\bigoplus}\psi^{\circ}G_x)\right)\right|= \left|\mathcal C\left(G_y, \underset{\psi\in \mathbb Q(x,y)}{\bigoplus}\psi^{\circ}G_x\right)\right|
		\end{equation}
		We note that  $\psi^{\circ}(G_x)=G_x$ as an object of $\mathcal C$ for every $\psi\in\mathbb Q(x,y)$. Applying Lemma 
		\ref{L2.1} with the generator $G'=G^{(I)}$, we now have
		%
		\begin{equation}
			\left|\mathcal C\left(G_y, \underset{\psi\in \mathbb Q(x,y)}{\bigoplus}\psi^{\circ}G_x\right)\right| \leq \left|\mathcal C\left(G^{(I_y)},\underset{\psi\in \mathbb Q(x,y)}{\bigoplus}G_x\right)\right| \leq \left|\mathcal C\left(G^{(I)},\underset{\psi\in \mathbb Q(x,y)}{\bigoplus}G_x\right)\right|\leq \delta^{\kappa(G')} 
		\end{equation}	
		From \eqref{eq7.4}, it is now clear that $|el_{\mathscr V}(\mathscr N)|\leq \delta^{\kappa(G')}.$
	\end{proof}
	
	\begin{Thm}\label{T7.7}
$Com^{cs}-\mathscr V$ is a Grothendieck category.    
	\end{Thm}
	\begin{proof}
		We already know that $Com^{cs}-\mathscr V$ is a cocomplete abelian category where filtered colimits are exact.
		Now, let $\mathscr M\in Com^{cs}-\mathscr V$ and $\eta\in el_{\mathscr V}(\mathscr M).$ By Lemma \ref{L7.5}and Lemma \ref{L7.6}, there exists a subobject $\mathscr N\subseteq \mathscr M$ such that $\eta\in el_{\mathscr V}(\mathscr N)$ and $|el_{\mathscr V}(\mathscr N)|\leq \delta^{\kappa(G')}.$ Since $G_x$ is a generator for $EM^{\mathscr V_x},$ then for any object $\mathscr P\in Com^{cs}-\mathscr V$, $\mathscr P_x$  is an epimorphic image of $G_x^{(EM^{\mathscr V_x}(G_x, \mathscr P_x))}.$  As $EM^{\mathscr V_x}$ is a Grothendieck category, it is well-powered. Therefore, the isomorphism classes of $\mathscr V$-comodules $\mathscr P$ with $|el_{\mathscr V}(\mathscr P)|\leq \delta^{\kappa(G')}$ give a set of generators for $Com^{cs}-\mathscr V.$
	\end{proof}
	
	In the remaining part of this section, we will assume that $\mathbb Q$ is a poset. We will now study projective generators in $Com^{cs}-\mathscr V$.

	\begin{lem} \label{L8.1}Let $x\in\mathbb V.$ Then,
		
		\smallskip
		
		(1)	There exists an extension functor $ex_x^{cs}: EM^{\mathscr V_x} \longrightarrow Com^{cs}-\mathscr V$ as 
		\begin{equation}
			ex_x^{cs}(M)_{y}=\left\{\begin{array}{ll} \psi^{\circ} M & \mbox{if $\psi\in\mathbb Q(x,y)$} \\
				0 & \mbox{$otherwise$}\\
			\end{array}\right.
		\end{equation}

	for each $y\in\mathbb V$ and $M\in EM^{\mathscr V_x}$. 
	
	\smallskip
	
	(2) The evaluation functor $ev_x^{cs}: Com^{cs}-\mathscr V\longrightarrow EM^{\mathscr V_x}$, defined as $ev_x^{cs}(\mathscr M)=\mathscr M_x$ is exact.
	
	\smallskip
	
	(3) $(ex_x^{cs},ev_x^{cs})$ forms a pair of adjoint functors, i.e., for any  $\mathscr M\in Com^{cs}-\mathscr V$ and $N\in EM^{\mathscr V_x},$ we have 
	\begin{equation*}Com^{cs}-\mathscr V(ex_x^{cs}(N),\mathscr M)\cong EM^{\mathscr V_x}(N,ev_x^{cs}(\mathscr M)).
	\end{equation*}
	
	\smallskip
	(4) The evaluation functor $ev_x^{cs}: Com^{cs}-\mathscr V\longrightarrow EM^{\mathscr V_x}$ also has a right adjoint given by 
	\begin{equation}
		coe_x^{cs}(M)_{y}=\left\{\begin{array}{ll} \psi_{\circ} M & \mbox{if $\psi\in\mathbb Q(y,x)$} \\
			0 & \mbox{$otherwise$}.\\
		\end{array}\right.
	\end{equation}
	
	\smallskip
	(5) For each $x\in \mathbb V$, the functor  $ex_x^{cs}: EM^{\mathscr V_x} \longrightarrow Com^{cs}-\mathscr V$ preserves projectives.
		\end{lem}
	\begin{proof} This is proved in a manner similar to Proposition \ref{P5.1}, Proposition \ref{P5.2} and Corollary \ref{C5.3}. 
	\end{proof}
	
	\begin{Thm}\label{T8.3}
		The category $Com^{cs}-\mathscr V$ has a set of projective generators.    
	\end{Thm}
	\begin{proof}
		We consider the family $\mathcal F=\{ex_x^{cs}(G_x)~|~x\in\mathbb V, G_x\in EM^{\mathscr V_x}\}$, where $G_x$ is the projective generator in $EM^{\mathscr V_x}$ as taken above.  By Lemma \ref{L8.1}, $ex_x^{cs}(G_x)$ is a projective object in $Com^{cs}-\mathscr V.$ As in the proof of Theorem \ref{Th5.4}, we can now show that 
		$\mathcal F$ is a set of generators for  $Com^{cs}-\mathscr V.$
			\end{proof}

	\subsection{Cartesian comodules over a comonad quiver}
	
	We continue with the setup of Section 7.1. Our next aim is to study conditions for  the category of cartesian cis-comodules to form a Grothendieck category. 
	 Let $\phi: (V,\delta,\epsilon) \longrightarrow (V',\delta',\epsilon')$ be a morphism of comonads over $\mathcal C$. Suppose that $V$ and $V'$ are exact and preserve colimits. Then, it is clear from the definitions that the induced functor $\phi^{\circ}:EM^V\longrightarrow EM^{V'}$ is exact. 
	\begin{defn}\label{D9.1} 
Let $\phi: (V,\delta,\epsilon) \longrightarrow (V',\delta',\epsilon')$ be a morphism of comonads over $\mathcal C$. We will say that $\phi$ is coflat if $\phi_{\circ}:EM^{V'}\longrightarrow EM^{V}$ is exact. A comonad quiver $\mathscr V:\mathbb Q\longrightarrow Cmd(\mathcal C)$ is  coflat if for each edge $\phi:x\longrightarrow y$ in $\mathbb Q$, the induced morphism $\mathscr V(\phi):\mathscr V_x\longrightarrow \mathscr V_y$ of comonads is coflat. 
	\end{defn}

	Let $\phi:(V,\delta,\epsilon) \longrightarrow (V',\delta',\epsilon')$  be a coflat morphism of comonads over $\mathcal C.$  By assumption, $\phi_\circ$ is exact. In particular, $\phi_\circ$ preserves finite colimits. Additionally, since $V$ and $V'$ are exact and preserve colimits, it is clear from the definition in \eqref{eq7.1} that $\phi_\circ$ preserves filtered colimits. Since every colimit can be expressed as a combination of filtered colimits and finite colimits, it follows that $\phi_\circ$ preserves all colimits.

	\begin{lem}\label{L9.3}
		Let $\phi: V\longrightarrow V'$ be a coflat morphism in $Cmd(\mathcal C)$. Suppose that 
		
		\smallskip
		(a) The comonads $V$ and $V'$ are semiperfect. Let $G_V$ and $G_{V'}$ be projective generators for $EM^{V}$ and $EM^{V'}$ respectively and let 
		$I$ be an indexing set such that we have epimorphisms in $\mathcal C$
		\begin{equation}
		G''=G^{(I)}\longrightarrow G_V \qquad G''=G^{(I)}\longrightarrow G_{V'}
		\end{equation} (b) We have  an infinite regular cardinal  
		$ \delta'\geq  \text{max}\{|Mor(\mathbb Q)|, \kappa(G''), ||G_V||^{G''}, ||G_{V'}||^{G''},\kappa(G_V), \kappa(G_{V'})\}
		$ and $\beta\geq \delta'$.

		\smallskip
		Let $(M',f^{M'})\in EM^{V'}$ and let $X\subseteq EM^{V}(G_V, \phi_{\circ}M')$ be such that $|X|\leq \beta.$ Then, there exists a subobject $N'\subseteq M'$ in $EM^{V'}$ such that $X\subseteq EM^{V}(G_V, \phi_{\circ}N')$,  $|EM^{V'}(G_{V'}, N')|\leq {\beta}^{\delta'}$ and $|\mathcal C(G'',N')|\leq \beta^{\delta'}$. 
	\end{lem}
	\begin{proof}
		Let $f\in X\subseteq EM^{V}(G_V, \phi_{\circ}M')$. Since $G_{V'}$ is a generator in $EM^{V'},$ there exists an epimorphism $g:(G_{V'})^{(J)}\longrightarrow M'$ in $EM^{V'}$ for some indexing set $J$. As $\phi_{\circ}$  preserves colimits, the induced morphism $\phi_{\circ}(g):\phi_{\circ}((G_{V'})^{(J)})=(\phi_{\circ}G_{V'})^{(J)}\longrightarrow \phi_{\circ}M'$ is an epimorphism in $EM^{V}.$  Since $G_{V}$ is  projective in $EM^V$, $f:G_{V}\longrightarrow \phi_{\circ}M'$ factors through a morphism $\eta:G_{V}\longrightarrow (\phi_{\circ}G_{V'})^{(J)}$ in $EM^{V}.$ Further, as  $\delta'\geq \kappa(G_V)$, $G_{V}$ is $\delta'$-presentable. Therefore, there exists a subset $J_f\subseteq J$ with $|J_f|<\delta'$ such that the following diagram commutes in $EM^V$:
		\begin{equation}\label{eq9.2}
			\begin{tikzcd}
				(\phi_{\circ}G_{V'})^{(J_f)} \arrow{dd}{\iota} & & G_{V}\arrow{dd}{f}\arrow{ll}{}\arrow{ddll}{\eta} \\
				& &\\
				(\phi_{\circ}G_{V'})^{(J)} \arrow{rr}{\phi_{\circ}(g)} & & \phi_{\circ}M' ~. \\
			\end{tikzcd}
		\end{equation}
		From \eqref{eq9.2}, we have a morphism $\eta_{f}':(G_{V'})^{(J_f)}\hookrightarrow (G_{V'})^{(J)}\overset{g}{\longrightarrow} M'$ in 
		$EM^{V'}$ such that   $f$ factors through $\phi_{\circ}(\eta_{f}').$  We now define a subobject $N'\subseteq M'$ as 
		\begin{equation}\label{eq9.3}
			N':=Im\left(\eta':= \underset{f\in X}{\bigoplus} \eta_{f}':\underset{f\in X}{\bigoplus} (G_{V'})^{(J_f)}\longrightarrow M'\right) ~.
		\end{equation} Since $G''$ is projective in $\mathcal C$, we use Lemma \ref{L2.1} to obtain
		\begin{equation}\label{7166t}
		|\mathcal C(G'',N')|\leq |\mathcal C(G'',\underset{f\in X}{\bigoplus} (G_{V'})^{(J_f)})|\leq ( \beta^{\delta'})^{\kappa(G'')}= \beta^{\delta'}
		\end{equation} 
		Since $\phi_{\circ}$ is exact and preserves colimits, we get 
		\begin{equation}\label{eq9.4}
			\phi_{\circ}N':= Im\left(\phi_{\circ}(\eta'):= \underset{f\in X}{\bigoplus} \phi_{\circ}(\eta_{f}'):\underset{f\in X}{\bigoplus} \phi_{\circ}((G_{V'})^{(J_f)})\longrightarrow \phi_{\circ}M' \right)~.
		\end{equation}
		From $(\ref{eq9.2})$ and $(\ref{eq9.4})$, it is clear that $X\subseteq  EM^{V}(G_V, \phi_{\circ}N')$. As in the proof of Lemma \ref{L7.6}, we note that
		the structure map $N'\longrightarrow V'N'$ is a monomorphism in $EM^{V'}$.  Using the epimorphism $G''\longrightarrow G_{V'}$, we now have  
		\begin{equation}\label{716t}
		|EM^{V'}(G_{V'}, N')| \leq |EM^{V'}(G_{V'}, V'N')| = |\mathcal C(G_{V'},N')|\leq |\mathcal C(G'',N')|\leq \beta^{\delta'}
		\end{equation} where the last inequality follows from \eqref{7166t}. 
		
	\end{proof}
	
	For the next result, we maintain the notation and setup from Lemma \ref{L9.3}.

	\begin{lem}\label{L9.4}
		For $(M',f^{M'})\in EM^{V'}$, let $X\subseteq EM^{V'}(G_{V'}, M')$, $Y\subseteq EM^{V}(G_{V}, \phi_{\circ}M')$ be such that $|X|\leq \beta^{\delta'}$ and $|Y|\leq \beta^{\delta'}.$ Then, there exists a subobject $N'\subseteq M'$ in $EM^{V'}$ such that 
		
		\smallskip
		
		(1) $X\subseteq EM^{V'}(G_{V'}, N')$ and $Y\subseteq EM^{V}(G_{V}, \phi_{\circ}N'),$
		
		\smallskip
		
		(2) $|EM^{V'}(G_{V'}, N')|\leq \beta^{\delta'}$ and $|EM^{V}(G_{V}, \phi_{\circ}N')|\leq \beta^{\delta'}.$
		
	\end{lem}
	\begin{proof}
		From Lemma \ref{L9.3}, there exists a subobject $N'_1 \subseteq M'$ such that $|EM^{V'}(G_{V'}, N'_{1})|\leq (\beta^{\delta'})^{\delta'}=\beta^{\delta'}$, 
		$|\mathcal C(G'',N_1)|\leq \beta^{\delta'}$ and $Y\subseteq EM^{V}(G_{V}, \phi_{\circ}N'_{1}).$ Now, taking $\phi=id$ and using Lemma \ref{L9.3} again we obtain a subobject $N'_2\subseteq M'$ such that $|EM^{V'}(G_{V'}, N'_{2})|\leq \beta^{\delta'}$, $|\mathcal C(G'',N_2)|\leq \beta^{\delta'}$ and $X\subseteq EM^{V'}(G_{V'}, N'_{2}).$ 
		Now, set $N'=N'_1+N'_2.$ Then, it is clear that $X\subseteq EM^{V'}(G_{V'}, N'_{2})\subseteq EM^{V'}(G_{V'}, N'),$ and as $\phi_{\circ}$ is exact, $Y\subseteq EM^{V}(G_{V}, \phi_{\circ}N'_{1})\subseteq EM^{V}(G_{V}, \phi_{\circ}N')$. By definition, we have an epimorphism $N'_1\oplus N'_2\longrightarrow N'$  in 
		$EM^{V'}$. Since $G_{V'}$ is projective in $EM^{V'}$, we now have 
			\begin{equation}\label{eq9.7}
				|EM^{V'}(G_{V'}, N')|\leq |EM^{V'}(G_{V'}, N'_1\oplus N'_2)| \leq  \beta^{\delta'}
			\end{equation}
			From (\ref{eq7.1}), we know that $\phi_\circ(N')\subseteq VN'$. Since $G_V$ is projective and we have an epimorphism $G''\longrightarrow G_V$, we now obtain
			\begin{equation}
				|EM^{V}(G_{V}, \phi_{\circ}N')|\leq |EM^{V}(G_{V}, VN')|\leq |EM^{V}(G_{V}, V(N_1'\oplus N_2')|= |\mathcal C(G_V, N'_1\oplus N'_2)|
				\leq |\mathcal C(G'',N_1'\oplus N'_2)|\leq \beta^{\delta'}
			\end{equation}
		\end{proof}

		We now study generators for $Com^{cs}_{c}-\mathscr V$. We note that colimits and finite limits in $Com^{cs}_{c}-\mathscr V$ are computed pointwise in 
		$Com^{cs}-\mathscr V$.  As in Section 7.1, for each $x\in \mathbb V$, we fix a projective generator $G_x$ for $EM^{\mathscr V_x}$. Since $G$ is a generator for $\mathcal C$, we can choose an indexing set $I$ that is large enough such that we have epimorphisms $G^{(I)}\longrightarrow G_x$  for each $x\in\mathbb V$. We set $G':=G^{(I)}$.  
		We now fix an infinite cardinal $\delta$ such that
		\begin{equation}\label{eq9.8}
			\delta\geq sup\{|Mor(\mathbb Q)|, \kappa(G'),  \kappa(G_x), ||G_x||^{G'}, \tau(G_x), x\in\mathbb V\}\
		\end{equation} where $G_x$ is $\tau(G_x)$-representable in the Grothendieck category $EM^{\mathscr V_x}$. We also choose $\beta\geq \delta$.

\smallskip
	
We now consider an object $\mathscr M\in Com^{cs}_{c}-\mathscr V$ and an element $\eta\in el_{\mathscr V}(\mathscr M)$ as in \eqref{eq7.4}. By definition, there is some $x\in\mathbb V$ such that $\eta:G_{x}\longrightarrow \mathscr M_x$ is a morphism in $EM^{\mathscr V_x}.$ By Lemma \ref{L7.5} and Lemma \ref{L7.6}, we can choose a subobject $\mathscr N\subseteq \mathscr M$ such that $\eta\in el_{\mathscr V}(\mathscr N)$ and  $|el_{\mathscr V}(\mathscr N)|\leq \delta^{\kappa(G')}\leq \beta^\delta$. Since $Mor(\mathbb Q)$ is well-ordered, we can take a lexicographic ordering on $\mathbb N\times Mor(\mathbb Q).$ As in previous sections, we construct a family of subobjects $\{\mathscr P(n,\phi):(n,\phi)\in \mathbb N \times Mor(\mathbb Q)\}$ of $\mathscr M$ in $Com^{cs}-\mathscr V$ that satisfies the following conditions
		
		\smallskip
		
		(1)\label{C6.1.1} If $\phi_0$ is the least element of $Mor(\mathbb Q)$, then $\eta\in el_{\mathscr V}(\mathscr P(1,\phi_0))$.
		
		\smallskip
		
		(2) If $(n,\phi)\leq (m,\psi)$ in $\mathbb N \times Mor(\mathbb Q)$, then $\mathscr P(n,\phi) \subseteq \mathscr P(m,\psi).$
		
		\smallskip
		
		(3) For given $(n,\phi:y \longrightarrow z)$ in $\mathbb N \times Mor(\mathbb Q)$, the morphism $\mathscr P(n,\phi)_{\phi}:\mathscr P(n,\phi)_y \longrightarrow \mathscr \phi_{\circ}\mathscr P(n,\phi)_z$ is an isomorphism in $EM^{\mathscr V_y}$.
		
		\smallskip
		
		(4) \label{C6.1.4}$|el_{\mathscr V}(\mathscr P(n,\phi))|\leq \beta^\delta$.
		
		\smallskip
		We fix $(n,\phi:y\longrightarrow z)\in \mathbb N\times Mor(\mathbb Q).$  As in previous sections, we begin the transfinite induction by setting  for each vertex $w\in\mathbb V,$ a subset $X_0^0(w)\subseteq EM^{\mathscr V_w}(G_w, \mathscr M_w)$ as
		\begin{equation*}
			X_0^0(w):= 
			\begin{cases} 
				EM^{\mathscr V_w}(G_w, \mathscr N_w),\quad & \text{if}~~(n,\phi) = (1,\phi_0) \\
				\underset{(m,\psi)< (n,\phi)}\bigcup EM^{\mathscr V_w}(G_w, \mathscr P(m,\psi)_w), \quad &~~ \mbox{otherwise}~.
			\end{cases}
		\end{equation*} We note that each $X_0^0(w)\subseteq EM^{\mathscr V_w}(G_w,\mathscr M_w)$ and $|X_0^0(w)|\leq 
\beta^\delta$. 
		Since $\mathscr M$ is cartesian, $X_0^0(z)\subseteq EM^{\mathscr V_z}(G_z, \mathscr M_z)$ and $X_0^0(y)\subseteq EM^{\mathscr V_y}(G_y, \mathscr M_y)= EM^{\mathscr V_y}(G_y, \phi_{\circ}\mathscr M_z)$. Using Lemma \ref{L9.4} there exists a subobject $X_{1}^0(z)\subseteq \mathscr M_z$ 
in $EM^{\mathscr V_z}$ such that 
		\begin{equation*}
			|EM^{\mathscr V_z}(G_z, X_1^0(z))|\leq  \beta^\delta\quad |EM^{\mathscr V_y}(G_y, \phi_{\circ}(X_1^0(z)))|\leq  \beta^\delta\quad X^0_0(z) \subseteq EM^{\mathscr V_z}(G_z, X_1^0(z))\quad X_0^0(y) \subseteq EM^{\mathscr V_y}(G_y, \phi_{\circ}(X_1^0(z))).
		\end{equation*} 
		We set $X_1^0(y) := \phi_{\circ}X_1^0(z)$.  For each $w\in \mathbb V$, we define
		\begin{equation*}
			Y_1^0(w):=\left\{\begin{array}{ll} EM^{\mathscr V_w}(G_w, (X_1^0(w)))  & \mbox{if $w=y,z$} \\
				X^0_0(w) & \mbox{otherwise}.\\
			\end{array}\right.
		\end{equation*}
		Then $X_{0}^{0}(w)\subseteq Y_{1}^{0}(w)$ and $|Y_{1}^{0}(w)|\leq \beta^\delta.$
		
		\smallskip
		\begin{lem}\label{L9.5}
			
			Let $X\subseteq el_{\mathscr V}(\mathscr M)$ such that $|X|\leq  \beta^\delta.$ Then there exists a subobject $\mathscr Q\subseteq \mathscr M$ in $Com^{cs}-\mathscr V$ such that $X\subseteq el_{\mathscr V}(\mathscr Q)$ and $|el_{\mathscr V}(\mathscr Q)|\leq  \beta^\delta.$
		\end{lem}
		\begin{proof}
			Let $\eta\in X\subseteq el_{\mathscr V}(\mathscr M).$ Then, by Lemma \ref{L7.5} and Lemma \ref{L7.6}, there exists a subobject $\mathscr Q_{\eta}\subseteq \mathscr M$ with $\eta\in el_{\mathscr V}(\mathscr Q_{\eta})$ and $|el_{\mathscr V}(\mathscr Q_{\eta})|=\underset{x\in\mathbb V}\sum|EM^{\mathscr V_x}(G_x,\mathscr Q_{\eta_x})|\leq  \beta^\delta.$ We set $\mathscr Q:=\underset{\eta\in X}\sum\mathscr Q_\eta$. Since $\mathscr Q$ is a quotient of $\underset{\eta\in X}\bigoplus\mathscr Q_\eta$, it follows by applying Lemma \ref{L2.1} to the Grothendieck category $EM^{\mathscr V_x}$ that 
			\begin{equation*}
				|EM^{\mathscr V_x}(G_x,\mathscr Q_x)|\leq |EM^{\mathscr V_x}(G_x,\underset{\eta\in X}\bigoplus\mathscr Q_{\eta_x})|\leq  ( \beta^\delta)^{\tau(G_{x})}\leq  \beta^\delta
			\end{equation*}
			From  \eqref{eq7.4}, it now follows that $|el_{\mathscr V}(\mathscr Q)|\leq \beta^\delta.$
		\end{proof}
		
		Since $\underset{w\in\mathbb V}\bigcup Y_1^0(w)\subseteq el_{\mathscr V}(\mathscr M)$ and $|\underset{w\in\mathbb V}\bigcup Y_1^0(w)|\leq  \beta^\delta,$ by Lemma \ref{L9.5}, there exists a subobject $\mathscr Q^{0}(n,\phi)\subseteq \mathscr M$ such that $\underset{w\in\mathbb V}\bigcup Y_1^0(w)\subseteq el_{\mathscr V}(\mathscr Q^{0}(n,\phi))$, i.e., for each $w\in\mathbb V$, $Y_1^0(w)\subseteq EM^{\mathscr V_w}(G_w,\mathscr Q^0(n,\phi)_w)$), and $|el_{\mathscr V}(\mathscr Q^{0}(n,\phi))|\leq  \beta^\delta.$ 
		
		\smallskip
		
		Now, we assume by induction that for every $s\leq t,$ we have obtained a subobject $\mathscr Q^{s}(n,\phi)\subseteq \mathscr M$ in $Com^{cs}-\mathscr V$ that satisfies $\underset{w\in\mathbb V}\bigcup Y_1^s(w)\subseteq el_{\mathscr V}(\mathscr Q^{s}(n,\phi))$ and $|el_{\mathscr V}(\mathscr Q^{s}(n,\phi))|\leq \beta^\delta.$ Then, for each $w\in\mathbb V,$   we set $$X_{0}^{t+1}(W):=EM^{\mathscr V_w}(G_w,\mathscr Q^{t}(n,\phi)_w).$$ 
		By Lemma \ref{L9.4} , there exists a subobject $X_{1}^{t+1}(z)\subseteq\mathscr M_z$ in $EM^{\mathscr V_z}$ such that 
		\begin{equation*}
			|EM^{\mathscr V_z}(G_z, (X_1^{t+1}(z)))|\leq \beta^\delta\quad |EM^{\mathscr V_y}(G_y, \phi_{\circ}(X_1^{t+1}(z)))|\leq \beta^\delta\quad  X^{t+1}_0(z) \subseteq EM^{\mathscr V_z}(G_z, (X_1^{t+1}(z)))\quad X_0^{t+1}(y) \subseteq EM^{\mathscr V_y}(G_y, \phi_{\circ}(X_1^{t+1}(z))).
		\end{equation*}
		Then, we set $X_{1}^{t+1}(y):=\phi_{\circ}X_{1}^{t+1}(z)$ and for each $w\in\mathbb V$
		\begin{equation*} 
			Y_1^{t+1}(w):=\left\{\begin{array}{ll} EM^{\mathscr V_w}(G_w, (X_1^{t+1}(w)))  & \mbox{if $w=y,z$} \\
				X_0^{t+1}(w) & \mbox{otherwise}\\
			\end{array}\right.
		\end{equation*}
		Then, each $X_{0}^{t+1}(w)\subseteq Y_{1}^{t+1}(w)$ and $|Y_{1}^{t+1}(w)|\leq \beta^\delta.$ Using Lemma \ref{L9.5}, we get a subobject $\mathscr Q^{t+1}(n,\phi)\subseteq \mathscr M$  in $Com^{cs}-\mathscr V$ such that 
		$\underset{w\in\mathbb V}\bigcup Y_1^{t+1}(w)\subseteq el_{\mathscr V}(\mathscr Q^{t+1}(n,\phi)),$ and $|el_{\mathscr V}(\mathscr Q^{t+1}(n,\phi))|\leq  \beta^\delta.$ In particular, $Y_1^{t+1}(w)\subseteq EM^{\mathscr V_w}(G_w,\mathscr Q^{t+1}(n,\phi)_w)$ for each $w\in\mathbb V$. 
		We have now constructed subobjects of  $\mathscr M$ in $Com^{cs}-\mathscr V$ such that 
		$
			\mathscr Q^{0}(n,\phi)\hookrightarrow\mathscr Q^{1}(n,\phi)\hookrightarrow\ldots\hookrightarrow\mathscr Q^{t}(n,\phi)\hookrightarrow\ldots~$ and 
		we set 
		\begin{equation}\label{eq9.9}
			\mathscr P(n,\phi):= \underset{t\geq 0}\bigcup \mathscr Q^{t}(n,\phi).
		\end{equation} As in Section 6.1, we see that the  family $\{\mathscr P(n,\phi)~|~(n,\phi)\in\mathbb N\times Mor(\mathbb Q)\}$  of subobjects of $\mathscr M$ in $Com^{cs}-\mathscr V$ satisfies the conditions (1)-(4).
	
		\begin{lem}\label{L9.7}
			Let $\mathscr M\in Com^{cs}_{c}-\mathscr V$ and $\eta\in el_{\mathscr V}(\mathscr M).$ Then there exists a subobject $\mathscr P\subseteq \mathscr M$ in $Com^{cs}_{c}-\mathscr V$ with $\eta\in el_{\mathscr V}(\mathscr P)$ such that $|el_{\mathscr V}(\mathscr P)|\leq \beta^\delta$.  
		\end{lem}
		\begin{proof} This is proved in a manner similar to Lemma \ref{L6.4}.
			
		\end{proof}
		\begin{Thm}\label{T5.9} $Com^{cs}_{c}-\mathscr V$ is a Grothendieck category.  
		\end{Thm}
		\begin{proof}
			It is already clear that $Com^{cs}_{c}-\mathscr V$ is a cocomplete abelian category where filtered colimits are exact. Now consider an element $\eta\in el_{\mathscr V}(\mathscr M).$  By Lemma \ref{L9.7}, there exists a subobject $\mathscr P_{\eta}\subseteq \mathscr M$ in $Com^{cs}_{c}-\mathscr V$ such that $\eta\in el_{\mathscr V}(\mathscr P_{\eta})$ and $|el_{\mathscr V}(\mathscr P_{\eta})|\leq\beta^\delta.$ Then, the isomorphism classes of cartesian comodules $\mathscr P$ satisfying $|el_{\mathscr V}(\mathscr P)|\leq\beta^\delta$ form a set of generators for $Com^{cs}_{c}-\mathscr V$. Hence, $Com^{cs}_{c}-\mathscr V$ is a Grothendieck category.
		\end{proof}
		\begin{Thm}\label{T5.10}
			The inclusion functor $\iota: Com^{cs}_{c}-\mathscr V \longrightarrow Com^{cs}-\mathscr V$ has a right adjoint.
		\end{Thm}
		\begin{proof}
			Since the inclusion functor $\iota: Com^{cs}_{c}-\mathscr V \longrightarrow Com^{cs}-\mathscr V$ preserves colimits and $Com^{cs}_{c}-\mathscr V$ is a Grothendieck category, it  follows from \cite[Proposition 8.3.27]{KS} that $\iota$ has a right adjoint. 
		\end{proof}
		
		\section{Rational pairing of a monad and a comonad quiver}
			We continue with $\mathcal C$ being a Grothendieck category with a projective generator $G$. Let $U=(U, \theta, \eta)$ be a monad and $V=(V,\delta, \epsilon)$ be a comonad on $\mathcal{C}$. We recall the following notion. 
		\begin{defn} (see \cite[\S 3.2]{MW})  A pairing $\mathcal P=(U,V,\varrho)$ between $U$ and $V$ is a natural transformation $\varrho:UV\longrightarrow I$ such that for each $M,N\in \mathcal C$, we have the following commutative diagram 
		\[
		\begin{tikzcd}
			& \mathcal{C}(M, N) \\
			\mathcal{C}(M, VN) \arrow{ur}{\mathcal{C}(M, \epsilon_{N})} \arrow{rr}{\beta_{M, N}^{\mathcal{P}}} \arrow{d}{\mathcal C(M, \delta_{N})} & & \mathcal{C}(UM, N) \arrow{drr}{\mathcal C(\theta_{M}, N)} \arrow{ul}[swap, pos=0.5]{\mathcal{C}(\eta_M, N)}\\
			\mathcal{C}(M, V^2N) \arrow{rr}{\beta_{M, VN}^{\mathcal{P}}} & & \mathcal{C}(UM, VN) \arrow{rr}{\beta_{UM, N}^{\mathcal{P}}} & & \mathcal{C}(U^2M, N)
		\end{tikzcd}
		\] 
			where $\beta_{M,N}^{\mathcal P}:\mathcal{C}(M,VN)\longrightarrow  \mathcal C(UM,N),~ f\mapsto \varrho_{N}\circ Uf.$ The pairing is said to be rational if $\beta_{M,N}^{\mathcal P}$ is injective for each $M,N\in \mathcal C.$
		\end{defn}
		
		\smallskip
		
		Given a pairing $\mathcal P=(U, V,\varrho)$, we note that any object $(N,f^N)\in EM^V$ may be treated as an object $(N, \varrho_{N}\circ U(f^N))\in EM_{U}.$ This gives a functor $\Phi_{\mathcal P}: EM^V\longrightarrow EM_U$.  We suppose that the pairing $\mathcal P=(U, V,\varrho)$ is rational and the monad $U$ has a right adjoint, say $V'$. In that case, the  functor $\Phi_{\mathcal P}:EM^V\longrightarrow EM_{U}$ is full and faithful (see \cite[\S 3.7]{MW}). Then, there is an induced morphism $t:V\longrightarrow V'$ of comonads such that $t_{M'}:VM'\longrightarrow V'M'$ is a monomorphism 
for each $M'\in \mathcal C$ (see \cite[$\S$ 3.16]{MW}). For any $(M, f^M)\in EM^{V'}$, we consider the following pullback in $EM^{V'}$ (see, \cite[\S 4.1]{MW})
		\begin{equation}
			\begin{CD}
				\Upsilon(M, f^M) @> {p_2}>>  VM\\
				@V p_1 VV @VV t_{M} V \\
				M @> f^{M}  >>  V'M ~.\\
			\end{CD}
		\end{equation}
		This induces a functor $\Upsilon:EM^{V'}\longrightarrow EM^{V'}$, $(M,f^{M})\mapsto \Upsilon(M, f^M)$. Further, if the comonad $V$ preserves equalizers, then we know from \cite[\S 4.6]{MW} that there is a functor 
		\begin{equation}
			Rat_{\mathcal P}:EM_{U}\xrightarrow{\cong} EM^{V'}\xrightarrow{\Upsilon} EM^{V'} \xrightarrow{\cong} EM_{U},
		\end{equation}
		such that  $Rat_{\mathcal P}(Rat_{\mathcal P}(M, f_M)) = Rat_{\mathcal P}(M, f_M)$ for any object $(M, f_M)\in EM_{U}$. A module $(M,f_{M})$ in $EM_U$ is said to be rational if $Rat_{\mathcal P}(M,f_M) = (M,f_M).$ Let $Rat_{\mathcal P}-U$ denote the full subcategory of rational modules over $U$. Then, an object $(M,f_{M})\in EM_U$ (also treated as $(M,f^M)\in EM^{V'}=EM_U$) lies in  $Rat_{\mathcal P}-U$ if and only if there exists a unique morphism $\mu:M\longrightarrow VM$ such that the following diagram commutes (see \cite[Proposition 4.7]{MW}).
		\begin{equation}
			\begin{tikzcd}
				&    &  {VM}  \arrow{ddr}{t_M} &\\
				&   &   \\[-2ex]
				&   M\arrow{rr}[swap]{f^M} \arrow{uur}{\mu}  & & V'M &\\
			\end{tikzcd}
		\end{equation} 
		Then $Rat_{\mathcal P}-U$ is a coreflective subcategory of $EM_U$, i.e. the inclusion functor $\iota_{\mathcal P}:Rat_{\mathcal P}-U\longrightarrow EM_{U}$ has a right adjoint that we continue to denote by $Rat_{\mathcal P}:EM_U\longrightarrow Rat_{\mathcal P}-U$ and the functor $\Phi_{\mathcal P}:EM^V\longrightarrow EM_{U}$ corestricts to an equivalence of categories $R_{\mathcal P}: EM^V \longrightarrow Rat_{\mathcal P}-U$ (see \cite[Theorem 4.8]{MW}). 
		\begin{thm}\label{P8.2xw}
			Let $U$ be a monad on $\mathcal C$ that is exact and preserves colimits and let $V$ be a comonad on $\mathcal C$ that is exact. Let $V'$ be the right adjoint of the monad $U$. Then, for a rational pairing $\mathcal P=(U, V,\varrho)$ on $\mathcal C$, the full subcategory $Rat_{\mathcal P}-U$ of rational modules over $U$ is closed under coproducts, subobjects and quotients in $EM_U.$
		\end{thm}
		\begin{proof}
			Since $Rat_{\mathcal P}-U$ is a coreflective subcategory of $EM_U$, and $EM_U$ is a Grothendieck category, it  follows from (\cite[Corollary 2]{HS}) that $Rat_{\mathcal P}-U$ is closed under coproducts. 
			
			\smallskip
			We now consider an object $(M,f_M)\in Rat_{\mathcal P}-U$. Then, we know from (\cite[\S 4.6]{MW}) that for any $U$-submodule $(N, f_{N})\xrightarrow{\iota} (M,f_M)$ in $EM_U$, the following diagram is a pullback
				\begin{equation}
				\begin{CD}
				Rat_{\mathcal P}(N, f_N) @> Rat_{\mathcal P}(\iota)>>  Rat_{\mathcal P
			}(M,f_M)\\
					@V \cong VV @VV \cong V \\
					N @> \iota  >>  M~ \\
				\end{CD}
			\end{equation}
			 Therefore, $N$ is rational module over $U$. Accordingly, for any $(M,f_M)\in Rat_{\mathcal P}-U$ and any subobject $(P,f_{P})\subseteq (M,f_M)$ in $EM_U$, we know that $(P,f_{P})\in Rat_{\mathcal P}-U$. Since the coreflective subcategory $Rat_{\mathcal P}-U$  must be closed
under colimits (see \cite[Theorem 1]{HS}), it is clear that the quotient $(M,f_M)/(P,f_P) \in Rat_{\mathcal P}-U$. 
		\end{proof}
		
		\smallskip
		
		We now define a rational pairing between a monad quiver $\mathscr U:\mathbb Q^{op}\longrightarrow Mnd(\mathcal{C})$ and a comonad quiver $\mathscr V:\mathbb Q\longrightarrow Cmd(\mathcal{C})$. We assume that for each $x\in\mathbb V,$ the monad $\mathscr U_x$ is exact and preserves colimits and the comonad $\mathscr V_x$ is exact. 
		\begin{defn}\label{D10.2}
			Let $\mathbb Q=(\mathbb V,\mathbb E)$ be a quiver. Let $\mathscr U:\mathbb Q^{op}\longrightarrow Mnd(\mathcal{C})$ be a monad quiver and let $\mathscr V:\mathbb Q\longrightarrow Cmd(\mathcal{C})$ be a   comonad quiver. For each vertex $x\in \mathbb V$, let $\mathscr V_x'$ denote the right adjoint of $\mathscr U_x$. A rational pairing $\mathcal P=(\mathscr U, \mathscr V, \gamma=\{\gamma_x\}_{x\in \mathbb V})$ between $\mathscr U$ and $\mathscr V$ is a triple such that 
		
		\smallskip
	(1)	for each $x\in \mathbb V$,  $\mathcal P_x:=(\mathscr U_x, \mathscr V_x, \gamma_x)$  is a rational pairing between the monad $\mathscr U_x$ and the comonad $\mathscr V_x$
				
				\smallskip
				(2) for each edge $\psi\in \mathbb Q(x,y)=\mathbb Q^{op}(y,x)$  the following diagram commutes
				\begin{equation}\label{eq10.2}
\begin{CD}
\mathscr V_x@>\mathscr V(\psi)>>\mathscr V_y\\
@VVV @VVV\\
\mathscr V'_x@>\mathscr V'(\psi)>>\mathscr V'_y\\
\end{CD}
				\end{equation}
			where the vertical morphisms $\mathscr V_x\longrightarrow \mathscr V_x'$ and $
\mathscr V_y\longrightarrow \mathscr V'_y$  in \eqref{eq10.2} are natural transformations of functors induced by the pairings $\mathcal P_x$ and $\mathcal P_y$ respectively. 
		\end{defn}
		
		We can now define  rational cis-modules (resp. rational trans-modules) in the setup of Section 4.1 (resp. Section 4.2). 
		\begin{defn}\label{D10.3}
			Let $\mathcal P=(\mathscr U, \mathscr V, \gamma)$ be a rational pairing between a monad quiver $\mathscr U:\mathbb Q^{op}\longrightarrow Mnd(\mathcal{C})$ and a comonad quiver $\mathscr V:\mathbb Q\longrightarrow Cmd(\mathcal{C})$. We say that a module $\mathscr M\in Mod^{cs}-\mathscr U$ (or $Mod^{tr}-\mathscr U$) is rational if $Rat_{\mathcal P_x}(\mathscr M_x)=\mathscr M_x$ for each $x\in\mathbb V$. We let $Rat_{\mathcal P}^{cs}-\mathscr U$(resp. $Rat_{\mathcal P}^{tr}-\mathscr U$) denote the full subcategory of rational cis-modules (resp. rational trans-modules) over $\mathscr U.$ 
		\end{defn}
		 
			\begin{Thm}\label{T10.4}
			Let $\mathcal P=(\mathscr U, \mathscr V, \gamma)$ be a rational pairing between a monad quiver $\mathscr U:\mathbb Q^{op}\longrightarrow Mnd(\mathcal{C})$ and a comonad quiver $\mathscr V:\mathbb Q\longrightarrow Cmd(\mathcal{C})$. Then,
			
			\smallskip
			(1) the inclusion functor $\iota_{\mathcal P}^{tr} :Rat_{\mathcal P}^{tr}-\mathscr U \longrightarrow Mod^{tr}-\mathscr U$ has a right adjoint $Rat_{\mathcal P}^{tr}: Mod^{tr}-\mathscr U\longrightarrow Rat_{\mathcal P}^{tr}-\mathscr U$. 
			 
		\smallskip
		(2) $\mathscr M\in\ Mod^{tr}-\mathscr U$ is rational if and only if $Rat_{\mathcal P}^{tr}(\mathscr M)= \mathscr M.$
		\end{Thm}
		\begin{proof}
			
			\smallskip
		(1)	For any object $\mathscr M\in  Mod^{tr}-\mathscr U$ and $y\in\mathbb V$, we set
			\begin{equation}\label{eq10.3}
				Rat_{\mathcal{P}}^{tr}(\mathscr M)_y:= \underset{N'\in \mathfrak R_{\mathcal{P}}^{tr}(\mathscr M_y)} \sum N',
			\end{equation}
			where $ \mathfrak R_{\mathcal{P}}^{tr}(\mathscr M_y):=\{N'\overset{i}{\hookrightarrow} \mathscr M_y~|~ Im({^\psi \mathscr M}\circ 
			\psi_*(i))  \subseteq Rat_{\mathcal P_x}(\mathscr M_x)~\textup{for all}~\psi\in\mathbb Q^{op}(x,y)\}$. We note in particular that $Rat_{\mathcal{P}}^{tr}(\mathscr M)_y\subseteq Rat_{\mathcal{P}_y}(\mathscr M_y).$  Hence, $Rat_{\mathcal{P}}^{tr}(\mathscr M)_y \in\ Rat_{\mathcal P_y}-\mathscr U_y$.  For any edge $\phi\in \mathbb Q^{op}(y,z)$ and any subobject $N''\overset{i'}{\hookrightarrow}\mathscr M_z$ in $\mathfrak R_{\mathcal P}^{tr}(\mathscr M_z)$, we must have $Im(^{\phi\psi}\mathscr M\circ (\phi\psi)_{\ast}(i')) \in Rat_{\mathcal P_x}(\mathscr M_x)$. Since $^{\phi\psi}\mathscr M = {^\psi\mathscr M}\circ \psi_{\ast}(^{\phi}\mathscr M)$, we get $Im(^\psi \mathscr M\circ \psi_{\ast}(^{\phi}\mathscr M\circ \phi_{\ast}(i')))\in Rat_{\mathcal P_x}(\mathscr M_x)$. It follows that $Im(^\phi\mathscr M\circ \phi_{\ast}(i'))\in \mathfrak R^{tr}_{\mathcal P}(\mathscr M_y)$. As the restriction functor $\phi_{\ast}$ is exact and preserves colimits, it follows that the morphism $^\phi\mathscr M: \phi_{\ast}\mathscr M_z\longrightarrow \mathscr M_y$ restricts to a morphism $^\phi Rat_{\mathcal{P}}^{tr}(\mathscr M): \phi_{\ast}Rat_{\mathcal{P}}^{tr}(\mathscr M)_z\longrightarrow  Rat_{\mathcal{P}}^{tr}(\mathscr M)_y$. 
Therefore, the subobjects $Rat_{\mathcal{P}}^{tr}(\mathscr M)_y$ as in \eqref{eq10.3} together determine an object $Rat_{\mathcal P}^{tr}(\mathscr{M})\in Rat_{\mathcal P}^{tr}-\mathscr U.$

			\smallskip
			
			We now consider an object $\mathscr P\in Rat_{\mathcal P}^{tr}-\mathscr U$ and a morphism $\zeta:\iota_{\mathcal P}^{tr}(\mathscr P)\longrightarrow \mathscr M$ in $Mod^{tr}-\mathscr U$. To prove that the functor $Rat_{\mathcal P}^{tr}$ is right adjoint to  $\iota_{\mathcal P}^{tr}$, it suffices to show that $Im(\zeta_y)\subseteq Rat_{\mathcal P}^{tr}(\mathscr M)_y$ for each $y\in \mathbb V$.
			For each $\psi\in \mathbb Q^{op}(x,y),$ we consider the following commutative diagram in $EM_{\mathscr U_x}$
			\begin{equation}
				\begin{CD}
					\mathscr \psi_{*}\mathscr{P}_y @> {\psi_{*}(\zeta_y)}>>  \psi_{*}\mathscr M_y\\
					@V ^{\psi} \mathscr P VV @VV ^{\psi}\mathscr{M} V \\
					\mathscr{P}_x @>\zeta_x  >>  \mathscr M_x ~.\\
				\end{CD}
			\end{equation}
			Since, $\mathscr P\in Rat_{\mathcal P}^{tr}-\mathscr U$, $Rat_{\mathcal P_x}(\mathscr P_x) = \mathscr P_x$. Therefore, $Im(\zeta_{x})\subseteq Rat_{\mathcal P_x}(\mathscr M_x).$ Further, as $^\psi \mathscr M\circ \psi_{*}(\zeta_y) = \zeta_x\circ~ ^{\psi} \mathscr P,$ it follows that $Im(^\psi\mathscr M\circ \psi_{*}(\zeta_y))\subseteq Rat_{\mathcal P_x}(\mathscr M_x)$ for each edge $\psi\in \mathbb Q^{op}(x,y).$ Therefore, $Im(\zeta_y)\subseteq Rat_{\mathcal P}^{tr}(\mathscr M)_y.$ This proves the result.
			
			\smallskip
			(2) Consider an object $\mathscr M\in Rat_{\mathcal P}^{tr}-\mathscr U$. Then, $Rat_{\mathcal P_x}(\mathscr M_x)=\mathscr M_x$ for each $x\in\mathbb V$. Therefore, from the definition in (\ref{eq10.3}), we get $Rat_{\mathcal P}^{tr}(\mathscr M)_y= \underset{N'\subseteq \mathscr M_y} \sum \mathscr N' = \mathscr M_y$. Hence, $Rat_{\mathcal P}^{tr}(\mathscr M)= \mathscr M.$ Conversely, if we have $\mathscr M = Rat_{\mathcal P}^{tr}(\mathscr M)$  for some $\mathscr M\in Mod^{tr}-\mathscr U$, then $\mathscr M_x = Rat_{\mathcal P}^{tr}(\mathscr M)_x\subseteq Rat_{\mathcal P_x}(\mathscr M_x)$ for each $x\in\mathbb V$. Therefore, $\mathscr M$ is a rational module in $Mod^{tr}-\mathscr U$.
		\end{proof}

		\begin{Thm}\label{T10.6}
			Let $\mathcal P=(\mathscr U, \mathscr V, \gamma)$ be a rational pairing. Then, $Com^{cs}-\mathscr V\cong Rat^{tr}_{\mathcal P}-\mathscr U.$    
		\end{Thm}
		\begin{proof}
			Let $\mathscr M\in Rat_{\mathcal P}^{tr}-\mathscr U.$ Then, $\mathscr M_x= Rat_{\mathscr P_x}(\mathscr M_x)\in Rat_{\mathcal P_x}-\mathscr U_x$ for each $x\in\mathbb V$. Since $EM^{\mathscr V_x}\cong Rat_{\mathcal P_x}-\mathscr U_x$ for each $x\in\mathbb V$ (see \cite[Theorem 4.8]{MW}), each $\mathscr M_x$ may be treated as an object in $EM^{\mathscr V_x}.$ For each edge $\psi\in \mathbb Q(x,y)=\mathbb Q^{op}(y,x),$ we have the morphism $^\psi\mathscr M:\psi_{*}\mathscr M_x\longrightarrow \mathscr M_y$ in $EM_{\mathscr U_y}$, the morphism $\mathscr V(\psi):
\mathscr  V_x\longrightarrow \mathscr V_y$ of comonads, and the morphism $\mathscr U(\psi):\mathscr U_y\longrightarrow\mathscr U_x$ of monads. In order to prove that $\mathscr M\in Com^{cs}-\mathscr V$, we must show that the underlying morphism $\mathscr M_x\longrightarrow \mathscr M_y$ in $\mathcal C$ is a morphism of $\mathscr V_y$-comodules. In other words, we must show that the following diagram commutes
\begin{equation}\label{qqq1x}
\begin{tikzcd}
\mathscr M_x\arrow{r}\arrow{d}&\mathscr V_x\mathscr M_x \arrow{r}&\mathscr V_y\mathscr M_x\arrow{d} \\
\mathscr M_y \arrow{rr}&& \mathscr V_y\mathscr M_y
\end{tikzcd}
\end{equation}
As in Definition \ref{D10.2}, let $\mathscr V_x'$ denote the right adjoint of $\mathscr U_x$ and let $\mathscr V'_y$ denote the right adjoint of 
$\mathscr U_y$.  Using condition (2) in Definition \ref{D10.2}, we see that composing with $\mathscr V_y\mathscr M_y\longrightarrow \mathscr V_y'\mathscr M_y$ transforms \eqref{qqq1x} into the diagram 
\begin{equation}\label{qqq2x}
\begin{tikzcd}
\mathscr M_x\arrow{rr}\arrow{d}&&\mathscr M_y\arrow{d} \\
\mathscr V'_x\mathscr M_x \arrow{r}&\mathscr V'_y\mathscr M_x\arrow{r}& \mathscr V_y'\mathscr M_y
\end{tikzcd}
\end{equation} The latter diagram \eqref{qqq2x} commutes because $\mathscr M_x\longrightarrow \mathscr M_y$ in $\mathcal C$ is a morphism of 
$\mathscr U_y$-modules (and hence a morphism of $\mathscr V'_y$-comodules). Since $\mathcal P_y$ is a rational pairing, we note that 
$\mathscr V_y\mathscr M_y\longrightarrow \mathscr V_y'\mathscr M_y$  is a monomorphism. Hence, the diagram \eqref{qqq1x} commutes. 
			
			\smallskip
			
			Conversely, suppose that $\mathscr N\in Com^{cs}-\mathscr V.$ Then, for each $x\in\mathbb V$, $\mathscr N_x$ may be treated as a rational module in $EM_{\mathscr U_x}$. Further, for each edge $\psi\in \mathbb Q(x,y)=\mathbb Q^{op}(y,x),$ the morphism $\mathscr N^{\psi}: \psi^{\circ}\mathscr N_x\longrightarrow \mathscr N_y$ in $EM^{\mathscr V_y}$ may be treated as a map $\mathscr N_x\longrightarrow \mathscr N_y$  of 
$\mathscr U_y$-modules. Hence, $\mathscr N$ corresponds to an object in $Rat^{tr}_{\mathcal P}-\mathscr U.$ 
		\end{proof}
		
		\smallskip
		
		We now recall (see, for instance, \cite[$\S$ 1.1]{BR}) that for an abelian category $\mathcal D$, a torsion theory on $\mathcal D$ consists of a pair $(\mathcal T, \mathcal F)$ of
		strict and full subcategories $\mathcal T$ and $\mathcal F$ such that $\mathcal D(X, Y) = 0$ for any $X \in \mathcal T$, $Y \in \mathcal F$ and for any $Z \in \mathcal D,$ one has a short exact sequence
		\begin{equation*}
			0\longrightarrow Z^{\mathcal T}\longrightarrow Z\longrightarrow Z^{\mathcal F}\longrightarrow 0,
		\end{equation*}
		with $Z^{\mathcal T}\in \mathcal T$ and $Z^{\mathcal F}\in \mathcal F$.
		The torsion pair $(\mathcal T, \mathcal F)$ is said to be hereditary if the torsion class $\mathcal T$ is
		closed under subobjects. We shall now give conditions for the category $Rat_{\mathcal P}-U$ of rational modules over the monad $U$ to be a torsion class in $EM_U,$ generalizing the result (see \cite{Lin}) for rational pairings of coalgebras and algebras.
		\begin{lem}\label{L10.7} 
			Let $U$ be a monad on $\mathcal C$ that is exact and preserves colimits and let $V$ be a comonad on $\mathcal C$ that is exact. Then, for a rational pairing $\mathcal P=(U, V,\varrho)$ on $\mathcal C$, $Rat_{\mathcal P}-U$ is a torsion class in $EM_U$ if and only if for any object $M\in EM_U,$ $Rat_{\mathcal P}(M/Rat_{\mathcal P}(M))=0$. 
		\end{lem}
		\begin{proof}
			Let $Rat_{\mathcal P}-U$ be a torsion class in $EM_U$. Then, it is clear that $Rat_{\mathcal P}(M/Rat_{\mathcal P}(M)) = 0$. Conversely, suppose that for any object $M\in EM_U$, $Rat_{\mathcal P}(M/Rat_{\mathcal P}(M))=0$. We now consider two full subcategories $\mathcal T$ and $\mathcal F$ of $EM_U$ defined as
			\begin{equation*}
				ob(\mathcal T) := \{ M\in EM_U~|~Rat_{\mathcal P}(M)=M\}\quad \textup{and}\quad ob(\mathcal F) := \{ M\in EM_U~|~Rat_{\mathcal P}(M)=0\}.
			\end{equation*}
			Since $(\iota_{\mathcal P}, Rat_{\mathcal P})$ is a pair of adjoint functors, we get $EM_{U}(M,N)=Rat_{\mathcal P}-U(M, Rat_{\mathcal P}(N))=0$ for any $M\in \mathcal T$ and $N\in \mathcal F.$ Further, for any $M\in EM_U$, we consider the following short exact sequence in $EM_{U}$
			\begin{equation*}
				0\longrightarrow Rat_{\mathcal P}(M)\longrightarrow M\longrightarrow M/Rat_{\mathcal P}(M)\longrightarrow 0.
			\end{equation*}
			Since $Rat_{\mathcal P}(Rat_{\mathcal P}(M))= Rat_{\mathcal P}(M)$, $Rat_{\mathcal P}(M)\in \mathcal T$. Further, as $Rat_{\mathcal P}(M/Rat_{\mathcal P}(M))=0$, $Rat_{\mathcal P}(M/Rat_{\mathcal P}(M))\in \mathcal F$. Therefore, $Rat_{\mathcal P}-U$ is a torsion class in $EM_U$. 
		\end{proof}
		 
		\begin{Thm}\label{T10.8}
			Let $\mathcal P=( \mathscr U, \mathscr V, \gamma)$ be a rational pairing. Suppose that for each vertex $x\in\mathbb V$ and $N\in EM_{\mathscr U_x}$, we have $Rat_{\mathcal P_x}(N/Rat_{\mathcal P_x}(N))=0.$ Then, the category $Rat^{tr}_{\mathcal P}-\mathscr{U}$ is a hereditary torsion class in $Mod^{tr}-\mathscr U.$    
		\end{Thm}
		\begin{proof}
		Since $Mod^{tr}-\mathscr U$ is a Grothendieck category by Theorem \ref{T4.13}, it is both complete and cocomplete.	From \cite[\S 1.1]{BR}, it follows that any
			full subcategory of $Mod^{tr}-\mathscr U$ which is closed under quotients, extensions, and coproducts must be
			a torsion class.   For each $x\in\mathbb V$ and $N\in EM_{\mathscr U_x}$, we know that $Rat_{\mathcal P_x}(N/Rat_{\mathcal P_x}(N))=0$. It  follows from Lemma \ref{L10.7} that each $Rat_{\mathcal P_x}-\mathscr U_x$ is a torsion class in $EM_{\mathscr U_x}.$ Suppose we have a short exact sequence 
			$
				0\longrightarrow \mathscr N\longrightarrow \mathscr M\longrightarrow \mathscr P\longrightarrow 0
			$
			in $Mod^{tr}-\mathscr U$. Then, for each $x\in\mathbb V,$ the following sequence
			\begin{equation}\label{eq10101}
				0\longrightarrow \mathscr N_x\longrightarrow \mathscr M_x\longrightarrow \mathscr P_x\longrightarrow 0\notag
			\end{equation}
			is exact in $EM_{\mathscr U_x}$.  Since $Rat_{\mathcal P_x}-\mathscr U_x$ is a torsion class in $EM_{\mathscr U_x},$ it is now clear from 
\eqref{eq10101} that $Rat_{\mathcal P}-U$ is closed under quotients, coproducts and extensions. Hence, $Rat_{\mathcal P}-U$ must be a torsion class 
in $Mod^{tr}-\mathscr U$.  By Proposition \ref{P8.2xw}, each $Rat_{\mathcal P_x}-\mathscr U_x$ is closed under subobjects. It now follows that  $Rat_{\mathcal P}-U$  is also closed under subobjects, and hence is a hereditary torsion class in $Mod^{tr}-\mathscr U$.  
			 
		\end{proof}

	\end{document}